\documentclass[12pt]{article}
\usepackage{amsmath,amssymb,amscd}
\usepackage{ulem}
\allowdisplaybreaks

\def\bilin#1#2{\left\langle#1,\,#2\right\rangle}

 \newtheorem{theorem}{Theorem}[section]
 
 \newtheorem{corollary}[theorem]{Corollary}
 \newtheorem{lemma}[theorem]{Lemma}
 \newtheorem{definition}[theorem]{Definition}
 \newtheorem{example}[theorem]{Example}
 \newtheorem{remark}[theorem]{Remark}

\numberwithin{equation}{section}

\newenvironment{proof}{\smallskip\par{\sc Proof.}\enspace}%
 {{\unskip\nobreak\hfil\penalty50\hskip2em
          \hbox{}\nobreak\hfil{\rule[-1pt]{5pt}{10pt}}
          \parfillskip=0pt\finalhyphendemerits=0
          \par\medskip}} %This is newly changed %%
\begin{document}

%{\small\tt \hfill Printed \today}
\vspace*{.3in}

\begin{center}
\LARGE
{\sf A Generalization of Littlewood-Paley Type Inequality for Evolution Systems Associated with Pseudo Differential Operators}
\end{center}

\bigskip
\begin{center}
Un Cig Ji \\%\footnote{First author}\\
Department of Mathematics\\
Institute for Industrial and Applied Mathematics\\
Chungbuk National University\\
Cheongju 28644, Korea \\
\texttt{E-Mail:uncigji@chungbuk.ac.kr }
\end{center}

\begin{center}
Jae Hun Kim \\%\footnote{First author}\\
Department of Mathematics\\
Chungbuk National University\\
Cheongju 28644, Korea \\
\texttt{E-Mail:jaehunkim@chungbuk.ac.kr }
\end{center}

\begin{abstract}
In this paper, we first prove that the Littlewood-Paley $g$-function,
related to the convolution corresponding to the composition of pseudo-differential operator and evolution system
associated with pseudo-differential operators, is a bounded operator from
$L^{q}((a,b)\times \mathbb{R}^{d};V)$ with a Hilbert space $V$ into $L^{q}((a,b)\times \mathbb{R}^{d})$.
Secondly, we prove that the sharp function of the Littlewood-Paley $g$-function is bounded by some maximal function.
Finally, by applying Fefferman-Stein theorem and Hardy-Littlewood maximal theorem,
we prove the Littlewood-Paley type inequality for evolution systems associated with pseudo-differential operators.
\end{abstract}

%28D05 View Publications (1980-now) Measure-preserving transformations
%37A30 View Publications (2000-now) Ergodic theorems, spectral theory, Markov operators [For operator ergodic theory, see mainly 47A35]

\bigskip
\noindent

\tableofcontents
\section{Introduction}
The Littlewood-Paley inequality for the heat semigroup $\{T_{t}\}_{t\geq 0}$ defined by
\begin{align*}
T_{t}f(x)=p(t,\cdot)*f(x),\quad p(t,x)=\frac{1}{(4\pi t)^{d/2}}e^{-\frac{|x|^2}{4t}},
\quad t>0,\quad x\in\mathbb{R}^{d},
\end{align*}
is as follows: for any $p>1$ and $f\in L^{p}(\mathbb{R}^d)$,
\begin{equation}\label{eqn: classcal LP ineq}
\int_{\mathbb{R}^d}\left(\int_{0}^{\infty}|(-\Delta)^{\frac{1}{2}}T_{t}f(x)|^{2}dt\right)^{\frac{p}{2}}dx
\leq N_{p}\int_{\mathbb{R}^d}|f(x)|^p dx.
\end{equation}

Let $\psi:[0,\infty)\times \mathbb{R}^{d}\rightarrow \mathbb{C}$ be a measurable function.
For each $t\geq 0$, the pseudo-differential operator $L_{\psi}(t)$ with the symbol $\psi$
is defined by
\begin{align*}
L_{\psi}(t)f(x)=\mathcal{F}^{-1}(\psi(t,\xi)\mathcal{F}f(\xi))(x)
\end{align*}
for suitable functions $f$ defined on $\mathbb{R}^{d}$,
where $\mathcal{F}$ and $\mathcal{F}^{-1}$ are the Fourier transform
and Fourier inverse transform, respectively (see \eqref{eqn:FTand FIT}).
For each $t>s$, we define the operator $\mathcal{T}_{\psi}(t,s)$ by
\begin{equation*}%\label{eqn: solution operator}
\mathcal{T}_{\psi}(t,s)f(x)=p_{\psi}(t,s,\cdot)*f(x)
\end{equation*}
for suitable functions $f$, where
\begin{align*}
p_{\psi}(t,s,x):=\mathcal{F}^{-1}\left(\exp\left(\int_{s}^{t}\psi(r,\xi)dr\right)\right)(x).
\end{align*}
if the right hand side is well-defined, and so we have
\begin{align*}
\mathcal{T}_{\psi}(t,s)f(x)
=\mathcal{F}^{-1}\left(\exp\left(\int_{s}^{t}\psi(r,\xi)dr\right)\mathcal{F}f(\xi)\right)(x).
\end{align*}

Recently, in \cite{Ji-Kim 2025}, the authors proved the following inequality under certain conditions:
for any $0<a<\infty$,
there exists a constant $C>0$ such that
for any $s,l\geq 0$ and $f\in L^{p}(\mathbb{R}^{d})$,
\begin{align*}%\label{eqn:LP-ineq in intro}
\int_{\mathbb{R}^{d}}\left(\int_{s}^{s+a}(t-s)^{\frac{q\gamma_{1}}{\gamma_{2}}-1}
   |L_{\psi_{1}}(l)\mathcal{T}_{\psi_{2}}(t,s)f(x)|^{q}dt\right)^{\frac{p}{q}}dx
\leq C\int_{\mathbb{R}^{d}}|f(x)|^{p}dx,
\end{align*}
where $\gamma_{1},\gamma_{2}$ are corresponding to the orders of the symbols $\psi_{1}$ and $\psi_{2}$, respectively.
(see \textbf{(S1)} in Section \ref{sec: symbol}).
Moreover, the author proved that
if $\psi_{1}(l,\xi)=\psi_{1}(\xi)$ and $\psi_{2}(r,\xi)=\psi_{2}(\xi)$, that is,
$\psi_{1}(l,\xi)$ and $\psi_{2}(r,\xi)$ are constant with respect to the variables $l$ and $r$, respectively,
and $\psi_{1}$, $\psi_{2}$ satisfy the homogeneity
\begin{align*}
\psi_{1}(\lambda \xi)=\lambda^{\gamma_{1}}\psi_{1}(\xi),\quad
\psi_{2}(\lambda \xi)=\lambda^{\gamma_{2}}\psi_{2}(\xi),\quad \lambda>0,
\end{align*}
then there exists a constant $C>0$ such that
for any $s,l\geq 0$ and $f\in L^{p}(\mathbb{R}^{d})$,
\begin{align*}%\label{eqn:LP-ineq homo in intro}
\int_{\mathbb{R}^{d}}\left(\int_{s}^{\infty}(t-s)^{\frac{q\gamma_{1}}{\gamma_{2}}-1}
   |L_{\psi_{1}}\mathcal{T}_{\psi_{2}}(t,s)f(x)|^{q}dt\right)^{\frac{p}{q}}dx
\leq C\int_{\mathbb{R}^{d}}|f(x)|^{p}dx,
\end{align*}
which is a generalization of \eqref{eqn: classcal LP ineq}.

In \cite{Krylov 1994}, the author proved the extended version of \eqref{eqn: classcal LP ineq}: if $V$ is a Hilbert space, then for any $p\geq 2$ and $f\in L^{p}((a,b)\times\mathbb{R}^d;V)$ with $-\infty\leq a<b\leq \infty$,
\begin{equation}\label{eqn: LP ineq of Krylov}
\int_{\mathbb{R}^d}\int_{a}^{b}\left(\int_{a}^{t}\|(-\Delta)^{\frac{1}{2}}T_{t-s}f(s,x)\|_{V}^{2}ds\right)^{\frac{p}{2}}dtdx
\leq N_{p}\int_{\mathbb{R}^d}\int_{a}^{b}\|f(t,x)\|_{V}^p dtdx.
\end{equation}
As an application, the inequality \eqref{eqn: LP ineq of Krylov} plays an important role for $L^{p}$-theory of the second order linear stochastic partial differential equations (SPDEs). For nore details, we refer to \cite{Krylov 1999}.

In \cite{I. Kim K.-H. Kim 2016}, the authors proved the inequality \eqref{eqn: LP ineq of Krylov}
for the operator $\mathcal{T}_{\psi}(t,s)$:
if $V$ is a Hilbert space, $p\geq 2$ and $f\in C_{c}^{\infty}((a,b)\times\mathbb{R}^d;V)$
with $-\infty\leq a<b\leq\infty$, then there exists a constant $C>0$ such that
\begin{align}
&\int_{\mathbb{R}^{d}}\int_{a}^{b}\left(\int_{a}^{t}
\|(-\Delta)^{\frac{\gamma}{4}}
\mathcal{T}_{\psi}(t,s)f(s,\cdot)(x)\|_{V}^{2}ds\right)^{\frac{p}{2}}dtdx\nonumber\\
&\qquad\leq C\int_{\mathbb{R}^{d}}\int_{a}^{b}
\|f(t,x)\|_{V}^{p}dtdx.\label{eqn: LP ineq in Kim 2016}
\end{align}
The inequality \eqref{eqn: LP ineq in Kim 2016} plays an important role for the $L^{p}$-theory of SPDEs with pseudo-differential operator of arbitrary order. For more details, we refer to  \cite{I. Kim 2016}.

In this paper, we prove the Littlewood-Paley type inequality for the pseudo differential operator $L_{\psi_{1}}(l)$
and the evolution system $\left\{\mathcal{T}_{\psi_{2}}(t,s)\right\}_{t\geq s\geq 0}$
for $(\psi_{1},\psi_{2})\in\mathfrak{S}_{\rm T}\times\mathfrak{S}$ (see Section \ref{sec: symbol}).
More precisely, if $V$ is a separable Hilbert space, $q\geq 2$ and $f\in L^{p}((a,b)\times \mathbb{R}^{d};V)$,
then there exists a constant $C>0$ such that for any $l\geq 0$,
\begin{align}\label{ineq:LP-ineq intro}
&\int_{\mathbb{R}^{d}}\int_{a}^{b}\left(\int_{a}^{t}(t-s)^{\frac{q\gamma_{1}}{\gamma_{2}}-1}
\|L_{\psi_{1}}(l)
\mathcal{T}_{\psi_{2}}(t,s)f(s,\cdot)(x)\|_{V}^{q}ds\right)^{\frac{p}{q}}dtdx\nonumber\\
&\qquad\leq C\int_{\mathbb{R}^{d}}\int_{a}^{b}
\|f(t,x)\|_{V}^{p}dtdx.
\end{align}

This paper is organized as follows.
In Section \ref{sec: symbol},
we recall the definitions and properties related to the pseudo-differential operator
and the operator $\mathcal{T}_{\psi}(t,s)$ with a symbol $\psi$.
Also we estimate the gradient of the kernel of convolution operators.
In Section \ref{sec: convol},
we prove that the Littlewood-Paley $g$-function $\mathcal{G}_{a,q}$ and $\widetilde{\mathcal{G}}_{a,q}$
(see \eqref{eqn: mathcal G a q} and \eqref{eqn: mathcal G a q 2}) is bounded from
$L^{q}((a,b)\times \mathbb{R}^{d};V)$ with a separable Hilbert space $V$
into $L^{q}((a,b)\times \mathbb{R}^{d})$ (see Theorem \ref{thm: LP ineq p=q}).
In Section \ref{sec:sharp},
we recall that the definition of the maximal function and sharp function and
we estimate the sharp function of $\mathcal{G}_{a,q}$ (see Theorem \ref{thm: sharp fct less than maximal}).
In Section \ref{sec:main},
we prove the Littlewood-Paley type inequality for evolution systems associated with pseudo-differential operators
(see Theorem \ref{thm:LP ineq}).
In Appendix \ref{sec: Appendix A},
we prove Theorem \ref{thm: sharp fct less than maximal}.
In Appendix \ref{sec: Appendix B},
as an auxiliary result, we prove that the Sobolev space $H_{p}^{\alpha}(\mathbb{R}^{d};V)$
is continuously embedded in the Besov space $B_{pp}^{\alpha}(\mathbb{R}^{d};V)$
for a separable Hilbert space $V$ (see Theorem \ref{thm: Hps(V) embedded into Bpps(V)}).

An $L^{p}$-theory for SPDEs with pseudo differential operator of arbitrary order driven by
a HIlbert space-valued stochastic process is important.
As an application of \eqref{ineq:LP-ineq intro}, we will construct an $L^{p}$-theory for SPDEs
with pseudo differential operator of arbitrary order driven by a HIlbert space-valued general Gaussian process.
This work is in progress and will appear in a separate paper.

%%%%%%%%%%%%%%%%%%%%%%%%%%%%%%%%%%%%%%%%%%%%%55
\section{Pseudo-Differential Operators and Evolution Systems}\label{sec: symbol}

Throughout this paper, let $V$ be a separable Hilbert space.
For $1\leq p<\infty$, we denote by $L^{p}(\mathbb{R}^{d};V)$
the Banach space of all strongly measurable functions $u:\mathbb{R}^{d}\rightarrow V$ such that
\begin{align*}
\|u\|_{L^{p}(\mathbb{R}^{d};V)}^{p}=\int_{\mathbb{R}^{d}}\|u(x)\|_{V}^{p}dx<\infty.
\end{align*}
In the case of $V=\mathbb{R}$, we write $L^{p}(\mathbb{R}^{d})=L^{p}(\mathbb{R}^{d};\mathbb{R})$.

For $f\in L^{1}(\mathbb{R}^{d};V)$, the Fourier transform and inverse Fourier transform of $f$ is defined by
\begin{align}\label{eqn:FTand FIT}
\mathcal{F}(f)(\xi)&:=\frac{1}{(2\pi)^{d/2}}\int_{\mathbb{R}^{d}}e^{-i x\cdot\xi}f(x)dx,\nonumber\quad\\
\mathcal{F}^{-1}(f)(x)&:=\frac{1}{(2\pi)^{d/2}}\int_{\mathbb{R}^{d}}e^{i x\cdot\xi}f(\xi)d\xi.
\end{align}
Let $C_{\rm c}^{\infty}(\mathbb{R}^{d};V)$ be the space of
all infinitely differentiable $V$-valued functions on $\mathbb{R}^{d}$ with compact support.
In the case of $V=\mathbb{R}$, we write $C_{\rm c}^{\infty}(\mathbb{R}^{d})=C_{\rm c}^{\infty}(\mathbb{R}^{d};\mathbb{R})$.

%%%%%%%%%%%%%%%%%%%%%%%%%%%%%%%%%%%%%
%\subsection{Symbols and Evolution Systems}\label{subsec: symbol and evol.sys}

Let $\psi:[0,\infty)\times \mathbb{R}^{d}\to \mathbb{C}$ be a given function.
For each $t\geq 0$ and $f\in C_{\rm c}^{\infty}(\mathbb{R}^{d};V)$,
the pseudo-differential operator $L_{\psi}(t)$ with symbol $\psi$ is defined by
\begin{align*}
L_{\psi}(t)f(x)=\mathcal{F}^{-1}(\psi(t,\xi)\mathcal{F}f(\xi))(x)
\end{align*}
if the inverse Fourier transform of right hand side exists.
Then it is obvious that
\begin{align*}
L_{\psi}(t)^{k}f(x)=L_{\psi^{k}}(t)f(x).
\end{align*}
for any $k\in\mathbb{N}$ and $f$ belonging to the domain of $L_{\psi}(t)^{k}$.

For symbols $\psi$ of pseudo-differential operators, we consider the following conditions.
There exist positive constants $\kappa:=\kappa_{\psi},\mu:=\mu_{\psi},\gamma:=\gamma_{\psi}$
and a natural number $N:=N_{\psi}\in\mathbb{N}$ with $N\ge \lfloor\frac{d}{2}\rfloor +1$,
where $\lfloor\eta\rfloor$ is the greatest integer less than or equal to $\eta$, such that
\begin{itemize}
  \item [\textbf{(S1)}] for almost all $t\ge0$ and $\xi\in\mathbb{R}^{d}$ (with respect to the Lebesgue measure),
\begin{equation*}
{\rm Re}[\psi(t,\xi)]\leq -\kappa|\xi|^{\gamma},\quad
\end{equation*}
where ${\rm Re}(z)$ is the real part of the complex number $z$,
  \item [\textbf{(S2)}] for any multi-indices $\alpha=(\alpha_1,\cdots,\alpha_d)\in \mathbb{N}_{0}^{d}$ with
  $|\alpha|:=\alpha_1+\cdots+\alpha_d\leq N$, where $\mathbb{N}_{0}=\mathbb{N}\cup\{0\}$,
   and for almost all $t\ge0$ and $\xi=(\xi_1,\cdots,\xi_d)\in\mathbb{R}^{d}\setminus\widetilde{\boldsymbol{0}}$,
   where $\widetilde{\boldsymbol{0}}:=\{(x_1,\cdots,x_d)|\,\,x_i=0\text{ for some }i=1,\cdots,d\}$,
      \begin{align*}
      |\partial_{\xi}^{\alpha}\psi(t,\xi)|\leq \mu|\xi|^{\gamma-|\alpha|},
      \end{align*}
where $\partial_{\xi}^{\alpha}=\partial_{\xi_1}^{\alpha_1}\cdots \partial_{\xi_d}^{\alpha_d}$
and $\partial_{\xi_i}^{\alpha_i}$ is the $\alpha_i$-th derivative in the variable $\xi_i$,

\item [\textbf{(S3)}] for any multi-indices $\alpha\in \mathbb{N}_{0}^{d}$ with $|\alpha|\leq N$
   and for almost all $t\ge0$ and $\xi\in\mathbb{R}^{d}\setminus\widetilde{\boldsymbol{0}}$,
      \begin{align*}
      |\partial_{t}^{m}\partial_{\xi}^{\alpha}\psi(t,\xi)|\leq \mu|\xi|^{\gamma-|\alpha|},\quad m=0,1.
      \end{align*}
\end{itemize}

Let $\mathfrak{S}$ be the set of all measurable functions (symbols of pseudo-differential operators)
$\psi:[0,\infty)\times \mathbb{R}^{d}\rightarrow \mathbb{C}$ satisfying the conditions \textbf{(S1)} and \textbf{(S2)}.
Also we let $\mathfrak{S}_{\rm T}$ the set of all measurable functions (symbols of pseudo-differential operators)
$\psi:[0,\infty)\times \mathbb{R}^{d}\rightarrow \mathbb{C}$ satisfying the conditions \textbf{(S1)} and \textbf{(S3)}.
By the definitions of $\mathfrak{S}_{\rm T}$ and $\mathfrak{S}$, it is obvious that $\mathfrak{S}_{\rm T}\subset \mathfrak{S}$.

\begin{remark}
\upshape
The conditions \textbf{(S1)} and \textbf{(S2)} have been considered
in \cite{Ji-Kim 2025}. In \cite{I. Kim K.-H. Kim 2016}, the authors considered \textbf{(S1)} and \textbf{(S2)}
with $N=\lfloor\frac{d}{2}\rfloor+1$.
For our study, we need the assumption for the time derivative of symbol
(see Theorems \ref{thm: sharp fct less than maximal} and \ref{thm:LP ineq}).
\end{remark}

\begin{example}[\cite{Ji-Kim 2025}]\label{ex: norm xi power gamma}
\upshape
Let $\gamma>0$ and $\kappa>0$ be given
and let $k:[0,\infty)\to [0,\infty)$ be a nonnegative bounded differentiable function
such that $|k(t)|\leq M_{1}$ and $|k^{\prime}(t)|\leq M_{2}$.
Consider the measurable function $\psi(t,\xi)=-(\kappa+k(t))|\xi|^{\gamma}$
for $t\ge0$ and $\xi\in\mathbb{R}^{d}$.
We can easily see that $\psi$ satisfies the conditions \textbf{(S1)}, \textbf{(S2)} and \textbf{(S3)}.
In fact, the conditions \textbf{(S1)} and \textbf{(S2)} have been checked in \cite{Ji-Kim 2025},
and for any $t\ge0$ and $\xi\in\mathbb{R}^{d}$, it holds that
\begin{align*}
|\partial_{t}\partial_{\xi}^{\alpha}\psi(t,\xi)|
&=|-k^{\prime}(t)\partial_{\xi}^{\alpha}|\xi|^{\gamma}|
\leq M_{2}K_{\gamma,\alpha}|\xi|^{\gamma-|\alpha|},
\end{align*}
which implies that $\psi$ satisfies \textbf{(S3)}.
\end{example}

Let $\psi\in \mathfrak{S}$ be a symbol.
From the condition \textbf{(S1)}, we see that $\exp\left(\int_{s}^{t}\psi(r,\cdot)dr\right)\in L^1(\mathbb{R}^{d})$ for each $t>s$.
Define
\begin{align*}%\label{eqn:p(t,s,x)}
p_{\psi}(t,s,x)=\mathcal{F}^{-1}\left(\exp\left(\int_{s}^{t}\psi(r,\cdot)dr\right)\right)(x),\quad t>s,\quad x\in\mathbb{R}^d.
\end{align*}
By Corollary 2.10 in \cite{Ji-Kim 2025}, we have $p_{\psi}(t,s,\cdot)\in L^{1}(\mathbb{R}^d)$ for each $t>s$.
Therefore, for each $f\in L^{p}(\mathbb{R}^{d};V)$ and $t>s$,
by applying the Young's convolution inequality, we can define
\begin{equation}\label{eqn: mathcal T}
\mathcal{T}_{\psi}(t,s)f(x)=p_{\psi}(t,s,\cdot)*f(x)=\int_{\mathbb{R}^{d}}p_{\psi}(t,s,x-y)f(y)dy
\end{equation}
as $\mathcal{T}_{\psi}(t,s)f\in L^{p}(\mathbb{R}^{d};V)$.
On the other hand, for any $t\geq r\geq s$ and $f\in L^{p}(\mathbb{R}^{d};V)$,
we can easily see that
\begin{align*}
\mathcal{T}_{\psi}(t,r)\mathcal{T}_{\psi}(r,s)f(x)
&=\mathcal{T}_{\psi}(t,s)f(x),
\end{align*}
which implies that the family $\{\mathcal{T}_{\psi}(t,s)\}_{t\geq s\geq 0}$ is an evolution system.

%%%%%%%%%%%%%%%%%%%%%%%%%%%%%%%%
\section{Boundedness of Littlewood-Paley $g$-Functions}\label{sec: convol}

Let $(\psi_{1},\psi_{2})\in\mathfrak{S}_{\rm T}\times\mathfrak{S}$
and let $l\geq 0$ and $q\geq 2$ be given.
Then for any $a,b\in\mathbb{R}$ with $a<b$, $f\in C_{c}^{\infty}((a,b)\times \mathbb{R}^{d}; V)$
and $x\in\mathbb{R}^{d}$, we define the Littlewood-Paley $g$-functions $\mathcal{G}_{a,q}$ and $\widetilde{\mathcal{G}}_{a,q}$
by
\begin{align}
\mathcal{G}_{a,q} f(l,t,x)
&=\left[\int_{a}^{t}(t-s)^{\frac{q\gamma_{1}}{\gamma_{2}}-1}
             \|L_{\psi_{1}}(l)\mathcal{T}_{\psi_{2}}( t,s)f(s,\cdot)(x)\|_{V}^{q}ds\right]^{\frac{1}{q}}\nonumber\\
&=\left[\int_{a}^{t}(t-s)^{\frac{q\gamma_{1}}{\gamma_{2}}-1}
            \|\left((L_{\psi_{1}}(l)p_{\psi_{2}}(t,s,\cdot))*f(s,\cdot)\right)(x)\|_{V}^{q}ds\right]^{\frac{1}{q}},
            \label{eqn: mathcal G a q}\\
\widetilde{\mathcal{G}}_{a,q} f(t,x)
&:=\mathcal{G}_{a,q} f(t,t,x)\nonumber\\
&=\left[\int_{a}^{t}(t-s)^{\frac{q\gamma_{1}}{\gamma_{2}}-1}
            \|\left((L_{\psi_{1}}(t)p_{\psi_{2}}(t,s,\cdot))*f(s,\cdot)\right)(x)\|_{V}^{q}ds\right]^{\frac{1}{q}}.\label{eqn: mathcal G a q 2}
\end{align}
Then the left-hand sides of the generalized Littlewood-Paley type inequality
(see \eqref{ineq:LP-ineq} and \eqref{ineq:LP-ineq l=t}) can be expressed as
\begin{align}\label{eqn:Lp norm of mathcal G}
\int_{\mathbb{R}^{d}}\int_{a}^{b}|\mathcal{G}_{a,q} f(l,t,x)|^{p}dtdx,\quad
\int_{\mathbb{R}^{d}}\int_{a}^{b}|\widetilde{\mathcal{G}}_{a,q} f(t,x)|^{p}dtdx,
\end{align}
respectively.
In this section, we prove that the operators $\mathcal{G}_{a,q}$ and $\widetilde{\mathcal{G}}_{a,q}$
are bounded from $L^{q}((a,b)\times \mathbb{R}^{d};V)$
into $L^{q}((a,b)\times \mathbb{R}^{d})$.
For our purpose, we need a continuous embedding theorem for vector-valued Sobolev space into the vector-valued Besov space
which will be proved in Appendix B (see Theorem \ref{thm: Hps(V) embedded into Bpps(V)}).

We now introduce the Besov space
and then we first recall the Littlewood-Paley decomposition (see \eqref{eqn: LP decomposition}).
It is well known that there exists a function $\Phi\in C_{\rm c}^{\infty}(\mathbb{R}^{d})$
such that $\Phi$ is nonnegative and
\begin{align}
\text{supp}\, \Phi =\{\xi\in\mathbb{R}^{d}\,|\, 1/2\leq |\xi|\leq 2\},\quad
\sum_{j=-\infty}^{\infty}\Phi(2^{-j}\xi)=1\label{eqn: sum of phi is one}
\end{align}
(see, e.g., Lemma 6.1.7 in \cite{Bergh 1976}).
For each $j\in\mathbb{Z}$ and $f\in L^{2}(\mathbb{R}^{d};V)$, put
\begin{align}
\Delta_{j}f(x)&=\mathcal{F}^{-1}(\Phi(2^{-j}\xi)\mathcal{F}f(\xi))(x),\label{eqn: LP operator}\\
S_{0}(f)(x)&=\sum_{j=-\infty}^{0}\Delta_{j}f(x).\label{eqn: S0}
\end{align}
From \eqref{eqn: sum of phi is one}, it is obvious that
\begin{equation}\label{eqn: LP decomposition}
f(x)=S_{0}(f)(x)+\sum_{j=1}^{\infty}\Delta_{j}f(x),
\end{equation}
which is called the Littlewood-Paley decomposition of $f$.
Note that for any $i,j\in\mathbb{Z}$ with $|i-j|\geq 2$, it holds that
\begin{align*}
\Delta_{i}\Delta_{j}=0,
\end{align*}
which is known as the almost orthogonality,
which implies that the following decomposition:
\begin{align}\label{eqn: pseudo orthogonal property}
\left(L_{\psi_{1}}(l)p_{\psi_{2}}(t,s,\cdot)\right)*f(s,\cdot)
&=\sum_{j=-\infty}^{1}\left(L_{\psi_{1}}(l)p_{\psi_{2},j}(t,s,\cdot)\right)*S_{0}(f(s,\cdot))\nonumber\\
&\qquad   +\sum_{j=1}^{\infty}\sum_{j-1\leq i\leq j+1}\left(L_{\psi_{1}}(l)p_{\psi_2,i}(t,s,\cdot)\right)*f_{j}(s,\cdot),
\end{align}
where $f_{i}=\Delta_{i}f$ and $p_{\psi_{2},j}(t,s,x)=\Delta_{ j}p_{\psi_{2}}(t,s,x)$.

\begin{lemma}[\cite{Ji-Kim 2025}, Lemma 4.2]\label{lem: boundedness of L psi of p j}
Let $(\psi_{1},\psi_{2})\in\mathfrak{S}^{2}$.
Then there exist constants $c,C>0$ depending on $ \boldsymbol{\kappa},\boldsymbol{\mu}$ and $d$ such that
for any $l\geq 0$, $j\in\mathbb{Z}$ and $t>s$,
\begin{align*}
\|L_{\psi_{1}}(l)p_{\psi_{2},j}(t,s,\cdot)\|_{L^{1}(\mathbb{R}^{d})}
&\leq C2^{j\gamma_{1}}e^{-c(t-s)2^{j\gamma_{2}}}.
\end{align*}
\end{lemma}

For $p\geq 1$ and $\alpha\in\mathbb{R}$, let $B_{pp}^{\alpha}(\mathbb{R}^{d};V)$
be the Besov space with the norm
\begin{align*}
\|f\|_{B_{pp}^{s}(\mathbb{R}^{d};V)}
=\|S_{0}(f)\|_{L^{p}(\mathbb{R}^{d};V)}+\left(\sum_{j=1}^{\infty}2^{j\alpha p}
            \|\Delta_{j}f\|_{L^{p}(\mathbb{R}^{d};V)}^{p}\right)^{\frac{1}{p}},
\end{align*}
and let $H_{p}^{\alpha}(\mathbb{R}^{d};V)$ be the Sobolev space with the norm
\begin{align*}
\|u\|_{H_{p}^{\alpha}(\mathbb{R}^{d};V)}
=\|(1-\Delta)^{\frac{\alpha}{2}}u\|_{L^{p}(\mathbb{R}^{d};V)}.
\end{align*}
If $V=\mathbb{R}$, then we write
\begin{align*}
B_{pp}^{\alpha}(\mathbb{R}^{d})=B_{pp}^{\alpha}(\mathbb{R}^{d};\mathbb{R}),
\quad
H_{p}^{\alpha}(\mathbb{R}^{d})=H_{p}^{\alpha}(\mathbb{R}^{d};\mathbb{R}).
\end{align*}
In Theorem \ref{thm: Hps(V) embedded into Bpps(V)}, we prove that if $\alpha\in\mathbb{R}$ and $p\geq 2$,
then it holds that
\begin{align*}
\|f\|_{B_{pp}^{\alpha}(\mathbb{R}^{d};V)}\leq C\|f\|_{H_{p}^{\alpha}(\mathbb{R}^{d};V)}
\end{align*}
for some constant $C>0$.
This will be used to prove
the inequalities given in \eqref{eqn:LP p=q 1} and \eqref{eqn:LP p=q 2} (in Theorem \ref{thm: LP ineq p=q}).

\begin{lemma}\label{lem: LP ineq for p equal lambda fixed s}
Let $(\psi_{1},\psi_{2})\in \mathfrak{S}_{\rm T}\times\mathfrak{S}$.
Let $q\geq 1$ be given. Let $f\in L^{q}((a,b);B_{qq}^{0}(\mathbb{R}^{d};V))$ with $-\infty< a<b< \infty$.
Then there exists a constant $C>0$ depending on $a,b,d,q,\boldsymbol{\gamma},\boldsymbol{\mu}$
and $\boldsymbol{\kappa}$ such that for any $l\geq 0$,
\begin{align}
\int_{\mathbb{R}^{d}}\int_{a}^{b}|\mathcal{G}_{a,q}f(l,t,x)|^{q}dtdx
&\leq C\int_{a}^{b}\|f(s,\cdot)\|_{B_{qq}^{0}(\mathbb{R}^{d};V)}^{q}ds,\label{eqn:FE for LPI}\\
\int_{\mathbb{R}^{d}}\int_{a}^{b}|\widetilde{\mathcal{G}}_{a,q}f(t,x)|^{q}dtdx
&\leq C\int_{a}^{b}\|f(s,\cdot)\|_{B_{qq}^{0}(\mathbb{R}^{d};V)}^{q}ds.\label{eqn:FE for LPI 2}
\end{align}
That is, the operators $\mathcal{G}_{a,q}$, $\widetilde{\mathcal{G}}_{a,q}$ are bounded
from $L^{q}((a,b);B_{qq}^{0}(\mathbb{R}^{d};V))$ into $L^{q}((a,b)\times \mathbb{R}^{d})$.
\end{lemma}

\begin{proof}
By the Fubini theorem, we obtain that
\begin{align*}
&\int_{\mathbb{R}^{d}}\int_{a}^{b}|\mathcal{G}_{a,q}f(l,t,x)|^{q}dtdx\\
&=\int_{\mathbb{R}^{d}}\int_{a}^{b}\int_{a}^{t}(t-s)^{\frac{q\gamma_{1}}{\gamma_{2}}-1}
                    \|\left((L_{\psi_{1}}(l)p_{\psi_{2}}(t,s,\cdot))*f(s,\cdot)\right)(x)\|_{V}^{q}dsdtdx\\
&=\int_{a}^{b}\int_{\mathbb{R}^{d}}\int_{s}^{b}(t-s)^{\frac{q\gamma_{1}}{\gamma_{2}}-1}
                    \|\left((L_{\psi_{1}}(l)p_{\psi_{2}}(t,s,\cdot))*f(s,\cdot)\right)(x)\|_{V}^{q}dtdxds.
\end{align*}
By applying \eqref{eqn: pseudo orthogonal property}, Lemma \ref{lem: boundedness of L psi of p j}
and the arguments used in the proof of Theorem 4.3 in \cite{Ji-Kim 2025},
we see that there exists a constant $C>0$ such that
\begin{align*}
\int_{\mathbb{R}^{d}}\int_{s}^{b}(t-s)^{\frac{q\gamma_{1}}{\gamma_{2}}-1}
\|\left((L_{\psi_{1}}(l)p_{\psi_{2}}(t,s,\cdot))*f(s,\cdot)\right)(x)\|_{V}^{q}dtdx
\leq C \|f(s,\cdot)\|_{B_{qq}^{0}(\mathbb{R}^{d};V)}^{q},
\end{align*}
which implies the inequality given in \eqref{eqn:FE for LPI}.
By the same arguments, we can prove the inequality given in \eqref{eqn:FE for LPI 2}.
\end{proof}

\begin{theorem} \label{thm: LP ineq p=q}
Let $(\psi_{1},\psi_{2})\in\mathfrak{S}_{\rm T}\times\mathfrak{S}$ and let $q\geq 2$ be given.
Then it holds that
\begin{itemize}
  \item [\rm{(i)}] if $-\infty<a<b<\infty$, then
  there exists a constant $C_{1}>0$ depending on $a,b,d,q, \boldsymbol{\gamma}$, $\boldsymbol{\mu}$ and $\boldsymbol{\kappa}$
such that for any $l\geq 0$  and $f\in L^{q}((a,b)\times\mathbb{R}^{d};V)$,
\begin{align}
\int_{\mathbb{R}^{d}}\int_{a}^{b}|\mathcal{G}_{a,q}f(l,t,x)|^{q}dtdx
\leq C_{1}\int_{a}^{b}\|f(s,\cdot)\|_{L^{q}(\mathbb{R}^{d};V)}^{q}ds,\label{eqn:LP p=q 1}\\
\int_{\mathbb{R}^{d}}\int_{a}^{b}|\widetilde{\mathcal{G}}_{a,q}f(t,x)|^{q}dtdx
\leq C_{1}\int_{a}^{b}\|f(s,\cdot)\|_{L^{q}(\mathbb{R}^{d};V)}^{q}ds.\label{eqn:LP p=q 2}
\end{align}
That is, the operators $\mathcal{G}_{a,q}$ and $\widetilde{\mathcal{G}}_{a,q}$
are bounded from $L^{q}((a,b)\times\mathbb{R}^{d};V)$ into $L^{q}((a,b)\times \mathbb{R}^{d})$.

  \item [\rm{(ii)}] if $q=2$, then there exists a constant $C_{2}>0$ depending on
  $\mu_{1},\kappa_{2},$ and $\boldsymbol{\gamma}$ such that for any $l\ge0$ and $f\in L^{2}(\mathbb{R}^{d+1};V)$,
\begin{align}\label{eqn: LP 2}
\int_{\mathbb{R}^{d}}\int_{-\infty}^{\infty}|\mathcal{G}_{-\infty,2}f(l,t,x)|^2 dtdx
\leq C_{2}\int_{-\infty}^{\infty}\|f(s,\cdot)\|_{L^{2}(\mathbb{R}^{d};V)}^{2}ds.
\end{align}
That is, the operator $\mathcal{G}_{-\infty,2}$ is bounded from $L^{2}(\mathbb{R}^{d+1};V)$
into $L^{2}(\mathbb{R}^{d+1})$.
\end{itemize}
\end{theorem}

\begin{proof}
(i) \enspace By applying Lemma \ref{lem: LP ineq for p equal lambda fixed s} and Theorem \ref{thm: Hps(V) embedded into Bpps(V)},  we obtain that
\begin{align*}
\int_{\mathbb{R}^{d}}\int_{a}^{b}|\mathcal{G}_{a,q}f(l,t,x)|^{q}dtdx
&\leq C\int_{a}^{b}\|f(s,\cdot)\|_{B_{qq}^{0}(\mathbb{R}^{d};V)}^{q}ds\\
&\leq CC'\int_{a}^{b}\|f(s,\cdot)\|_{H_{q}^{0}(\mathbb{R}^{d};V)}^{q}ds\\
&=CC'\int_{a}^{b}\|f(s,\cdot)\|_{L^{q}(\mathbb{R}^{d};V)}^{q}ds,
\end{align*}
and then by taking $C_{1}=CC'$, we have the inequality given in \eqref{eqn:LP p=q 1}.
By the same arguments, we can prove the inequality given in  \eqref{eqn:LP p=q 2}.

(ii) \enspace Let $q=2$. Then by Plancherel theorem and the conditions
\textbf{(S1)}, \textbf{(S2)} and \textbf{(S3)}, we obtain that
\begin{align*}
&\int_{-\infty}^{\infty}\int_{\mathbb{R}^{d}}|\mathcal{G}_{-\infty,2}f(l,t,x)|^2 dxdt\\
&=\int_{-\infty}^{\infty}\int_{-\infty}^{t}\int_{\mathbb{R}^{d}}(t-s)^{\frac{2\gamma_{1}}{\gamma_{2}}-1}
      \|L_{\psi_{1}}(l)\mathcal{T}_{\psi_{2}}( t,s)f(s,\cdot)(x)\|_{V}^{2}dxdsdt\\
&=\int_{-\infty}^{\infty}\int_{s}^{\infty}\int_{\mathbb{R}^{d}}(t-s)^{\frac{2\gamma_{1}}{\gamma_{2}}-1}
   \|L_{\psi_{1}}(l)\mathcal{T}_{\psi_{2}}( t,s)f(s,\cdot)(x)\|_{V}^{2}dxdtds\\
&=\int_{-\infty}^{\infty}\int_{s}^{\infty}\int_{\mathbb{R}^{d}}(t-s)^{\frac{2\gamma_{1}}{\gamma_{2}}-1}\left\|\psi_{1}(l,\xi)\exp\left(\int_{s}^{t}
\psi_{2}(r,\xi)dr\right)\mathcal{F}f(s,\xi)\right\|_{V}^{2}d\xi dtds\\
&\leq \int_{-\infty}^{\infty}\int_{s}^{\infty}\int_{\mathbb{R}^{d}}(t-s)^{\frac{2\gamma_{1}}{\gamma_{2}}-1}
   \mu_{1}^{2}|\xi|^{2\gamma_{1}}e^{-2\kappa_{2}(t-s)|\xi|^{\gamma_{2}}}\left\|\mathcal{F}f(s,\xi)\right\|_{V}^{2}d\xi dtds\\
&=\mu_{1}^{2}\int_{-\infty}^{\infty}\int_{\mathbb{R}^{d}}|\xi|^{2\gamma_{1}}
  \left(\int_{s}^{\infty}(t-s)^{\frac{2\gamma_{1}}{\gamma_{2}}-1}e^{-2\kappa_{2}(t-s)|\xi|^{\gamma_{2}}}dt\right)
  \left\|\mathcal{F}f(s,\xi)\right\|_{V}^{2} d\xi ds\\
&=\mu_{1}^{2}\int_{-\infty}^{\infty}\int_{\mathbb{R}^{d}}|\xi|^{2\gamma_{1}}\Gamma
  \left(\frac{2\gamma_{1}}{\gamma_{2}}\right)(2\kappa_{2}|\xi|^{\gamma_{2}})^{-\frac{2\gamma_{1}}{\gamma_{2}}}
    \left\|\mathcal{F}f(s,\xi)\right\|_{V}^{2} d\xi ds\\
&=\mu_{1}^{2}\Gamma\left(\frac{2\gamma_{1}}{\gamma_{2}}\right)(2\kappa_{2})^{-\frac{2\gamma_{1}}{\gamma_{2}}}
   \int_{-\infty}^{\infty}\int_{\mathbb{R}^{d}}\left\|f(s,x)\right\|_{V}^{2} dxds,
\end{align*}
and then by taking $C_{2}=\mu_{1}^{2}\Gamma\left(\frac{2\gamma_{1}}{\gamma_{2}}\right)(2\kappa_{2})^{-\frac{2\gamma_{1}}{\gamma_{2}}}$,
we have the inequality given in \eqref{eqn: LP 2}.
\end{proof}

%%%%%%%%%%%%%%%%%%%%%%%%%%%%%%%%%%%
\section{Estimations of Sharp Functions}\label{sec:sharp}

We now recall that the definitions of the maximal function and the sharp function.
For $x\in\mathbb{R}^{d}$ and $r>0$, we denote $B_{r}(x)=\{y\in \mathbb{R}^{d}\,:\, |x-y|<r\}$ and $B_{r}=B_{r}(0)$.
Then for any $R\geq 0$ and a $V$-valued locally integrable function $h$ on $\mathbb{R}^{d}$,
the maximal function $\mathbb{M}_{x}^{R}h(x)$ is defined by
\begin{align*}
\mathbb{M}_{x}^{R}h(x)=\sup_{r>R}\frac{1}{|B_{r}(x)|}\int_{B_{r}(x)}\|h(y)\|_{V}dy,\quad
\mathbb{M}_{x}h(x)=\mathbb{M}_{x}^{0}h(x).
\end{align*}
If $R_{1}\geq R_{2}$, then it is obvious that
\begin{align*}
\mathbb{M}_{x}^{R_{1}}h(x)\leq \mathbb{M}_{x}^{R_{2}}h(x).
\end{align*}
Similarly, for a $V$-valued locally integrable function $h$ on $\mathbb{R}$, we define
\begin{align*}
\mathbb{M}_{t}^{R}h(t)=\sup_{r>R}\frac{1}{2r}\int_{-r}^{r}\|h(t)\|_{V}dt,\quad
\mathbb{M}_{t}h(t)=\mathbb{M}_{t}^{0}h(t).
\end{align*}
If $h=h(t,x)$, we denote
\begin{align*}
\mathbb{M}_{x}^{R}h(t,x)=\mathbb{M}_{x}^{R}(h(t,\cdot))(x),\quad
\mathbb{M}_{t}^{R}h(t,x)=\mathbb{M}_{t}^{R}(h(\cdot,x))(t).
\end{align*}

For a function $h$ defined on $\mathbb{R}^{d+1}$ with $h\in L^{1}_{\rm loc}(\mathbb{R}^{d+1})$
(the space of all locally integrable functions on $\mathbb{R}^{d+1}$), the sharp function $h^{\#}$ of $h$ is defined by
\begin{align*}
h^{\#}(t,x):=\sup_{Q}\frac{1}{|Q|}\int_{Q}|h(r,z)-h_{Q}|drdz,\quad (t,x)\in\mathbb{R}^{d+1},
\end{align*}
where $h_{Q}=\frac{1}{|Q|}\int_{Q}h(r,z)drdz$ and the supremum is taken all $Q$ containing $(t,x)$ of the type
\begin{align*}
Q=(t-R,t+R)\times B_{R^{1/\gamma_{2}}}(x),\quad R>0.
\end{align*}

For our proofs of the generalized Littlewood-Paley type inequalities given in \eqref{ineq:LP-ineq} and \eqref{ineq:LP-ineq l=t},
we need the estimations (in the following theorem) of the sharp functions
$(\mathcal{G}_{a,q}f(l,\cdot,\cdot))^{\#}$ and $(\widetilde{\mathcal{G}}_{a,q}f)^{\#}$.

\begin{theorem}\label{thm: sharp fct less than maximal}
Let $(\psi_{1},\psi_{2})\in \mathfrak{S}_{\rm T}\times \mathfrak{S}$
satisfying that $N_{1},N_{2}>d+2+\lfloor\gamma_{1}\rfloor+\lfloor\gamma_{2}\rfloor$.
Let $q\geq 2 $ be given. Then it holds that
\begin{itemize}
  \item [\rm (i)] if $-\infty< a<b<\infty$, then there exist constants $C_{1},C_{2},C_{3}>0$ such that
  for any $f\in C_{\rm c}^{\infty}((a,b)\times\mathbb{R}^{d};V)$, $l\geq 0$ and $(t,x)\in (a,b)\times \mathbb{R}^{d}$,
\begin{align}
(\mathcal{G}_{a,q}f(l,\cdot,\cdot))^{\#}(t,x)
&\leq C_{1} \left(\mathbb{M}_{t}\mathbb{M}_{x}\|f\|_{V}^{q}(t,x)\right)^{\frac{1}{q}},\label{eqn: sharp fct esti}\\
(\widetilde{\mathcal{G}}_{a,q}f)^{\#}(t,x)
&\leq (C_{2}+C_{3}(b-a)^{q})^{\frac{1}{q}} \left(\mathbb{M}_{t}\mathbb{M}_{x}\|f\|_{V}^{q}(t,x)\right)^{\frac{1}{q}}.\label{eqn: sharp fct esti 2}
\end{align}
  \item [\rm (ii)] if $q=2$ and $-\infty\leq  a<b \leq \infty$,then there exists a constant $C_{4}>0$ such that
  for any $f\in C_{\rm c}^{\infty}(\mathbb{R}\times\mathbb{R}^{d};V)$, $l\geq 0$ and $(t,x)\in (a,b)\times \mathbb{R}^{d}$,
\begin{align}
(\mathcal{G}_{a,2}f(l,\cdot,\cdot))^{\#}(t,x)
&\leq C_{4}\left(\mathbb{M}_{t}\mathbb{M}_{x}\|f\|_{V}^{q}(t,x)\right)^{\frac{1}{q}}.\label{eqn: sharp fct esti 3}
\end{align}
\end{itemize}
\end{theorem}

A proof of Theorem \ref{thm: sharp fct less than maximal} will be given in Appendix \ref{sec: Appendix A}.

%%%%%%%%%%%%%%%%%%%%%%%%%%%%%%%%%%%%%%%%%%%
\section{Generalized Littlewood-Paley Type Inequality}\label{sec:main}

In this section, we prove a generalization of Littlewood-Paley type inequality for evolution systems
associated with pseudo-differential operators.
For our proof, we need the following two theorems.

\begin{theorem}[Fefferman-Stein]\label{thm:FS}
Let $1<p<\infty$.
Then there exists a constant $N$ depending on $p$ such that for any $f\in L^{p}(\mathbb{R}^{d})$,
\begin{align*}
\|f\|_{L^{p}(\mathbb{R}^{d})}\leq N\|f^{\#}\|_{L^{p}(\mathbb{R}^{d})}.
\end{align*}
\end{theorem}

A proof of Theorem \ref{thm:FS} will be given in Appendix \ref{sec:Appendix C}.

\begin{theorem}[Hardy-Littlewood maximal theorem]\label{thm:HL max}
Let $1<r\leq \infty$.
Then there exists a constant $N>0$ depending on $r$ and $d$
such that for any $f\in L^{r}(\mathbb{R}^{d})$,
\begin{align*}
\|\mathbb{M}f\|_{L^{r}(\mathbb{R}^{d})}\leq N\|f\|_{L^{r}(\mathbb{R}^{d})}.
\end{align*}
\end{theorem}

\begin{proof}
See Theorem 1.3.1 in \cite{Stein 1970}.
\end{proof}

\begin{theorem}\label{thm:LP ineq}
Let $(\psi_{1},\psi_{2})\in\mathfrak{S}_{\rm T}\times\mathfrak{S}$
such that $N_{1},N_{2}>d+2+\lfloor\gamma_{1}\rfloor+\lfloor\gamma_{2}\rfloor$.
Then for any $q\geq 2$ and $p\geq q$, it holds that

\begin{itemize}
  \item [\rm{(i)}] if  $-\infty< a< b<\infty$, then
there exist constants $C_{1},C_{2}>0$ depending on $a,b,p,q,d,\boldsymbol{\gamma},\boldsymbol{\kappa}$ and $\boldsymbol{\mu}$
such that for any $f\in L^p((a,b)\times\mathbb{R}^{d}; V)$,
\begin{align}\label{ineq:LP-ineq}
&\int_{\mathbb{R}^{d}}\int_{a}^{b}\left(\int_{a}^{t}(t-s)^{\frac{q\gamma_{1}}{\gamma_{2}}-1}
\|L_{\psi_{1}}(l)
\mathcal{T}_{\psi_{2}}(t,s)f(s,\cdot)(x)\|_{V}^{q}ds\right)^{\frac{p}{q}}dtdx\nonumber\\
&\qquad\leq C_{1}\int_{\mathbb{R}^{d}}\int_{a}^{b}
\|f(t,x)\|_{V}^{p}dtdx
\end{align}
for all $l\geq 0$, and
\begin{align}\label{ineq:LP-ineq l=t}
&\int_{\mathbb{R}^{d}}\int_{a}^{b}\left(\int_{a}^{t}(t-s)^{\frac{q\gamma_{1}}{\gamma_{2}}-1}
\|L_{\psi_{1}}(t)
\mathcal{T}_{\psi_{2}}(t,s)f(s,\cdot)(x)\|_{V}^{q}ds\right)^{\frac{p}{q}}dtdx\nonumber\\
&\qquad\leq C_{2}\int_{\mathbb{R}^{d}}\int_{a}^{b}
\|f(t,x)\|_{V}^{p}dtdx,
\end{align}
  \item [\rm{(ii)}] (if $q=2$,) there exists a constant $C_{3}>0$ depending on
  $p,d,\boldsymbol{\gamma},\boldsymbol{\kappa}$ and $\boldsymbol{\mu}$
such that for any $l\geq 0$ and $f\in L^p(\mathbb{R}^{d+1}; V)$,
\begin{align}\label{ineq:LP-ineq q=2}
&\int_{\mathbb{R}^{d}}\int_{-\infty}^{\infty}\left(\int_{-\infty}^{t}(t-s)^{\frac{2\gamma_{1}}{\gamma_{2}}-1}
\|L_{\psi_{1}}(l)
\mathcal{T}_{\psi_{2}}(t,s)f(s,\cdot)(x)\|_{V}^{2}ds\right)^{\frac{p}{2}}dtdx\nonumber\\
&\qquad\leq C_{3}\int_{\mathbb{R}^{d}}\int_{-\infty}^{\infty}
\|f(t,x)\|_{V}^{p}dtdx.
\end{align}
\end{itemize}
\end{theorem}

\begin{proof}
(i)\enspace If $p=q$, then the inequalities \eqref{ineq:LP-ineq} and \eqref{ineq:LP-ineq l=t} hold
by \eqref{eqn:LP p=q 1} and \eqref{eqn:LP p=q 2}, respectively, in Theorem \ref{thm: LP ineq p=q}.
Suppose that $p>q$.
By applying Theorem \ref{thm:FS}, \eqref{eqn: sharp fct esti} and Theorem \ref{thm:HL max} with $r=\frac{p}{q}>1$,
we obtain that
 \begin{align*}
&\int_{\mathbb{R}^{d}}\int_{a}^{b}\left(\int_{a}^{t}(t-s)^{\frac{q\gamma_{1}}{\gamma_{2}}-1}
\|L_{\psi_{1}}(l)
\mathcal{T}_{\psi_{2}}(t,s)f(s,\cdot)(x)\|_{V}^{q}ds\right)^{\frac{p}{q}}dtdx\\
&=\|\mathcal{G}_{a,q}f(l,\cdot,\cdot)\|_{L^{p}((a,b)\times\mathbb{R}^{d})}^{p}\\
&\leq C^{(1)}\|(\mathcal{G}_{a,q}f(l,\cdot,\cdot))^{\#}\|_{L^{p}((a,b)\times\mathbb{R}^{d})}^{p}\\
&=C^{(1)}\int_{\mathbb{R}^{d}}\int_{a}^{b}|(\mathcal{G}_{a,q}f(l,\cdot,\cdot))^{\#}(t,x)|^{p}dtdx\\
&\leq C^{(1)}C^{(2)}\int_{\mathbb{R}^{d}}\int_{a}^{b}(\mathbb{M}_{t}\,\,\mathbb{M}_{x}\|f\|_{V}^{q}(t,x))^{\frac{p}{q}}dtdx\\
&\leq C^{(1)}C^{(2)}C^{(3)}\int_{\mathbb{R}^{d}}\int_{a}^{b}(\mathbb{M}_{x}\|f\|_{V}^{q}(t,x))^{\frac{p}{q}}dtdx\\
&\leq C^{(1)}C^{(2)}C^{(3)}C^{(4)}\int_{\mathbb{R}^{d}}\int_{a}^{b}\|f(t,x)\|_{V}^{p}dtdx
\end{align*}
for some constants $C^{(1)},C^{(2)},C^{(3)},C^{(4)}>0$. This implies \eqref{ineq:LP-ineq}.

Similarly, by applying Theorem \ref{thm:FS}, \eqref{eqn: sharp fct esti 2}
and Theorem \ref{thm:HL max}, we obtain that
\begin{align*}
&\int_{\mathbb{R}^{d}}\int_{a}^{b}\left(\int_{a}^{t}(t-s)^{\frac{q\gamma_{1}}{\gamma_{2}}-1}
\|L_{\psi_{1}}(t)
\mathcal{T}_{\psi_{2}}(t,s)f(s,\cdot)(x)\|_{V}^{q}ds\right)^{\frac{p}{q}}dtdx\\
&=\|\widetilde{\mathcal{G}}_{a,q}f\|_{L^{p}((a,b)\times\mathbb{R}^{d})}^{p}\\
&\leq C^{(1)}\|(\widetilde{\mathcal{G}}_{a,q}f)^{\#}\|_{L^{p}((a,b)\times\mathbb{R}^{d})}^{p}\\
&=C^{(1)}\int_{\mathbb{R}^{d}}\int_{a}^{b}|(\widetilde{\mathcal{G}}_{a,q}f)^{\#}(t,x)|^{p}dtdx\\
&\leq C^{(1)}C^{(5)}\int_{\mathbb{R}^{d}}\int_{a}^{b}(\mathbb{M}_{t}\,\,\mathbb{M}_{x}\|f\|_{V}^{q}(t,x))^{\frac{p}{q}}dtdx\\
&\leq C^{(1)}C^{(5)}C^{(3)}\int_{\mathbb{R}^{d}}\int_{a}^{b}(\mathbb{M}_{x}\|f\|_{V}^{q}(t,x))^{\frac{p}{q}}dtdx\\
&\leq C^{(1)}C^{(5)}C^{(3)}C^{(4)}\int_{\mathbb{R}^{d}}\int_{a}^{b}\|f(t,x)\|_{V}^{p}dtdx
\end{align*}
for some constants $C^{(1)},C^{(5)},C^{(3)},C^{(4)}>0$. This implies \eqref{ineq:LP-ineq l=t}.

(ii)\enspace Since $q=2$, $p\ge q=2$.
If $p=2$, then the inequality \eqref{ineq:LP-ineq q=2} holds by \eqref{eqn: LP 2} in Theorem \ref{thm: LP ineq p=q}.
Suppose that $p>2$. By applying Theorem \ref{thm:FS}, \eqref{eqn: sharp fct esti 3} and Theorem \ref{thm:HL max} with $r=\frac{p}{2}>1$,
we obtain that
 \begin{align*}
 &\int_{\mathbb{R}^{d}}\int_{-\infty}^{\infty}\left(\int_{-\infty}^{t}(t-s)^{\frac{2\gamma_{1}}{\gamma_{2}}-1}
\|L_{\psi_{1}}(l)
\mathcal{T}_{\psi_{2}}(t,s)f(s,\cdot)(x)\|_{V}^{2}ds\right)^{\frac{p}{2}}dtdx\\
&=\|\mathcal{G}_{-\infty,2}f(l,\cdot,\cdot)\|_{L^{p}(\mathbb{R}^{d+1})}^{p}\\
&\leq C^{(1)}\|(\mathcal{G}_{-\infty,2}f(l,\cdot,\cdot))^{\#}\|_{L^{p}(\mathbb{R}^{d+1})}^{p}\\
&=C^{(1)}\int_{\mathbb{R}^{d}}\int_{-\infty}^{\infty}|(\mathcal{G}_{-\infty,2}f(l,\cdot,\cdot))^{\#}(t,x)|^{p}dtdx\\
&\leq C^{(1)}C^{(2)}\int_{\mathbb{R}^{d}}\int_{-\infty}^{\infty}
            (\mathbb{M}_{t}\,\,\mathbb{M}_{x}\|f\|_{V}^{2}(t,x))^{\frac{p}{2}}dtdx\\
&\leq C^{(1)}C^{(2)}C^{(3)}\int_{\mathbb{R}^{d}}\int_{-\infty}^{\infty}(\mathbb{M}_{x}\|f\|_{V}^{2}(t,x))^{\frac{p}{2}}dtdx\\
&\leq C^{(1)}C^{(2)}C^{(3)}C^{(4)}\int_{\mathbb{R}^{d}}\int_{-\infty}^{\infty}\|f(t,x)\|_{V}^{p}dtdx
\end{align*}
for some constants $C^{(1)},C^{(2)},C^{(3)},C^{(4)}>0$. This implies \eqref{ineq:LP-ineq q=2}.

\end{proof}

\begin{corollary}\label{cor: LP Lpsi is Laplacian}
Let $\psi\in\mathfrak{S}$ such that $N>d+2+\lfloor\frac{\gamma}{q}\rfloor+\lfloor\gamma\rfloor$.
Then for any $q\geq 2$ and $p\geq q$, it holds that

\begin{itemize}
  \item [\rm{(i)}] if $-\infty< a< b<\infty$, then
there exists a constant $C_{1}>0$ depending on $a,b,p,q,d$, $\gamma$, $\kappa$ and $\mu$
such that for any $f\in L^p((a,b)\times\mathbb{R}^{d}; V)$,
\begin{align}%\label{ineq: first parabolic LP-ineq}
&\int_{\mathbb{R}^{d}}\int_{a}^{b}\left(\int_{a}^{t}
\|(-\Delta)^{\frac{\gamma}{2q}}
\mathcal{T}_{\psi}(t,s)f(s,\cdot)(x)\|_{V}^{q}ds\right)^{\frac{p}{q}}dtdx\nonumber\\
&\qquad\leq C_{1}\int_{\mathbb{R}^{d}}\int_{a}^{b}
\|f(t,x)\|_{V}^{p}dtdx,
\end{align}

  \item [\rm{(ii)}] (if $q=2$,) there exists a constant $C_{2}>0$ depending on $p,q,d,\gamma,\kappa$ and $\mu$
such that for any $f\in L^p(\mathbb{R}^{d+1}; V)$,
\begin{align}\label{ineq: first parabolic LP-ineq}
&\int_{\mathbb{R}^{d}}\int_{-\infty}^{\infty}\left(\int_{-\infty}^{t}
\|(-\Delta)^{\frac{\gamma}{4}}
\mathcal{T}_{\psi}(t,s)f(s,\cdot)(x)\|_{V}^{2}ds\right)^{\frac{p}{2}}dtdx\nonumber\\
&\qquad\leq C_{2}\int_{\mathbb{R}^{d}}\int_{-\infty}^{\infty}
\|f(t,x)\|_{V}^{p}dtdx.
\end{align}
\end{itemize}
\end{corollary}

\begin{proof}
By taking $\psi_{1}(l,\xi)=-|\xi|^{\frac{\gamma}{q}}$ and $\psi_{2}(t,\xi)=\psi(t,\xi)$,
it is easy to see that the symbol $\psi_{1}$ satisfies the conditions \textbf{(S1)} and \textbf{(S3)}
with $\gamma_{1}=\frac{\gamma}{q}$ and $\kappa_{1}=1$ (see Example \ref{ex: norm xi power gamma}).
In this case, $L_{\psi_{1}}(l)=(-\Delta)^{\frac{\gamma}{2q}}$.
Therefore, the proof is immediate from Theorem \ref{thm:LP ineq}.
\end{proof}

\begin{remark}
\upshape
The inequality given in \eqref{ineq: first parabolic LP-ineq}
is called the parabolic Littlewood–Paley inequality
and proved in Theorem 3.1 of \cite{I. Kim K.-H. Kim 2016}
under the conditions \textbf{(S1)} and \textbf{(S2)} with $N=\lfloor\frac{d}{2}\rfloor+1$.
In fact, in \cite{I. Kim K.-H. Kim 2016},
the authors proved the inequality given in \eqref{ineq: first parabolic LP-ineq} using different arguments
than the one used to prove Lemmas 6.1, 6.2, and 6.3, which can be seen in Lemmas 2.9, 2.10, and 2.11
(see also Lemma 4.1 and Corollary 4.3) in \cite{I. Kim K.-H. Kim 2016}.
\end{remark}

\begin{corollary}
Let $\psi\in\mathfrak{S}_{\rm T}$ such that $N>d+2+2\lfloor\gamma\rfloor$.
Then for any $q\geq 2$ and $p\geq q$, it holds that
\begin{itemize}
  \item [\rm{(i)}] if $-\infty< a< b<\infty$, then
there exist constants $C_{1},C_{2}>0$ depending on $a,b,p,q,d,\gamma,\kappa$ and $\mu$
such that $f\in L^p((a,b)\times\mathbb{R}^{d}; V)$,
\begin{align}\label{ineq: LP-ineq dt evol.sys}
&\int_{\mathbb{R}^{d}}\int_{a}^{b}\left(\int_{a}^{t}(t-s)^{q-1}
        \left\|\frac{\partial}{\partial t}\mathcal{T}_{\psi}(t,s)f(s,\cdot)(x)\right\|_{V}^{q}ds\right)^{\frac{p}{q}}dtdx\nonumber\\
&\qquad\leq C_{1}\int_{\mathbb{R}^{d}}\int_{a}^{b}
\|f(t,x)\|_{V}^{p}dtdx.
\end{align}
and
\begin{align}\label{ineq: LP-ineq ds evol.sys}
&\int_{\mathbb{R}^{d}}\int_{a}^{b}\left(\int_{a}^{t}(t-s)^{q-1}
\left\|\frac{\partial}{\partial s}
\mathcal{T}_{\psi}(t,s)f(s,\cdot)(x)\right\|_{V}^{q}ds\right)^{\frac{p}{q}}dtdx\nonumber\\
&\qquad\leq C_{2}\int_{\mathbb{R}^{d}}\int_{a}^{b}
\|f(t,x)\|_{V}^{p}dtdx,
\end{align}
  \item [\rm{(ii)}] (if $q=2$,) there exists a constant $C_{3}>0$
  depending on $p,q,d,\gamma,\kappa$ and $\mu$ such that for any $f\in L^{p}(\mathbb{R}^{d+1};V)$,
\begin{align}\label{ineq: LP-ineq ds evol.sys q=2}
&\int_{\mathbb{R}^{d}}\int_{-\infty}^{\infty}\left(\int_{-\infty}^{t}(t-s)
\left\|\frac{\partial}{\partial s}
\mathcal{T}_{\psi}(t,s)f(s,\cdot)(x)\right\|_{V}^{2}ds\right)^{\frac{p}{2}}dtdx\nonumber\\
&\qquad\leq C_{3}\int_{\mathbb{R}^{d}}\int_{-\infty}^{\infty}
\|f(t,x)\|_{V}^{p}dtdx.
\end{align}
\end{itemize}
\end{corollary}

\begin{proof}
(i)\enspace Note that
\begin{align}
\frac{\partial}{\partial t}\mathcal{T}_{\psi}(t,s)f(x)&=L_{\psi}(t)\mathcal{T}_{\psi}(t,s)f(x),\label{eqn: dt evol. sys}\\
\frac{\partial}{\partial s}\mathcal{T}_{\psi}(t,s)f(x)&=-L_{\psi}(s)\mathcal{T}_{\psi}(t,s)f(x). \label{eqn: ds evol. sys}
\end{align}
Hence, by \eqref{eqn: dt evol. sys} and \eqref{ineq:LP-ineq l=t} with $\psi_{1}=\psi_{2}=\psi$, we get \eqref{ineq: LP-ineq dt evol.sys}.
Also, by \eqref{eqn: ds evol. sys} and \eqref{ineq:LP-ineq} with $\psi_{1}=\psi_{2}=\psi$
and $l=s$, we get \eqref{ineq: LP-ineq ds evol.sys}.

(ii) \enspace By \eqref{eqn: ds evol. sys} and \eqref{ineq:LP-ineq q=2} with $\psi_{1}=\psi_{2}=\psi$
and $l=s$, we get \eqref{ineq: LP-ineq ds evol.sys q=2}.
\end{proof}

For $\gamma>0$ and a suitable function $f$ defined on $\mathbb{R}^{d}$,
the fractional heat semigroup $\{T_{\gamma}(t)\}_{t\geq 0}$ is defined by
\begin{align}\label{eqn:frac.heat.semi}
T_{\gamma}(t)f(x)=p_{\gamma}(t,\cdot)*f(x),\quad p_{\gamma}(t,x)=\mathcal{F}^{-1}(e^{-t|\xi|^{\gamma}})(x).
\end{align}
\begin{corollary}
Let $\gamma>0$ and let $\{T_{\gamma}(t)\}_{t\geq 0}$
be the fractional heat semigroup given in \eqref{eqn:frac.heat.semi}.
Then for any $q\geq 2$ and $p\geq q$, it holds that

\begin{itemize}
  \item [\rm{(i)}] if $-\infty< a< b<\infty$, then
there exists a constant $C_{1}>0$ depending on $a,b,p,d,\gamma,\kappa$ and $\mu$
such that for any $f\in L^p((a,b)\times\mathbb{R}^{d}; V)$,
\begin{align}%\label{ineq: first parabolic LP-ineq}
&\int_{\mathbb{R}^{d}}\int_{a}^{b}\left(\int_{a}^{t}\|(-\Delta)^{\frac{\gamma}{2q}}
            T_{\gamma}(t-s)f(s,\cdot)(x)\|_{V}^{q}ds\right)^{\frac{p}{q}}dtdx\nonumber\\
&\qquad\leq C_{1}\int_{\mathbb{R}^{d}}\int_{a}^{b}
\|f(t,x)\|_{V}^{p}dtdx,
\end{align}

  \item [\rm{(ii)}] (if $q=2$,) there exists a constant $C_{2}>0$
  depending on $p,d,\gamma,\kappa$ and $\mu$ such that for any $f\in L^p(\mathbb{R}^{d+1}; V)$,
\begin{align}\label{ineq:LP-ineq frac.heat.semi}
&\int_{\mathbb{R}^{d}}\int_{-\infty}^{\infty}\left(\int_{-\infty}^{t}\|(-\Delta)^{\frac{\gamma}{4}}
            T_{\gamma}(t-s)f(s,\cdot)(x)\|_{V}^{2}ds\right)^{\frac{p}{2}}dtdx\nonumber\\
&\qquad\leq C_{2}\int_{\mathbb{R}^{d}}\int_{-\infty}^{\infty}
\|f(t,x)\|_{V}^{p}dtdx.
\end{align}
\end{itemize}
\end{corollary}

\begin{proof}
By taking $\psi(t,\xi)=-|\xi|^{\gamma}$,
it is easy to see that the symbol $\psi$ satisfy the conditions \textbf{(S1)} and \textbf{(S2)}
 (see Example \ref{ex: norm xi power gamma}).
In this case, we have
\begin{align*}
\mathcal{T}_{\psi}(t,s)f(s,x)=\mathcal{F}^{-1}\left(e^{-(t-s)|\xi|^{\gamma}}\mathcal{F}f(s,\xi)\right)(x)
=T_{\gamma}(t-s)f(s,x).
\end{align*}
Thus, by applying Corollary \ref{cor: LP Lpsi is Laplacian}, the proof is immediate.
\end{proof}

\begin{remark}
\upshape
The inequality \eqref{ineq:LP-ineq frac.heat.semi} with $\gamma=2$
is a consquence of Theorem 1.1 of \cite{Krylov 1994}.
On the other hand, the inequality in \eqref{ineq:LP-ineq frac.heat.semi} with $0<\gamma<2$
is a consquence of Theorem 2.3 of \cite{I. Kim 2012}.
\end{remark}

%%%%%%%%%%%%%%%%%%%%%%%%%%%%%%%%%%%%%%%%%%%%%
\appendix
\section{Embedding Theorem for Sobolev Spaces}\label{sec: Appendix B}
In this section, we prove that if $V$ is a separable Hilbert space, $\alpha\in\mathbb{R}$ and $p\geq 2$,
then the Sobolev space $H_{p}^{\alpha}(\mathbb{R}^{d};V)$ is continuously embedded into the Besov space $B_{pp}^{\alpha}(\mathbb{R}^{d};V)$.

The following lemma is well known as the L\'{e}vy-Khintchin inequality.

\begin{lemma}\label{lem: Khintchin's ineq}
Let $(\Omega,\mathcal{F},P)$ be a probability space.
Let $\{r_{i}\}_{i=1}^{\infty}$ be the Rademacher's system, i.e., $\{r_{i}\}_{i=1}^{\infty}$ is a sequence of i.i.d random variables on $\Omega$ with
\[
P(r_{i}=1)=P(r_{i}=-1)=\frac{1}{2},\quad i=1,2,\cdots.
\]
Let $\{a_{i}\}_{i=1}^{\infty}$ be a sequence of real numbers such that $\sum_{i=1}^{\infty}|a_{i}|^{2}<\infty$. For any $0<p<\infty$, there exist constants $C_{1},C_{2}>0$ depending on $p$ such that
\begin{align*}
C_{1}\left(\sum_{i=1}^{\infty}|a_{i}|^{2}\right)^{\frac{p}{2}}
\leq \mathbb{E}\left|\sum_{i=1}^{\infty}a_{i}r_{i}\right|^{p}
\leq C_{2}\left(\sum_{i=1}^{\infty}|a_{i}|^{2}\right)^{\frac{p}{2}},
\end{align*}
where $\mathbb{E}$ is the expectation with respect to the probability measure $P$.
\end{lemma}

\begin{proof}
See, e.g., Appendix D of \cite{Stein 1970}.
\end{proof}

\begin{theorem}\label{thm: Hps(V) embedded into Bpps(V)}
Let $V$ be a separable Hilbert space.
Let $\alpha\in\mathbb{R}$ and let $p\geq 2$. Then there exists a constant $C>0$ such that
\begin{align*}
\|f\|_{B_{pp}^{\alpha}(\mathbb{R}^{d};V)}\leq C\|f\|_{H_{p}^{\alpha}(\mathbb{R}^{d};V)}.
\end{align*}
That is, the Sobolev space $H_{p}^{\alpha}(\mathbb{R}^{d};V)$ is continuously embedded into the Besov space $B_{pp}^{\alpha}(\mathbb{R}^{d};V)$.
\end{theorem}

\begin{proof}
Let $\alpha\in\mathbb{R}$ and $p\geq 2$ be given.
Let $(\Omega,\mathcal{F},P)$ be a probability space and let $\{r_{i}\}_{i=1}^{\infty}$ be the Rademacher's system.
We denote by $\mathbb{E}$  the expectation with respect to the probability measure $P$.
Let $\{e_{i}\}_{i=1}^{\infty}$ be a complete orthonormal basis for $V$ and let $f\in H_{p}^{\alpha}(\mathbb{R}^d;V)$.
For each $n\in\mathbb{N}$, we put
\begin{align*}
f_{n}(x)=\sum_{i=1}^{n}\bilin{f(x)}{e_{i}}_{V}e_{i},\quad x\in\mathbb{R}^{d}.
\end{align*}
Note that if $h(x)=\sum_{i=1}^{n}h_{i}(x)e_{i}$, where $h_{i}$ is a real-valued function defined on $\mathbb{R}^{d}$,
by applying Lemma \ref{lem: Khintchin's ineq}, we obtain that
\begin{align}\label{eqn: Lp norm of Hilbert valued function}
\|h\|_{L^{p}(\mathbb{R}^{d};V)}^{p}
&=\int_{\mathbb{R}^{d}}\|h(x)\|_{V}^{p}dx
=\int_{\mathbb{R}^{d}}\left(\sum_{i=1}^{n}|h_{i}(x)|^{2}\right)^{\frac{p}{2}}dx\nonumber\\
&\leq C_{1}^{-1}\int_{\mathbb{R}^{d}}\mathbb{E}\left|\sum_{i=1}^{n}h_{i}(x)r_{i}\right|^{p}dx.
\end{align}
Let $\Delta_{j}$ and $S_{0}$ be the operators given as in \eqref{eqn: LP operator} and \eqref{eqn: S0}, respectively.
Then by applying \eqref{eqn: Lp norm of Hilbert valued function}
with the function $h=\sum_{i=1}^{n}S_{0}(\bilin{f(\cdot)}{e_{i}}_{V})e_{i}$, we obtain that
\begin{align}
\|S_{0}(f_{n})\|_{L^{p}(\mathbb{R}^{d};V)}^{p}
&=\left\|\sum_{i=1}^{n}S_{0}(\bilin{f(\cdot)}{e_{i}}_{V})e_{i}\right\|_{L^{p}(\mathbb{R}^{d};V)}^{p}\nonumber\\
&\leq C_{1}^{-1}\int_{\mathbb{R}^{d}}\mathbb{E}
        \left|\sum_{i=1}^{n}S_{0}\left(\bilin{f(\cdot)}{e_{i}}_{V}\right)(x)r_{i}\right|^{p}dx\nonumber\\
&=C_{1}^{-1}\int_{\mathbb{R}^{d}}\mathbb{E}
        \left|S_{0}\left(\sum_{i=1}^{n}\bilin{f(\cdot)}{e_{i}}_{V}r_{i}\right)(x)\right|^{p}dx\nonumber\\
&=C_{1}^{-1}\mathbb{E}\left\|S_{0}\left(\sum_{i=1}^{n}
        \bilin{f(\cdot)}{e_{i}}_{V}r_{i}\right)\right\|_{L^{p}(\mathbb{R}^{d})}^{p}
\label{eqn:Lp norm of S0}
\end{align}
and by applying \eqref{eqn: Lp norm of Hilbert valued function}
with the function $h=\sum_{i=1}^{n}\Delta_{j}\bilin{f(\cdot)}{e_{i}}e_{i}$, we obtain that
\begin{align}
\sum_{j=1}^{\infty}2^{j\alpha p}\|\Delta_{j}f_{n}\|_{L^{p}(\mathbb{R}^{d};V)}^{p}
&=\sum_{j=1}^{\infty}2^{j\alpha p}
       \left\|\sum_{i=1}^{n}\Delta_{j}\bilin{f(\cdot)}{e_{i}}e_{i}\right\|_{L^{p}(\mathbb{R}^{d};V)}^{p}\nonumber\\
&\leq C_{1}^{-1}\sum_{j=1}^{\infty}2^{j\alpha p}\int_{\mathbb{R}^{d}}
        \mathbb{E}\left|\sum_{i=1}^{n}\Delta_{j}\bilin{f(\cdot)}{e_{i}}_{V}(x)r_{i}\right|^{p}dx\nonumber\\
%
%&=C_{1}^{-1}\sum_{j=1}^{\infty}2^{j\alpha p}\int_{\mathbb{R}^{d}}
%         \mathbb{E}\left|\Delta_{j}\left(\sum_{i=1}^{n}\bilin{f(\cdot)}{e_{i}}_{V}r_{i}\right)(x)\right|^{p}dx\nonumber\\
%
&=\sum_{j=1}^{\infty}2^{j\alpha p}
         \mathbb{E}\left\|\Delta_{j}\left(\sum_{i=1}^{n}\bilin{f(\cdot)}{e_{i}}_{V}r_{i}\right)\right\|_{L^{p}(\mathbb{R}^{d})}^{p}.
\label{eqn:sum 2 alpha j delta j}
\end{align}
On the other hand, by the Jensen's inequality, we obtain that
\begin{align}
\|f_{n}\|_{B_{pp}^{\alpha}(\mathbb{R}^{d};V)}^{p}
&=\left(\|S_{0}(f_{n})\|_{L^{p}(\mathbb{R}^{d};V)}+\left(\sum_{j=1}^{\infty}2^{j\alpha p}
            \|\Delta_{j}f_{n}\|_{L^{p}(\mathbb{R}^{d};V)}^{p}\right)^{\frac{1}{p}}\right)^{p}\nonumber\\
&\leq 2^{p-1}\left(\|S_{0}(f_{n})\|_{L^{p}(\mathbb{R}^{d};V)}^{p}+\sum_{j=1}^{\infty}2^{j\alpha p}
            \|\Delta_{j}f_{n}\|_{L^{p}(\mathbb{R}^{d};V)}^{p}\right)\label{eqn: equiv norm of Besov sp norm 1}
\end{align}
and by the fact that $(a+b)^{\frac{1}{p}}\leq a^{\frac{1}{p}}+b^{\frac{1}{p}}$ for $a,b\geq 0$ and $p\geq 1$,
for any $g\in B_{pp}^{\alpha}(\mathbb{R}^{d};V)$, we obtain that
\begin{align}
&\left(\|S_{0}(g)\|_{L^{p}(\mathbb{R}^{d};V)}^{p}+\sum_{j=1}^{\infty}2^{j\alpha p}
        \|\Delta_{j}g\|_{L^{p}(\mathbb{R}^{d};V)}^{p}\right)^{\frac{1}{p}}\nonumber\\
&\leq\|S_{0}(g)\|_{L^{p}(\mathbb{R}^{d};V)}+\left(\sum_{j=1}^{\infty}2^{j\alpha p}
            \|\Delta_{j}g\|_{L^{p}(\mathbb{R}^{d};V)}^{p}\right)^{\frac{1}{p}}\nonumber\\
&=\|g\|_{B_{pp}^{\alpha}(\mathbb{R}^{d};V)}.\label{eqn: equiv norm of Besov sp norm 2}
\end{align}
Thus, by applying \eqref{eqn: equiv norm of Besov sp norm 1}, \eqref{eqn:Lp norm of S0},
\eqref{eqn:sum 2 alpha j delta j}, \eqref{eqn: equiv norm of Besov sp norm 2}
and the interpolation of Sobolev and Besov spaces (see Theorem 6.4.4 in \cite{Bergh 1976};
in fact, since $\alpha\in\mathbb{R}$ and $p\geq 2$, the Sobolev space $H_{p}^{\alpha}(\mathbb{R}^{d})$
is continuously embedded into the Besov space $B_{pp}^{\alpha}(\mathbb{R}^{d})$),
we obtain that
\begin{align}
\|f_{n}\|_{B_{pp}^{\alpha}(\mathbb{R}^{d};V)}^{p}
&\leq 2^{p-1}\left(\|S_{0}(f_{n})\|_{L^{p}(\mathbb{R}^{d};V)}^{p}
        +\sum_{j=1}^{\infty}2^{j\alpha p} \|\Delta_{j}f_{n}\|_{L^{p}(\mathbb{R}^{d};V)}^{p}\right)\nonumber\\
&\leq 2^{p-1}C_{1}^{-1}\mathbb{E}\left(\left\|S_{0}\left(\sum_{i=1}^{n}
            \bilin{f(\cdot)}{e_{i}}_{V}r_{i}\right)\right\|_{L^{p}(\mathbb{R}^{d})}^{p}\right.\nonumber\\
            &\left.\qquad+\sum_{j=1}^{\infty}2^{j\alpha p}\left\|\Delta_{j}\left(\sum_{i=1}^{n}
            \bilin{f(\cdot)}{e_{i}}_{V}r_{i}\right)\right\|_{L^{p}(\mathbb{R}^{d})}^{p}\right)\nonumber\\
&\leq 2^{p-1}C_{1}^{-1} \mathbb{E}\left\|\sum_{i=1}^{n}
        \bilin{f(\cdot)}{e_{i}}_{V}r_{i}\right\|_{B_{pp}^{\alpha}(\mathbb{R}^{d})}^{p}\nonumber\\
&\leq 2^{p-1}C_{1}^{-1}C_{3} \mathbb{E}\left\|\sum_{i=1}^{n}\bilin{f(\cdot)}{e_{i}}_{V}r_{i}\right\|_{H_{p}^{\alpha}(\mathbb{R}^{d})}^{p}\label{eqn:Besov norm of f}
\end{align}
for some $C_{3}>0$.
Therefore, by applying Lemma \ref{lem: Khintchin's ineq}, we obtain that
\begin{align*}
\|f_{n}\|_{B_{pp}^{\alpha}(\mathbb{R}^{d};V)}^{p}
&\leq 2^{p-1}C_{1}^{-1}C_{3}
        \mathbb{E}\left\|\sum_{i=1}^{n}\bilin{f(\cdot)}{e_{i}}_{V}r_{i}\right\|_{H_{p}^{\alpha}(\mathbb{R}^{d})}^{p}\\
&=2^{p-1}C_{1}^{-1}C_{3}
       \mathbb{E}\int_{\mathbb{R}^{d}}   \left|(1-\Delta)^{\frac{\alpha}{2}}\sum_{i=1}^{n}\bilin{f(x)}{e_{i}}_{V}r_{i}\right|^{p}dx\\
&=2^{p-1}C_{1}^{-1}C_{3}
       \int_{\mathbb{R}^{d}}\mathbb{E}   \left|\sum_{i=1}^{n}\bilin{(1-\Delta)^{\frac{\alpha}{2}}f(x)}{e_{i}}_{V}r_{i}\right|^{p}dx\\
&\leq 2^{p-1}C_{1}^{-1}C_{3}C_{2}
          \int_{\mathbb{R}^{d}} \left(\sum_{i=1}^{n}|\bilin{(1-\Delta)^{\frac{\alpha}{2}}f(x)}{e_{i}}_{V}|^{2}\right)^{\frac{p}{2}}dx\\
&=2^{p-1}C_{1}^{-1}C_{3}C_{2}
      \int_{\mathbb{R}^{d}}\|(1-\Delta)^{\frac{\alpha}{2}}f_{n}(x)\|_{V}^{p}dx\\
&=2^{p-1}C_{1}^{-1}C_{3}C_{2}\|f_{n}\|_{H_{p}^{\alpha}(\mathbb{R}^{d};V)}^{p}.
\end{align*}
By applying the completeness of $B_{pp}^{\alpha}(\mathbb{R}^{d};V)$ and $H_{p}^{\alpha}(\mathbb{R}^{d};V)$,
and taking limit as $n\rightarrow \infty$, the proof is complete.
\end{proof}
%%%%%%%%%%%%%%%%%%%%%%%%%%%%%%%%%%%%
\section{A Proof of Theorem \ref{thm: sharp fct less than maximal}}\label{sec: Appendix A}
%\addcontentsline{toc}{section}{Appendix}
%\renewcommand{\thesubsection}{\Alph{subsection}}
\counterwithin{theorem}{section}

\subsection{Estimation of Kernels}
In this subsection, we consider the pair $(\psi_{1},\psi_{2})\in \mathfrak{S}_{\rm T}\times\mathfrak{S}$ of symbols.
For $l\geq 0$, $t>s$ and $x\in\mathbb{R}^{d}$, we put
\begin{align*}
K_{\psi_{1},\psi_{2}}(l,t,s,x)&=L_{\psi_{1}}(l)p_{\psi_{2}}(t,s,x)
=\mathcal{F}^{-1}\left(\psi_{1}(l,\xi)\exp\left(\int_{s}^{t}\psi_{2}(r,\xi)dr\right)\right)(x).
\end{align*}
In fact, to prove our main theorem, the estimations of gradients of $K_{\psi_{1},\psi_{2}}$
is useful (see Lemma \ref{lem: bound of Dt mathcalG g-2}) given as in Lemma \ref{lem: esti of grad K},
and then in this subsection, we mainly will prove Lemma \ref{lem: esti of grad K}.

By the definitions of two classes $\mathfrak{S}_{\rm T}$ and $\mathfrak{S}$, for each $i=1,2$,
there exist positive constants $\kappa_{i}:=\kappa_{\psi_{i}}$, $\mu_{i}:=\mu_{\psi_{i}}$,
$\gamma_{i}:=\gamma_{\psi_{i}}$ and $N_{i}:=N_{\psi_{i}}$ with $N_{i}\ge \lfloor\frac{d}{2}\rfloor +1$ such that
\textbf{(S1)} and \textbf{(S3)} hold for $\psi=\psi_{1}$, and separately, \textbf{(S1)} and \textbf{(S2)} hold for $\psi=\psi_{2}$.
In this case, for notational convenience, we write
\begin{align*}
\boldsymbol{\kappa}:=(\kappa_1,\kappa_2),\quad
\boldsymbol{\mu}:=(\mu_{1},\mu_{2}),\quad
\boldsymbol{\gamma}:=(\gamma_{1},\gamma_{2}),\quad
\boldsymbol{N}:=(N_{1},N_{2}).
\end{align*}
In the left hand side of the inequality \eqref{ineq:LP-ineq}, the operator $Tf=L_{\psi_{1}}(l)\mathcal{T}_{\psi_{2}}(t,s)f(s,\cdot)(x)$
can be expressed as
\begin{align*}
Tf &=L_{\psi_{1}}(l)\mathcal{T}_{\psi_{2}}(t,s)f(s,\cdot)(x)
=L_{\psi_{1}}(l)\left(p_{\psi_{2}}(t,s,\cdot)*f(s,\cdot)\right)(x)\\
&=\left(\left(L_{\psi_{1}}(l)p_{\psi_{2}}(t,s,\cdot)\right)*f(s,\cdot)\right)(x).
\end{align*}
In order to prove Lemma \ref{lem: esti of grad K},
we need estimations for the gradient and time derivative of the kernel $L_{\psi_{1}}(l)p_{\psi_{2}}(t,s,x)$.

\begin{lemma}\label{lem: esti of partial psi 1 psi 2}
Let $\psi_{1}\in\mathfrak{S}_{\rm T}$ and let $(\psi_{2},\psi_{3})\in\mathfrak{S}^{2}$.
Let $\alpha\in\mathbb{N}_{0}^{d}$ be a multi-index with $|\alpha|\leq \min\{N_{1},N_{2},N_{3}\}$.
Let $m,n\in\{0,1\}$.
Then there exist constants $c,C>0$ depending on $\alpha,\,\kappa_{i}$, $\mu_{i}$ and $\gamma_{i}$ ($i=1,2,3$) such that
for any  $l\geq 0$ and $t>s$,
\begin{align}
\left|\partial_{\xi}^{\alpha}\left[\left(\partial_{t}^{m}\psi_{1}(t,\xi)\right)\left(\psi_{2}(t,\xi)\right)^{n}
            \exp\left(\int_{s}^{t}\psi_{3}(r,\xi)dr\right)\right]\right|
&\leq C|\xi|^{\gamma_{1}+n\gamma_{2}-|\alpha|}e^{-c(t-s)|\xi|^{\gamma_{3}}}.
    \label{eqn: F(3) esti-1}
\end{align}
\end{lemma}

\begin{proof}
We now prove \eqref{eqn: F(3) esti-1} only for the case $n=1$.
The proof for the case of $n=0$ is similar and simpler.
By direct computation
and the fact that $xe^{-x}\leq 1$ for any $x>0$,
for any multi index $\beta\in\mathbb{N}_{0}^{d}$ with $|\beta|\leq N_{3}$, we obtain that
\begin{align}
\left|\partial_{\xi}^{\beta}\exp\left(\int_{s}^{t}\psi_{3}(r,\xi)dr\right)\right|
&\leq C(t-s)|\xi|^{\gamma_{3}-|\beta|}e^{-c_{1}(t-s)|\xi|^{\gamma_{3}}}\nonumber\\
&=C'|\xi|^{-|\beta|}e^{-c(t-s)|\xi|^{\gamma_{3}}}\label{eqn:psi 2 ineq-2}
\end{align}
for some constant $C,C',c_{1},c>0$.
Also, by applying the Leibniz rule, \textbf{(S2)} and \textbf{(S3)}, we obtain that
\begin{align}
 \left|\partial_{\xi}^{\alpha-\beta}\left[\left(\partial_{t}^{m}\psi_{1}(t,\xi)\right)\psi_{2}(t,\xi)\right]\right|
&\leq \sum_{\delta\leq \alpha-\beta}\binom{\alpha-\beta}{\beta}      \left|\partial_{\xi}^{\delta}\partial_{t}^{m}\psi_{1}(t,\xi)\right|\left|\partial_{\xi}^{\delta}\psi_{2}(t,\xi)\right|\nonumber\\
&\leq \sum_{\delta\leq \alpha-\beta}\binom{\alpha-\beta}{\beta}
                \mu_{1}|\xi|^{\gamma_{1}-|\delta|}\mu_{2}|\xi|^{\gamma_{2}-(|\alpha-\beta|-|\delta|)}\nonumber\\
&=C'' |\xi|^{\gamma_{1}+\gamma_{2}-|\alpha-\beta|}\label{eqn: partial t m psi1 psi2}
\end{align}
for some constant $C''>0$.
By applying the Leibniz rule, \eqref{eqn:psi 2 ineq-2} and \eqref{eqn: partial t m psi1 psi2}, we obtain that
\begin{align*}
&\left|\partial_{\xi}^{\alpha}\left[\left(\partial_{t}^{m}\psi_{1}(t,\xi)\right)\psi_{2}(t,\xi)
                \exp\left(\int_{s}^{t}\psi_{3}(r,\xi)dr\right)\right]\right|\\
&\leq \sum_{\beta\leq \alpha}\binom{\alpha}{\beta}
            \left|\partial_{\xi}^{\beta}\exp\left(\int_{s}^{t}\psi_{2}(r,\xi)dr\right)\right|
                \left|\partial_{\xi}^{\alpha-\beta}\left[\left(\partial_{t}^{m}\psi_{1}(t,\xi)\right)\psi_{2}(t,\xi)\right]\right|\\
&\leq \sum_{\beta\leq \alpha}\binom{\alpha}{\beta}
        C'|\xi|^{-|\beta|}e^{-c(t-s)|\xi|^{\gamma_{3}}}
                C'' |\xi|^{\gamma_{1}+\gamma_{2}-|\alpha-\beta|},
\end{align*}
which implies the desired result.
\end{proof}

For $l\geq 0$, $t>s$ and $x\in\mathbb{R}^{d}$, we put
\begin{align*}
F_{\psi_{1},\psi_{2}}^{(1)}(l,t,s,x)
 &=\sum_{i=1}^{d}\left|\mathcal{F}^{-1}\left(\xi_{i}\psi_{1}(l,\xi)\exp\left(\int_{s}^{t}\psi_{2}(r,\xi)dr\right)\right)(x)\right|,\\
F_{\psi_{1},\psi_{2}}^{(2)}(l,t,s,x)
 &=\sum_{i,j=1}^{d}\left|\mathcal{F}^{-1}\left(\xi_{i}\xi_{j}\psi_{1}(l,\xi)
            \exp\left(\int_{s}^{t}\psi_{2}(r,\xi)dr\right)\right)(x)\right|,\\
F_{\psi_{1},\psi_{2}}^{(3)}(l,t,s,x)
 &=\sum_{i=1}^{d}\left|\mathcal{F}^{-1}\left(\xi_{i}
                        \psi_{1}(l,\xi)\psi_{2}(t,\xi)\exp\left(\int_{s}^{t}\psi_{2}(r,\xi)dr\right)\right)(x)\right|,\\
F_{\psi_{1},\psi_{2}}^{(4)}(t,t,s,x)
 &=\sum_{i=1}^{d}\left|\mathcal{F}^{-1}\left(\xi_{i}(\partial_{t}\psi_{1}(t,\xi)
                        +\psi_{1}(t,\xi)\psi_{2}(t,\xi))\right.\right.\\
 &\qquad\quad\left.\left.\times\exp\left(\int_{s}^{t}\psi_{2}(r,\xi)dr\right)\right)(x)\right|.
\end{align*}

\begin{lemma}\label{lem: esti gradient kernel 1}
Let $(\psi_{1},\psi_{2})\in \mathfrak{S}_{\rm T}\times\mathfrak{S}$ and let $N=\min\{N_{1},N_{2}\}$.
Assume that $N> d+2+\lfloor\gamma_{1}\rfloor+\lfloor\gamma_{2}\rfloor$.
Then there exist constants $C_{1},C_{2},C_{3},C_{4}>0$ depending on $\boldsymbol{\kappa},\boldsymbol{\mu},d,N$ such that
for any  $l\geq 0$ and $t>s$,
\begin{align}
F_{\psi_{1},\psi_{2}}^{(1)}(l,t,s,x)&\leq C_{1}|x|^{-(\gamma_{1}+1+d)},\label{eqn: esti of F (1)}\\
F_{\psi_{1},\psi_{2}}^{(2)}(l,t,s,x)&\leq C_{2}|x|^{-(\gamma_{1}+2+d)},\label{eqn: esti of F (2)}\\
F_{\psi_{1},\psi_{2}}^{(3)}(l,t,s,x)&\leq C_{3}|x|^{-(\gamma_{1}+\gamma_{2}+1+d)},\label{eqn: esti of F (3)}\\
F_{\psi_{1},\psi_{2}}^{(4)}(t,t,s,x)&\leq C_{4}(|x|^{-(\gamma_{1}+1+d)}+|x|^{-(\gamma_{1}+\gamma_{2}+1+d)})
\label{eqn: esti of F (4)}
\end{align}
for all $x\in\mathbb{R}^{d}\setminus \{0\}$.
\end{lemma}

\begin{proof}
The inequality \eqref{eqn: esti of F (1)} is a consequence of Lemma 3.2 in \cite{Ji-Kim 2025}.
Since the proofs of \eqref{eqn: esti of F (2)}, \eqref{eqn: esti of F (3)} and \eqref{eqn: esti of F (4)} are similar,
we only prove \eqref{eqn: esti of F (4)}.
We use the arguments used in the proof of Lemma 3.2 in \cite{Ji-Kim 2025} (see, also \cite{Miao 2008}).
We take a function $\rho\in C_{c}^{\infty}(\mathbb{R}^{d})$ such that
\begin{align*}
\rho(\xi)=\left\{
            \begin{array}{ll}
              1, & \hbox{$|\xi|\leq 1$,} \\
              0, & \hbox{$|\xi|> 2$,}
            \end{array}
          \right.\qquad \xi\in\mathbb{R}^{d}.
\end{align*}
Let $\lambda>0$ be given. Note that we have
\begin{equation}\label{eqn: rho lambda}
\rho\left(\frac{\xi}{\lambda}\right)=\left\{
            \begin{array}{ll}
              1, & \hbox{$|\xi|\leq \lambda$,} \\
              0, & \hbox{$|\xi|> 2\lambda$,}
            \end{array}
          \right.\qquad
1-\rho\left(\frac{\xi}{\lambda}\right)=\left\{
            \begin{array}{ll}
              0, & \hbox{$|\xi|\leq \lambda$,} \\
              1, & \hbox{$|\xi|> 2\lambda$.}
            \end{array}
          \right.
\end{equation}
For each $j=1,2,\cdots,d$, we
put
\begin{align*}
G_{j}(\xi)
=\xi_{j}\left(\partial_{t}\psi_{1}(t,\xi)+\psi_{1}(t,\xi)\psi_{2}(t,\xi)\right)\exp\left(\int_{s}^{t}\psi_{2}(r,\xi)dr\right).
\end{align*}
Let $x\in\mathbb{R}^{d}\setminus \{0\}$ be given.
Then by applying \eqref{eqn: rho lambda}, we have
\begin{align}
\mathcal{F}^{-1}\left(G_{j}(\xi)\right)(x)%\nonumber\\
&=\int_{\mathbb{R}^d}e^{ix\cdot\xi}G_{j}(\xi)\rho\left(\frac{\xi}{\lambda}\right)d\xi
     +\int_{\mathbb{R}^d}e^{ix\cdot\xi}G_{j}(\xi)\left(1-\rho\left(\frac{\xi}{\lambda}\right)\right)d\xi\nonumber\\
&=I_{1}+I_{2},\label{eqn: I1 I2}
\end{align}
where
\begin{align*}
I_{1}=\int_{|\xi|\leq 2\lambda}e^{ix\cdot\xi}G_{j}(\xi)\rho\left(\frac{\xi}{\lambda}\right)d\xi,
\quad
I_{2}=\int_{|\xi|>\lambda}e^{ix\cdot\xi}G_{j}(\xi)\left(1-\rho\left(\frac{\xi}{\lambda}\right)\right)d\xi.
\end{align*}
First we estimate $I_{1}$. Since $|\rho(\xi)|\leq M$ for some constant $M>0$,
by the conditions \textbf{(S1)}, \textbf{(S2)} and \textbf{(S3)},
and applying the change of variables to polar coordinates,
we obtain that
\begin{align}
|I_{1}|
&\leq M\int_{|\xi|\leq 2\lambda}|\partial_{t}\psi_{1}(t,\xi)+\psi_{1}(t,\xi)\psi_{2}(t,\xi)|
                                         |\xi_{i}|\left|\exp\left(\int_{s}^{t}\psi_{2}(r,\xi)dr\right)\right|d\xi\nonumber\\
&\leq M\int_{|\xi|\leq 2\lambda}\left(\mu_{1}|\xi|^{\gamma_{1}}
                +\mu_{1}\mu_{2}|\xi|^{\gamma_{1}+\gamma_{2}}\right)|\xi|d\xi\nonumber\\
&=M\mu_{1}\int_{|\xi|\leq 2\lambda}|\xi|^{\gamma_{1}+1}d\xi
     +M\mu_{1}\mu_{2}\int_{|\xi|\leq 2\lambda}|\xi|^{\gamma_{1}+\gamma_{2}+1}d\xi\nonumber\\
&=M\mu_{1}\frac{2\pi^{\frac{d}{2}}}{\Gamma(\frac{d}{2})}\int_{0}^{2\lambda}r^{\gamma_{1}+1+d-1}dr
     +M\mu_{1}\mu_{2}\frac{2\pi^{\frac{d}{2}}}{\Gamma(\frac{d}{2})}
        \int_{0}^{2\lambda}r^{\gamma_{1}+\gamma_{2}+1+d-1}dr\nonumber\\
&=M\mu_{1}\frac{2\pi^{\frac{d}{2}}}{\Gamma(\frac{d}{2})}(2\lambda)^{\gamma_{1}+1+d}
     +M\mu_{1}\mu_{2}\frac{2\pi^{\frac{d}{2}}}{\Gamma(\frac{d}{2})}
        (2\lambda)^{\gamma_{1}+\gamma_{2}+1+d}.\label{eqn:esti I1}
\end{align}
Now we estimate $I_{2}$. For a differentiable function $f$ on $\mathbb{R}^{d}$, define
\[
\mathcal{D}f(\xi)=\frac{x\cdot\nabla_{\xi}f(\xi)}{i|x|^2},\qquad \xi\in\mathbb{R}^{d}.
\]
Then it is obvious that $\mathcal{D}e^{ix\cdot\xi}=e^{i x\cdot \xi}$.
Therefore, we obtain that
\begin{align}
I_{2}&=\int_{|\xi|>\lambda}e^{ix\cdot\xi}G_{j}(\xi)\left(1-\rho\left(\frac{\xi}{\lambda}\right)\right)d\xi
=\int_{|\xi|>\lambda}\left(\mathcal{D}^{N}e^{ix\cdot\xi}\right)G_{j}(\xi)\left(1-\rho\left(\frac{\xi}{\lambda}\right)\right)d\xi\nonumber\\
&=\int_{|\xi|>\lambda}e^{ix\cdot\xi}
     \left[(\mathcal{D}^*)^{N}\left(G_{j}(\xi)\left(1-\rho\left(\frac{\xi}{\lambda}\right)\right)\right)\right]d\xi\nonumber\\
&=I_{21}+I_{22},\label{eqn: decomp of I2}
\end{align}
where $\mathcal{D}^{*}=-\frac{x\cdot\nabla_{\xi}}{i|x|^2}$ and
\begin{align*}
I_{21}
=\int_{\lambda< |\xi|\leq 2\lambda}
    e^{ix\cdot\xi}(\mathcal{D}^{*})^{N}\left(G_{j}(\xi)\left(1-\rho\left(\frac{\xi}{\lambda}\right)\right)\right)d\xi,
~~
I_{22}=\int_{|\xi|> 2\lambda}e^{ix\cdot\xi}(\mathcal{D}^{*})^{N}G_{j}(\xi)d\xi.
\end{align*}
We estimate $I_{21}$. Note that
\[
(\mathcal{D}^{*})^{N}
=\frac{(-1)^{N}}{i^{N}|x|^{2N}}\sum_{|\alpha|=N}x^{\alpha}\partial_{\xi}^{\alpha},
\]
where $x^{\alpha}=x_{1}^{\alpha_{1}}\cdots x_{d}^{\alpha_{d}}$ for $x=(x_{1},\cdots,x_{d})\in\mathbb{R}^{d}$
and a multi-index $\alpha=(\alpha_{1},\cdots,\alpha_{d})\in\mathbb{N}_{0}^{d}$.
Therefore, by applying the Leibniz rule, we obtain that
\begin{align*}
(\mathcal{D}^{*})^{N}\left(G_{j}(\xi)\left(1-\rho\left(\frac{\xi}{\lambda}\right)\right)\right)
&=\frac{(-1)^{N}}{i^{N}|x|^{2N}}\sum_{|\alpha|=N}x^{\alpha}\partial_{\xi}^{\alpha}
    \left(G_{j}(\xi)\left(1-\rho\left(\frac{\xi}{\lambda}\right)\right)\right)\\
&=\frac{(-1)^{N}}{i^{N}|x|^{2N}}\sum_{|\alpha|=N}x^{\alpha}\sum_{\beta\leq \alpha}\binom{\alpha}{\beta}\\
&\qquad\times    \left(\partial_{\xi}^{\alpha-\beta}G_{j}(\xi)\right)
       \left(\partial_{\xi}^{\beta}\left(1-\rho\left(\frac{\xi}{\lambda}\right)\right)\right).
\end{align*}
We estimate $\partial_{\xi}^{\alpha-\beta}G_{j}(\xi)$.
Put $G_{j}(\xi)=G_{j1}(\xi)+G_{j2}(\xi)$, where
\begin{align*}
G_{j1}(\xi)&=\xi_{j}\partial_{t}\psi_{1}(t,\xi)\exp\left(\int_{s}^{t}\psi_{2}(r,\xi)dr\right),\\
G_{j2}(\xi)&=\xi_{j}\psi_{1}(t,\xi)\psi_{2}(t,\xi)\exp\left(\int_{s}^{t}\psi_{2}(r,\xi)dr\right).
\end{align*}
By applying Lemma \ref{lem: esti of partial psi 1 psi 2} to $G_{j1}(\xi)$ and $G_{j2}(\xi)$ with $m=1, n=0$, $\psi_{3}=\psi_{2}$,
and $m=0, n=1$, $\psi_{3}=\psi_{2}$, respectively,
we obtain that for any multi index $\delta\in \mathbb{N}_{0}^{d}$ with $|\delta|\leq \min\{N_{1},N_{2}\}$,
\begin{align*}
|\partial_{\xi}^{\delta}G_{j1}(\xi)|
&\leq C|\xi|^{\gamma_{1}+1-|\delta|}e^{-c(t-s)|\xi|^{\gamma_{2}}},\\
|\partial_{\xi}^{\delta}G_{j2}(\xi)|
&\leq C'|\xi|^{\gamma_{1}+\gamma_{2}+1-|\delta|}e^{-c(t-s)|\xi|^{\gamma_{2}}}
\end{align*}
for some constant $c,C,C'>0$. Thus we obtain that
\begin{align}
|\partial_{\xi}^{\delta}G_{j}(\xi)|
&\leq |\partial_{\xi}^{\delta}G_{j1}(\xi)|+|\partial_{\xi}^{\delta}G_{j2}(\xi)|\nonumber\\
&\leq C''\left(|\xi|^{\gamma_{1}+1-|\delta|}+|\xi|^{\gamma_{1}+\gamma_{2}+1-|\delta|}\right)
        e^{-c(t-s)|\xi|^{\gamma_{2}}}\label{eqn: esti partial Gj}
\end{align}
for some constant $C''>0$.
Therefore, by applying \eqref{eqn: esti partial Gj} with $\delta=\alpha-\beta$ and the fact that $e^{-c(t-s)|\xi|^{\gamma_{2}}}\leq 1$,
we have
\begin{align}
|\partial_{\xi}^{\alpha-\beta}G_{j}(\xi)|
&\leq C_{\alpha-\beta}^{(1)}\left(|\xi|^{\gamma_{1}+1-|\alpha-\beta|}
            +|\xi|^{\gamma_{1}+\gamma_{2}+1-|\alpha-\beta|}\right)
            \label{eqn:esti partial Gj}
\end{align}
for some constant $C_{\alpha-\beta}^{(1)}>0$.
On the other hand, since $\lambda^{-|\beta|}\leq 2^{|\beta|}|\xi|^{-|\beta|}$ for $|\xi|\leq 2\lambda$, we have
\begin{equation}\label{eqn: partial 1-rho}
\left|\partial_{\xi}^{\beta}\left(1-\rho\left(\frac{\xi}{\lambda}\right)\right)\right|
\leq C_{\beta}^{(2)}\lambda^{-|\beta|}
\le C_{\beta}^{(2)}2^{|\beta|}|\xi|^{-|\beta|}
\end{equation}
for some constant $C_{\beta}^{(2)}>0$.
Then by applying  \eqref{eqn:esti partial Gj} and \eqref{eqn: partial 1-rho}, we obtain that
\begin{align}
&\left|(\mathcal{D}^{*})^{N}\left(G_{j}(\xi)\left(1-\rho(\frac{\xi}{\lambda})\right)\right)\right|\nonumber\\
&\leq \frac{1}{|x|^{2N}}\sum_{|\alpha|=N}|x^{\alpha}|
  \sum_{\beta\leq \alpha}\binom{\alpha}{\beta}\left|\partial_{\xi}^{\alpha-\beta}G_{j}(\xi)\right|
      \left|\partial_{\xi}^{\beta}\left(1-\rho(\frac{\xi}{\delta})\right)\right|\nonumber\\
&\leq\frac{1}{|x|^{N}}\sum_{|\alpha|=N}
  \sum_{\beta\leq \alpha}\binom{\alpha}{\beta}C_{\alpha-\beta}^{(1)}C_{\beta}^{(2)}2^{|\beta|}
\left(|\xi|^{\gamma_{1}+1-|\alpha-\beta|}+|\xi|^{\gamma_{1}+\gamma_{2}+1-|\alpha-\beta|}\right)
|\xi|^{-|\beta|}\nonumber\\
&=C^{(3)}\frac{1}{|x|^{N}}
\left(|\xi|^{\gamma_{1}+1-N}+|\xi|^{\gamma_{1}+\gamma_{2}+1-N}\right)
  \label{eqn: esti DGj rho}
\end{align}
for some $C^{(3)}>0$.
By applying \eqref{eqn: esti DGj rho}, the changing variables to the polar coordinates,
the assumption $N>d+2+\lfloor\gamma_{1}\rfloor+\lfloor\gamma_{2}\rfloor$, we obtain that
\begin{align*}
|I_{21}|
&\leq \int_{\lambda\leq |\xi|\leq 2\lambda}
           \left|(\mathcal{D}^{*})^{N}\left(G_{j}(\xi)\left(1-\rho(\frac{\xi}{\delta})\right)\right)\right|d\xi\\
&\le C^{(3)}\frac{1}{|x|^{N}}\int_{\lambda\leq |\xi|\leq 2\lambda}\left(|\xi|^{\gamma_{1}+1-N}+|\xi|^{\gamma_{1}+\gamma_{2}+1-N}\right)d\xi\\
&=C^{(3)}\frac{1}{|x|^{N}}\frac{2\pi^{\frac{d}{2}}}{\Gamma(\frac{d}{2})}
\int_{\lambda}^{2\lambda}\left(r^{\gamma_{1}+1-N}+r^{\gamma_{1}+\gamma_{2}+1-N}\right)r^{d-1}dr\\
&=C^{(4)}\frac{1}{|x|^{N}}\left(\lambda^{\gamma_{1}+1-N+d}+\lambda^{\gamma_{1}+\gamma_{2}+1-N+d}\right)
\end{align*}
for some $C^{(4)}>0$.

We now estimate $I_{22}$. By applying \eqref{eqn: esti partial Gj} with $\delta=\alpha$ and
the fact that $e^{-c(t-s)|\xi|^{\gamma_{2}}}\leq 1$, we obtain that
\begin{align*}
\left|((\mathcal{D}^{*})^{N}G_{j})(\xi)\right|
&\leq \frac{1}{|x|^{2N}}\sum_{|\alpha|=N}|x^{\alpha}
|\left|\partial_{\xi}^{\alpha}G_{j}(\xi)\right|\nonumber\\
&\leq \frac{1}{|x|^{2N}}\sum_{|\alpha|=N}|x|^{|\alpha|}C_{\alpha}^{(5)}
         \left(|\xi|^{\gamma_{1}+1-|\alpha|}+|\xi|^{\gamma_{1}+\gamma_{2}+1-|\alpha|}\right)
         e^{-c(t-s)|\xi|^{\gamma_{2}}}\nonumber\\
&=C^{(6)}\frac{1}{|x|^{N}}\left(|\xi|^{\gamma_{1}+1-N}+|\xi|^{\gamma_{1}+\gamma_{2}+1-N}\right)
\end{align*}
for some $C^{(6)}>0$.
By the assumption $N>d+2+\lfloor\gamma_{1}\rfloor+\lfloor\gamma_{2}\rfloor$,
we obtain that
\begin{align*}
|I_{22}|&\leq \int_{|\xi|\geq 2\lambda}\left|(\mathcal{D}^{*})^{N}\left(G_{j}(\xi)\right)\right|d\xi
\leq C^{(6)}\frac{1}{|x|^{N}}\int_{|\xi|\geq 2\lambda}
        \left(|\xi|^{\gamma_{1}+1-N}+|\xi|^{\gamma_{1}+\gamma_{2}+1-N}\right)d\xi\\
&\leq C^{(6)}\frac{1}{|x|^{N}}\frac{2\pi^{\frac{d}{2}}}{\Gamma(\frac{d}{2})}
        \int_{2\lambda}^{\infty}\left(r^{\gamma_{1}+1-N}+r^{\gamma_{1}+\gamma_{2}+1-N}\right)r^{d-1}dr\\
&=C^{(7)}\frac{1}{|x|^{N}}\left(\lambda^{\gamma_{1}+1-N+d}
                    +\lambda^{\gamma_{1}+\gamma_{2}+1-N+d}\right)
\end{align*}
for some $C^{(7)}>0$.
Therefore, by \eqref{eqn: I1 I2}, \eqref{eqn:esti I1} and \eqref{eqn: decomp of I2}, we have for any $j=1,2,\cdots,d,$
\begin{align}
|\mathcal{F}^{-1}(G_{j}(\xi))(x)|
&\leq C^{(8)}\left(\lambda^{\gamma_{1}+1+d}+\lambda^{\gamma_{1}+\gamma_{2}+1+d}\right)\nonumber\\
&\qquad+C^{(9)}\frac{1}{|x|^{N}}\left(\lambda^{\gamma_{1}+1-N+d}
                +\lambda^{\gamma_{1}+\gamma_{2}+1-N+d}\right)\label{eqn: IFT of Gi less lambda}
\end{align}
for some $C^{(8)},C^{(9)}>0$.
By taking $\lambda=|x|^{-1}$ in \eqref{eqn: IFT of Gi less lambda}, we have
\begin{align*}
F_{\psi_{1},\psi_{2}}^{(4)}(t,t,s,x)
&=\sum_{j=1}^{d}|\mathcal{F}^{-1}(G_{j}(\xi))(x)|
\leq C\left(|x|^{-(\gamma_{1}+1+d)}+|x|^{-(\gamma_{1}+\gamma_{2}+1+d)}\right)
\end{align*}
for some $C>0$.
The proof is complete.
\end{proof}

For a real-valued differentiable function $f$ on $\mathbb{R}^{d}$, we denote the gradient of $f$  by $\nabla f$.
Also, we denote the matrix $\left[\frac{\partial^2 f}{\partial x_{i}\partial x_{j}}\right]_{i,j=1}^{d}$ by $\nabla(\nabla f)$.

\begin{lemma}\label{lem: esti of grad K}
Let $(\psi_{1},\psi_{2})\in \mathfrak{S}_{\rm T}\times\mathfrak{S}$ and let $N=\min\{N_{1},N_{2}\}$.
Assume that $N> d+2+\lfloor\gamma_{1}\rfloor+\lfloor\gamma_{2}\rfloor$.
 Then there exist constants $C_{1},C_{2},C_{3},C_{4}>0$ depending on $d, \boldsymbol{\mu},\boldsymbol{\gamma}$ and $\boldsymbol{\kappa}$ such that
for any $l\geq 0$ and $t>s$,
\begin{align}
|\nabla K_{\psi_{1},\psi_{2}}(l,t,s,x)|
&\leq C_{1}\left(|x|^{-(\gamma_{1}+1+d)}\wedge (t-s)^{-\frac{\gamma_{1}+1+d}{\gamma_{2}}}\right) \label{eqn: bound of nabla K},\\
|\nabla(\nabla K_{\psi_{1},\psi_{2}})(l,t,s,x)|
&\leq C_{2}\left(|x|^{-(\gamma_{1}+2+d)}\wedge (t-s)^{-\frac{\gamma_{1}+2+d}{\gamma_{2}}}\right),\label{eqn: bound of nabla nabla K}\\
\left|\partial_{t}\nabla K_{\psi_{1},\psi_{2}}(l,t,s,x)\right|
&\leq C_{3}\left(|x|^{-(\gamma_{1}+\gamma_{2}+1+d)}\wedge
                      (t-s)^{-\frac{\gamma_{1}+\gamma_{2}+1+d}{\gamma_{2}}}\right),\label{eqn: bound of dt nabla K}\\
\left|\partial_{t}\nabla K_{\psi_{1},\psi_{2}}(t,t,s,x)\right|
&\leq C_{4}\left((|x|^{-(\gamma_{1}+1+d)}+|x|^{-(\gamma_{1}+\gamma_{2}+1+d)})\right.\nonumber\\
&\qquad\left.\wedge \left((t-s)^{-\frac{\gamma_{1}+1+d}{\gamma_{2}}}
                      +(t-s)^{-\frac{\gamma_{1}+\gamma_{2}+1+d}{\gamma_{2}}}\right)\right)
\label{eqn: bound of dt nabla K l=t}
\end{align}
for all $x\in\mathbb{R}^{d}\setminus \{0\}$,
where the absolute value of the left hand side means the Euclidean norm.
\end{lemma}

\begin{proof}
The inequality \eqref{eqn: bound of nabla K} is a consequence of Lemma 3.3 in \cite{Ji-Kim 2025}.
Since the proofs of \eqref{eqn: bound of nabla nabla K}, \eqref{eqn: bound of dt nabla K} and \eqref{eqn: bound of dt nabla K l=t}
are similar, we only prove \eqref{eqn: bound of dt nabla K l=t}.
For $i=1,2,\cdots,d$, it holds that
\begin{align*}
&\partial_{t}\partial_{x_{i}}K_{\psi_{1},\psi_{2}}(t,t,s,x)\\
&=\partial_{t}\mathcal{F}^{-1}\left(\xi_{i}\psi_{1}(t,\xi)  \exp\left(\int_{s}^{t}\psi_{2}(r,\xi)dr\right)\right)(x)\\
&=\mathcal{F}^{-1}\left(\xi_{i}\left(\partial_{t}\psi_{1}(t,\xi)+\psi_{1}(t,\xi)\psi_{2}(t,\xi)\right)
     \exp\left(\int_{s}^{t}\psi_{2}(r,\xi)dr\right)\right)(x).
\end{align*}
Since $(a+b)^{\frac{1}{p}}\leq a^{\frac{1}{p}}+b^{\frac{1}{p}}$ for $p\geq 1$,
by applying \eqref{eqn: esti of F (4)}, we obtain that
\begin{align}
\left|\partial_{t}\nabla K_{\psi_{1},\psi_{2}}(t,t,s,x)\right|
&=\left(\sum_{i=1}^{d}\left|\partial_{t}\partial_{x_{i}}
        K_{\psi_{1},\psi_{2}}(t,t,s,x)\right|^{2}\right)^{\frac{1}{2}}\nonumber\\
&\leq \sum_{i=1}^{d}\left|\partial_{t}\partial_{x_{i}}K_{\psi_{1},\psi_{2}}(t,t,s,x)\right|\nonumber\\
&=F_{\psi_{1},\psi_{2}}^{(4)}(t,t,s,x)\nonumber\\
&\leq C^{(1)}(|x|^{-(\gamma_{1}+1+d)}+|x|^{-(\gamma_{1}+\gamma_{2}+1+d)})\label{eqn: esti nabla K 1}
\end{align}
for some constant $C^{(1)}>0$.
On the other hand, by the conditions \textbf{(S1)}, \textbf{(S2)} and \textbf{(S3)}, we obtain that for any $x\in\mathbb{R}^{d}$
\begin{align}
&F_{\psi_{1},\psi_{2}}^{(4)}(t,t,s,x)\nonumber\\
&=\sum_{i=1}^{d}\left|\mathcal{F}^{-1}\left(\xi_{i}\left(\partial_{t}\psi_{1}(t,\xi)
    +\psi_{1}(t,\xi)\psi_{2}(t,\xi)\right)   \exp\left(\int_{s}^{t}\psi_{2}(r,\xi)dr\right)\right)(x)\right|\nonumber\\
&\leq \sum_{i=1}^{d}\int_{\mathbb{R}^{d}}|\xi_{i}|\left|\partial_{t}\psi_{1}(t,\xi)
    +\psi_{1}(t,\xi)\psi_{2}(t,\xi)\right|  \left|\exp\left(\int_{s}^{t}\psi_{2}(r,\xi)dr\right)\right|d\xi\nonumber\\
&\leq \sum_{i=1}^{d}\int_{\mathbb{R}^{d}} |\xi|  \left(\mu_{1}|\xi|^{\gamma_{1}}
+\mu_{1}\mu_{2}|\xi|^{\gamma_{1}+\gamma_{2}}\right)   e^{-\kappa_{2}(t-s)|\xi|^{\gamma_{2}}}d\xi\nonumber\\
&=\mu_{1}d\int_{\mathbb{R}^{d}} |\xi|^{\gamma_{1}+1}e^{-\kappa_{2}(t-s)|\xi|^{\gamma_{2}}}d\xi
+\mu_{1}\mu_{2}d\int_{\mathbb{R}^{d}}|\xi|^{\gamma_{1}+\gamma_{2}+1}
                                      e^{-\kappa_{2}(t-s)|\xi|^{\gamma_{2}}}d\xi\nonumber\\
&=\mu_{1}d\int_{\mathbb{R}^{d}}  |(t-s)^{-\frac{1}{\gamma_{2}}}y|^{\gamma_{1}+1}
                                        e^{-\kappa_{2}|y|^{\gamma_{2}}}(t-s)^{-\frac{d}{\gamma_{2}}}dy\nonumber\\
&\qquad+\mu_{1}\mu_{2}d\int_{\mathbb{R}^{d}}  |(t-s)^{-\frac{1}{\gamma_{2}}}y|^{\gamma_{1}+\gamma_{2}+1}
                                      e^{-\kappa_{2}|y|^{\gamma_{2}}}(t-s)^{-\frac{d}{\gamma_{2}}}dy\nonumber\\
&\leq C^{(2)}\left((t-s)^{-\frac{\gamma_{1}+1+d}{\gamma_{2}}}
+(t-s)^{-\frac{\gamma_{1}+\gamma_{2}+1+d}{\gamma_{2}}}\right)\label{eqn: partial t grad K 2}
\end{align}
for some constants $C^{(2)}>0$.
Hence, by combining \eqref{eqn: esti nabla K 1} and \eqref{eqn: partial t grad K 2},
we have the inequality given in \eqref{eqn: bound of dt nabla K l=t}.
\end{proof}

%%%%%%%%%%%%%%%%%%%%%%%%%%%%%%%
\subsection{A Proof of Theorem \ref{thm: sharp fct less than maximal}}
In order to prove Theorem \ref{thm: sharp fct less than maximal}, we need some lemmas.
Throughout this subsection, we let $(\psi_{1},\psi_{2})\in\mathfrak{S}_{\rm T}\times \mathfrak{S}$ such that
$N_{1},N_{2}>d+2+\lfloor\gamma_{1}\rfloor+\lfloor\gamma_{2}\rfloor$.

For a notational convenience, we put
\begin{align*}
L_{\rm{1-loc}}^{p}(\mathbb{R}^{d+1}; V)
&=\left\{
    \begin{array}{ll}
\text{the set of strong measurable functions $f:\mathbb{R}^{d+1}\to V$} \\
\text{such that $f\in L^{p}(K\times \mathbb{R}^{d}; V)$ for any compact subset $K\subset \mathbb{R}$.}
    \end{array}
  \right.
\end{align*}

Motivated by Corollary 2.7 (for $\psi_{1}(t,\xi)=-|\xi|^{\gamma/2}$ and $q=2$) in \cite{I. Kim K.-H. Kim 2016},
we have the following lemma which will be applied
for the proof of Lemma \ref{lem: bound of double average}.

\begin{lemma}\label{lem: g equal 0 outside of B3r1}
Let $q\geq 2$ be given.
Then for any $r'>0$, it holds that

\begin{itemize}
  \item [\rm (i)] there exists a constant $C_{1}>0$ depending on
  $r',d,q,\boldsymbol{\gamma},\boldsymbol{\mu}$ and $\boldsymbol{\kappa}$ such that
  for any $r>0$, $l\geq 0$ and $f\in L_{\rm{1-loc}}^{q}(\mathbb{R}^{d+1}; V)$ satisfying that $f(t,x)=0$ for $x\notin B_{3r}$,
\begin{align}
\int_{-2r'}^{0}\int_{B_{r}}|\mathcal{G}_{-2r',q}f(l,s,y)|^{q}dyds
\leq C_{1}\int_{-2r'}^{0}\int_{B_{3r}}\|f(s,y)\|_{V}^{q}dyds,\label{eqn:int -2r' B3r}\\
\int_{-2r'}^{0}\int_{B_{r}}|\widetilde{\mathcal{G}}_{-2r',q}f(s,y)|^{q}dyds
\leq C_{1}\int_{-2r'}^{0}\int_{B_{3r}}\|f(s,y)\|_{V}^{q}dyds,\label{eqn:int -2r' B3r 2}
\end{align}
  \item [\rm (ii)] (if $q=2$,)
there exists a constant $C_{2}>0$ depending on
$d,\boldsymbol{\gamma},\boldsymbol{\mu}$ and $\boldsymbol{\kappa}$
such that for any $r>0$, $l\geq 0$ and $f\in L_{\rm{1-loc}}^{2}(\mathbb{R}^{d+1}; V)$
satisfying that $f(t,x)=0$ for $x\notin B_{3r}$,
\begin{align}
\int_{-2r'}^{0}\int_{B_{r}}|\mathcal{G}_{-2r',2}f(l,s,y)|^{2}dyds
\leq C_{2}\int_{-2r'}^{0}\int_{B_{3r}}\|f(s,y)\|_{V}^{2}dyds.\label{eqn:int -2r' B3r 3}
\end{align}
\end{itemize}
\end{lemma}

\begin{proof}
(i)\enspace By considering $a=-2r'$ and $b=0$ in \eqref{eqn:LP p=q 1} of Theorem \ref{thm: LP ineq p=q}, we obtain that
\begin{align*}
\int_{-2r'}^{0}\int_{B_{r}}|\mathcal{G}_{-2r',q}f(l,s,y)|^{q}dyds
&\leq \int_{-2r'}^{0}\int_{\mathbb{R}^{d}}|\mathcal{G}_{-2r',q}f(l,s,y)|^{q}dyds\\
&\leq C_{1}\int_{-2r'}^{0}\|f(s,\cdot)\|_{L^{q}(\mathbb{R}^{d};V)}^{q}ds\\
&=C_{1}\int_{-2r'}^{0}\int_{\mathbb{R}^{d}}\|f(s,y)\|_{V}^{q}dyds\\
&=C_{1}\int_{-2r'}^{0}\int_{B_{3r}}\|f(s,y)\|_{V}^{q}dyds,
\end{align*}
which proves \eqref{eqn:int -2r' B3r}.
By applying the same arguments, we can prove the inequality given in \eqref{eqn:int -2r' B3r 2}.

(ii)\enspace By \eqref{eqn: LP 2} in Theorem \ref{thm: LP ineq p=q}, it is obvious that for any $-\infty\leq a<b\leq \infty$,
\begin{align}
\int_{a}^{b}\int_{\mathbb{R}^{d}}|\mathcal{G}_{a,2}f(l,s,y)|^{2}dyds
\leq C_{2}\int_{a}^{b}\|f(s,\cdot)\|_{L^{q}(\mathbb{R}^{d};V)}^{2}ds.\label{eqn: int a b q=2}
\end{align}
Thus, by considering $a=-2r'$ and $b=0$ in \eqref{eqn: int a b q=2}, we obtain that
\begin{align*}
\int_{-2r'}^{0}\int_{B_{r}}|\mathcal{G}_{-2r',2}f(l,s,y)|^{2}dyds
&\leq \int_{-2r'}^{0}\int_{\mathbb{R}^{d}}|\mathcal{G}_{-2r',q}f(l,s,y)|^{2}dyds\\
&\leq C_{2}\int_{-2r'}^{0}\|f(s,\cdot)\|_{L^{2}(\mathbb{R}^{d};V)}^{2}ds\\
&=C_{2}\int_{-2r'}^{0}\int_{\mathbb{R}^{d}}\|f(s,y)\|_{V}^{2}dyds\\
&=C_{2}\int_{-2r'}^{0}\int_{B_{3r}}\|f(s,y)\|_{V}^{2}dyds,
\end{align*}
which proves \eqref{eqn:int -2r' B3r 3}.
\end{proof}

Let $C_{0}^{1}(\mathbb{R}^{d})$ be the space of all continuously differentiable real valued functions vanishing at infinity.
Let $S_{1}(dw)$ denote the surface measure on the unit sphere.
If $d=1$, then $S_{1}(dw)$ becomes the counting measure on $\{-1,1\}$.

Motivated by Lemma 2.8 (for $V=\mathbb{R}$ and $q=2$) in \cite{I. Kim K.-H. Kim 2016},
we have the following lemma which will be applied
for the proofs of Lemmas \ref{lem: bound of integral of mathcal G over Qr2r1}--\ref{lem: bound of Dt mathcalG g-2}.

\begin{lemma}\label{lem: bound of convolution}
Let $q>1$ and let $q'$ be the conjugate number of $q$.
Let $V$ be a Banach space and let $f\in C_{\rm c}(\mathbb{R}^{d}; V)$
and $v\in C_{0}^{1}(\mathbb{R}^{d})$.
Let $R_{1},R_{2}>0$ be given.
Then for any $x,y\in\mathbb{R}^{d}$ satisfying that $|x-y|\leq R_{2}$ and $f(y-z)=0$ if $|z|\leq R_{1}$,
\begin{align}
\|(f*v)(y)\|_{V}
&\leq M(\mathbb{M}_{x}^{R_{1}+R_{2}}\|f\|_{V}^{q}(x))^{\frac{1}{q}}\nonumber\\
&\qquad \times \int_{R_{1}}^{\infty}(R_{2}+\rho)^{d}\left(\int_{\partial B_{1}}
 \left|\nabla v(\rho w)\cdot w\right|^{q'}S_{1}(dw)\right)^{\frac{1}{q'}}d\rho,
  \label{eqn:bound of convolution}
\end{align}
where $M=\left(\frac{\pi^{\frac{d}{2}}}{\Gamma\left(\frac{d}{2}+1\right)}\right)^{\frac{1}{q}}(\frac{1}{d})^{\frac{1}{q'}}$.
\end{lemma}

\begin{proof}
We assume that $d\geq 2$.
Let $x,y\in\mathbb{R}^{d}$ be given such that the assumption holds.
Let $\phi\in V^{*}$, where $V^{*}$ is the dual space of $V$.
We denote by $\bilin{\cdot}{\cdot}_{V^{*},V}$ the canonical bilinear form, i.e.,
\[
\bilin{\psi}{u}_{V^{*},V}=\psi(u)
\]
for all $\psi\in V^{*}$ and $u\in V$.
By applying the polar coordinate representation, the Fubini theorem
and the assumption that $f(y-z)=0$ if $|z|\leq R_{1}$, we obtain that
\begin{align}\label{eqn: convolution 1}
\bilin{\phi}{(f*v)(y)}_{V^{*},V}
&=\bilin{\phi}{\int_{|z|>R_{1}}f(y-z)v(z)dz}_{V^{*},V}\nonumber\\
&=\int_{R_{1}}^{\infty}\int_{\partial B_{1}}\bilin{\phi}{f(y-\rho w)}_{V^{*},V}v(\rho w)\rho^{d-1}S_{1}(dw)d\rho\nonumber\\
&=\int_{\partial B_{1}}I_{R_1}(\phi,\rho,w)\,S_{1}(dw),
\end{align}
where
\begin{align*}
I_{R_1}(\phi,\rho,w):=\int_{R_{1}}^{\infty}v(\rho w)\left(\frac{d}{d\rho}\int_{R_{1}}^{\rho}
         \bilin{\phi}{f(y-\tau w)}_{V^{*},V}\tau^{d-1}d\tau\right) d\rho.
\end{align*}
Since $v\in C_{0}^{1}(\mathbb{R}^{d})$ and $f\in C_{c}(\mathbb{R}^{d}; V)$,
by applying the integration by parts formula, for almost all $w$, we obtain that
\begin{align}\label{eqn: integration by part}
I_{R_1}(\phi,\rho,w)
&=\int_{R_{1}}^{\infty}v(\rho w)
    \left(\frac{d}{d\rho}\int_{R_{1}}^{\rho}\bilin{\phi}{f(y-\tau w)}_{V^{*},V}\tau^{d-1}d\tau\right) d\rho\nonumber\\
&=-\int_{R_{1}}^{\infty}\left(\nabla v(\rho w)\cdot w\right)
      \int_{R_{1}}^{\rho}\bilin{\phi}{f(y-\tau w)}_{V^{*},V}\tau^{d-1}d\tau d\rho\nonumber\\
&=\bilin{\phi}{-\int_{R_{1}}^{\infty} \left(\nabla v(\rho w)\cdot w\right)
          \int_{R_{1}}^{\rho}f(y-\tau w)\tau^{d-1}d\tau d\rho}_{V^{*},V}.
\end{align}
Since $\phi\in V^{*}$ is arbitrary,
by combining \eqref{eqn: convolution 1} and \eqref{eqn: integration by part}, we have
\begin{align}
(f*v)(y)
&=\int_{|z|>R_{1}}f(y-z)v(z)dz\nonumber\\
&=-\int_{R_{1}}^{\infty}\int_{\partial B_{1}}\left(\nabla v(\rho w)\cdot w\right)
          \int_{R_{1}}^{\rho}f(y-\tau w)\tau^{d-1}d\tau \,S_{1}(dw)d\rho. \label{eqn:f*v}
\end{align}
Since $|x-y|\leq R_{2}$, for any $\rho>R_{1}$, we obtain that
\begin{align*}
\int_{B_{\rho}}\|f(y-z)\|_{V}^{q}dz
&=\int_{B_{\rho}(y)}\|f(z)\|_{V}^{q}dz
 \leq \int_{B_{R_{2}+\rho}(x)}\|f(z)\|_{V}^{q}dz\\
&\leq \frac{\pi^{\frac{d}{2}}}{\Gamma\left(\frac{d}{2}+1\right)}(R_{2}+\rho)^{d}\mathbb{M}_{x}^{R_{1}+R_{2}}\|f\|_{V}^{q}(x),
\end{align*}
from which, by applying the H\"{o}lder's inequality, we obtain that
\begin{align*}
&\int_{\partial B_{1}}\int_{R_{1}}^{\rho}|\nabla v(\rho w)\cdot w|\tau^{\frac{d-1}{q'}}
          \|f(y-\tau w)\|_{V}\tau^{\frac{d-1}{q}}d\tau  S_{1}(dw)\\
&\leq \left(\int_{\partial B_{1}}\int_{R_{1}}^{\rho}
     \left|\nabla v(\rho w)\cdot w\right|^{q'}\tau^{d-1}d\tau S_{1}(dw)\right)^{\frac{1}{q'}}
\left(\int_{\partial B_{1}}\int_{R_{1}}^{\rho}\|f(y-\tau w)\|_{V}^{q}\tau^{d-1}
              d\tau  S_{1}(dw)\right)^{\frac{1}{q}}\\
&\leq \left(\frac{1}{d}\right)^{\frac{1}{q'}}\rho^{\frac{d}{q'}}
    \left(\int_{\partial B_{1}}\left|\nabla v(\rho w)\cdot w\right|^{q'} S_{1}(dw)\right)^{\frac{1}{q'}}
       \left(\int_{B_{\rho}}\|f(y-z)\|_{V}^{q}dz\right)^{\frac{1}{q}}\\
&\leq \left(\frac{1}{d}\right)^{\frac{1}{q'}}\rho^{\frac{d}{q'}}
    \left(\int_{\partial B_{1}}\left|\nabla v(\rho w)\cdot w\right|^{q'} S_{1}(dw)\right)^{\frac{1}{q'}}
       \left(\frac{\pi^{\frac{d}{2}}}{\Gamma\left(\frac{d}{2}+1\right)}(R_{2}+\rho)^{d}\mathbb{M}_{x}^{R_{1}+R_{2}}\|f\|_{V}^{q}(x)
              \right)^{\frac{1}{q}}\\
&\leq N_d\left(\mathbb{M}_{x}^{R_{1}+R_{2}}\|f\|_{V}^{q}(x)\right)^{\frac{1}{q}}
  (R_{2}+\rho)^{d} \left(\int_{\partial B_{1}}\left|\nabla v(\rho w)\cdot w\right|^{q'} S_{1}(dw)\right)^{\frac{1}{q'}}.
\end{align*}
Therefore, from \eqref{eqn:f*v}, by applying Fubini's theorem, we obtain that
\begin{align*}
\|(f*v)(y)\|_{V}
&=\left\|\int_{R_{1}}^{\infty}\int_{\partial B_{1}}\left(\nabla v(\rho w)\cdot w\right)
    \int_{R_{1}}^{\rho}f(y-\tau w)\tau^{d-1}d\tau S_{1}(dw)d\rho \right\|_{V}\\
&\leq \int_{R_{1}}^{\infty}\int_{\partial B_{1}}\int_{R_{1}}^{\rho}|\nabla v(\rho w)\cdot w|\tau^{\frac{d-1}{q'}}
          \|f(y-\tau w)\|_{V}\tau^{\frac{d-1}{q}}d\tau  S_{1}(dw) d\rho\\
&\le N_d\left(\mathbb{M}_{x}^{R_{1}+R_{2}}\|f\|_{V}^{q}(x)\right)^{\frac{1}{q}}\\
&\qquad \times
  \int_{R_{1}}^{\infty}(R_{2}+\rho)^{d}
    \left(\int_{\partial B_{1}}\left|\nabla v(\rho w)\cdot w\right|^{q'} S_{1}(dw)\right)^{\frac{1}{q'}}d\rho,
\end{align*}
which implies the desired assertion.
By similar arguments used above and using the counting measure $S_{1}(dw)$ on $\{-1,1\}$,
we can prove the case of $d=1$.
\end{proof}

For $r_{1},r_{2}>0$, we put
\begin{align*}
Q_{r_{1},r_{2}}:=(-2r_{1},0)\times B_{r_{2}}.
\end{align*}

Motivated by Lemma 2.9 (for $\psi_{1}(t,\xi)=-|\xi|^{\gamma/2}$ and $q=2$) in \cite{I. Kim K.-H. Kim 2016},
we have the following lemma which will be applied
for the proof of Lemma \ref{lem: bound of double average}.

\begin{lemma}\label{lem: bound of integral of mathcal G over Qr2r1}
Let $q\geq 2$ and $r_1,r_2>0$ be given.
Let $-\infty\leq a<b\leq\infty$.
Then there exists a constant $C>0$ depending on $d,\boldsymbol{\mu},\boldsymbol{\gamma},\boldsymbol{\kappa}$ and $q$
such that for any $x\in B_{r_{2}}$, $l\geq 0$ and $f\in C_{\rm c}^{\infty}((a,b)\times\mathbb{R}^{d};V)$
with support in $(-10r_{1},10r_{1})\times \mathbb{R}^{d}\setminus B_{2r_{2}}$,
\begin{align}
\int_{Q_{r_{1},r_{2}}}|\mathcal{G}_{a,q}f(l,s,y)|^{q}dsdy
&\leq Cr_{2}^{d-\gamma_{1}q}r_{1}^{\frac{q\gamma_{1}}{\gamma_{2}}}
   \int_{-10r_{1}}^{0}
\mathbb{M}_{x}^{3r_{2}}\|f\|_{V}^{q}(r,x)dr,\label{eqn:int-Q-Gg}\\
 \int_{Q_{r_{1},r_{2}}}|\widetilde{\mathcal{G}}_{a,q}f(s,y)|^{q}dsdy
&\leq Cr_{2}^{d-\gamma_{1}q}r_{1}^{\frac{q\gamma_{1}}{\gamma_{2}}}
   \int_{-10r_{1}}^{0}
\mathbb{M}_{x}^{3r_{2}}\|f\|_{V}^{q}(r,x)dr.\label{eqn:int-Q-Gg 2}
\end{align}
\end{lemma}

\begin{proof}
Since the proofs of \eqref{eqn:int-Q-Gg} and \eqref{eqn:int-Q-Gg 2} are similar,
we only prove \eqref{eqn:int-Q-Gg}.
Let $x\in B_{r_{2}}$ and $(s,y)\in Q_{r_{1},r_{2}}$.
If $|z|\leq r_{2}$, then $|y-z|\leq 2r_{2}$ and $f(r, y-z)=0$ from the assumption for $f$.
Therefore, we have
\begin{align*}
L_{\psi_{1}}(l)p_{\psi_{2}}(s,r,\cdot)*f(r,\cdot)(y)
=\int_{|z|\geq r_{2}}L_{\psi_{1}}(l)p_{\psi_{2}}(s,r,z)f(r,y-z)dz.
\end{align*}
Let $q'$ be the conjugate number of $q$.
By Lemma \ref{lem: bound of convolution} with $R_{1}=r_{2}$ and $R_{2}=2r_{2}$ and \eqref{eqn: bound of nabla K},
we obtain that
\begin{align}\label{eqn:est of lp*f}
&\|\left((L_{\psi_{1}}(l)p_{\psi_{2}}(s,r,\cdot))*f(r,\cdot)\right)(y)\|_{V}^{q}\nonumber\\
&\leq M^{q}\mathbb{M}_{x}^{3r_{2}}\|f\|_{V}^{q}(r,x)
        \left(\int_{r_{2}}^{\infty}(2r_{2}+\rho)^{d}\left(\int_{\partial B_{1}}|\nabla L_{\psi_{1}}(l)p_{\psi_{2}}(s,r,\rho w)|^{q'} S_{1}(dw)\right)^{\frac{1}{q'}}d\rho\right)^{q}\nonumber\\
&\leq M^{q}3^{d}\mathbb{M}_{x}^{3r_{2}}\|f\|_{V}^{q}(r,x)
         \left(\int_{r_{2}}^{\infty}\rho^{d}\left(\int_{\partial B_{1}}|\nabla L_{\psi_{1}}(l)p_{\psi_{2}}(s,r,\rho w)|^{q'} S_{1}(dw)\right)^{\frac{1}{q'}}d\rho\right)^{q}\nonumber\\
&\leq C_{1}^{q}M^{q}3^{d}\mathbb{M}_{x}^{3r_{2}}\|f\|_{V}^{q}(r,x)\nonumber\\
&\qquad\times\left(\int_{r_{2}}^{\infty}\rho^{d}\left(\int_{\partial B_{1}}\left(|\rho w|^{-(\gamma_{1}+1+d)}
              \wedge (s-r)^{-\frac{\gamma_{1}+1+d}{\gamma_{2}}}\right)^{q'} S_{1}(dw)\right)^{\frac{1}{q'}}d\rho\right)^{q}\nonumber\\
&=C_{1}^{q}M^{q}3^{d}\left(\frac{2\pi^{\frac{d}{2}}}{\Gamma(\frac{d}{2})}\right)^{\frac{q}{q^{\prime}}}
     \mathbb{M}_{x}^{3r_{2}}  \|f\|_{V}^{q}(r,x)  \left(\int_{r_{2}}^{\infty}\rho^{d}\left(\rho^{-(\gamma_{1}+1+d)}
         \wedge (s-r)^{-\frac{\gamma_{1}+1+d}{\gamma_{2}}}\right)  d\rho\right)^{q}\nonumber\\
&\leq C_{1}^{q}M^{q}3^{d}\left(\frac{2\pi^{\frac{d}{2}}}{\Gamma(\frac{d}{2})}\right)^{\frac{q}{q^{\prime}}}
\mathbb{M}_{x}^{3r_{2}}  \|f\|_{V}^{q}(r,x)   \left(\int_{r_{2}}^{\infty}\rho^{-(\gamma_{1}+1)} d\rho\right)^{q}\nonumber\\
&=C_{1}^{q}M^{q}3^{d}\left(\frac{2\pi^{\frac{d}{2}}}{\Gamma(\frac{d}{2})}\right)^{\frac{q}{q^{\prime}}}
     \left(\frac{1}{\gamma_{1}}\right)^{q}r_{2}^{-\gamma_{1}q}\mathbb{M}_{x}^{3r_{2}}\|f\|_{V}^{q}(r,x)\nonumber\\
&=:C^{(1)}r_{2}^{-\gamma_{1}q}\mathbb{M}_{x}^{3r_{2}}\|f\|_{V}^{q}(r,x),
\end{align}
where $C_{1}$ and $M$ are given as in \eqref{eqn: bound of nabla K} and \eqref{eqn:bound of convolution}, respectively.
By the assumption, the support of $f$ is in $(-10r_{1},10r_{1})\times\mathbb{R}^{d}\setminus B_{2r_{2}}$,
and so by direct computation, we can see that
\begin{align*}
\mathcal{G}_{a,q}f(l,s,y)
=\mathcal{G}_{\max\{a,-10r_{1}\},q}f(l,s,y).
\end{align*}
In fact, if $a<-10r_{1}$, then $\max\{a,-10r_{1}\}=-10r_{1}$ and we obtain that
\begin{align*}
\mathcal{G}_{a,q}f(l,s,y)
&=\left[\int_{a}^{s}(s-r)^{\frac{q\gamma_{1}}{\gamma_{2}}-1}
             \|L_{\psi_{1}}(l)p_{\psi_{2}}(s,r,\cdot)*f(r,\cdot)(y)\|_{V}^{q}dr\right]^{\frac{1}{q}}\\
&=\left[\int_{-10r_{1}}^{s}(s-r)^{\frac{q\gamma_{1}}{\gamma_{2}}-1}
            \|L_{\psi_{1}}(l)p_{\psi_{2}}(s,r,\cdot)*f(r,\cdot)(y)\|_{V}^{q}dr\right]^{\frac{1}{q}}\\
&=\mathcal{G}_{-10r_{1},q}f(l,s,y)\\
&=\mathcal{G}_{\max\{a,-10r_{1}\},q}f(l,s,y).
\end{align*}
On the other hand, if $a\geq -10r_{1}$, then $\max\{a,-10r_{1}\}=a$ and we have
\[
\mathcal{G}_{a,q}f(l,s,y)=\mathcal{G}_{\max\{a,-10r_{1}\},q}f(l,s,y).
\]
Thus by applying \eqref{eqn:est of lp*f}, we obtain that
\begin{align*}
&\int_{Q_{r_{1},r_{2}}}|\mathcal{G}_{a,q}f(l,s,y)|^{q}dsdy\\
&=\int_{Q_{r_{1},r_{2}}}|\mathcal{G}_{\max\{a,-10r_{1}\},q}f(l,s,y)|^{q}dsdy\\
&=\int_{B_{r_{2}}}  \int_{-2r_{1}}^{0}  \int_{\max\{a,-10r_{1}\}}^{s}  (s-r)^{\frac{q\gamma_{1}}{\gamma_{2}}-1}
         \|L_{\psi_{1}}(l)p_{\psi_{2}}(s,r,\cdot)*f(r,\cdot)(y)\|_{V}^{q}drdsdy\\
&\leq C^{(1)}r_{2}^{-\gamma_{1}q}\int_{B_{r_{2}}}\int_{-2r_{1}}^{0}\int_{\max\{a,-10r_{1}\}}^{s}
        (s-r)^{\frac{q\gamma_{1}}{\gamma_{2}}-1}  \mathbb{M}_{x}^{3r_{2}}\|f\|_{V}^{q}(r,x) dr ds dy\\
&\leq C^{(1)}r_{2}^{-\gamma_{1}q}\int_{B_{r_{2}}}\int_{-2r_{1}}^{0}\int_{-10r_{1}}^{s}
         (s-r)^{\frac{q\gamma_{1}}{\gamma_{2}}-1}\mathbb{M}_{x}^{3r_{2}}\|f\|_{V}^{q}(r,x)drdsdy\\
&= C^{(1)}r_{2}^{-\gamma_{1}q}|B_{r_{2}}|  \int_{-2r_{1}}^{0}\int_{-10r_{1}}^{s}
       (s-r)^{\frac{q\gamma_{1}}{\gamma_{2}}-1}\mathbb{M}_{x}^{3r_{2}}\|f\|_{V}^{q}(r,x)drds\\
&\leq C^{(1)}\left(\frac{\pi^{d/2}}{\Gamma\left(\frac{d}{2}+1\right)}\right)r_{2}^{d-\gamma_{1}q}
         \int_{-10r_{1}}^{0}\int_{-10r_{1}}^{s} (s-r)^{\frac{q\gamma_{1}}{\gamma_{2}}-1} \mathbb{M}_{x}^{3r_{2}}\|f\|_{V}^{q}(r,x)drds\\
&=C^{(1)}\left(\frac{\pi^{d/2}}{\Gamma\left(\frac{d}{2}+1\right)}\right) r_{2}^{d-\gamma_{1}q}
    \int_{-10r_{1}}^{0}\left(\int_{r}^{0}  (s-r)^{\frac{q\gamma_{1}}{\gamma_{2}}-1}ds\right)
     \mathbb{M}_{x}^{3r_{2}}\|f\|_{V}^{q}(r,x)dr\\
&=C^{(1)}\left(\frac{\pi^{d/2}}{\Gamma\left(\frac{d}{2}+1\right)}\right) r_{2}^{d-\gamma_{1}q}
     \frac{\gamma_{2}}{q\gamma_{1}}\int_{-10r_{1}}^{0}(-r)^{\frac{q\gamma_{1}}{\gamma_{2}}}
     \mathbb{M}_{x}^{3r_{2}}\|f\|_{V}^{q}(r,x)dr\\
&\leq C^{(1)}\left(\frac{\pi^{d/2}}{\Gamma\left(\frac{d}{2}+1\right)}\right)\frac{\gamma_{2}}{q\gamma_{1}}
        10^{\frac{q\gamma_{1}}{\gamma_{2}}} r_{2}^{d-\gamma_{1}q}r_{1}^{\frac{q\gamma_{1}}{\gamma_{2}}}
        \int_{-10r_{1}}^{0}  \mathbb{M}_{x}^{3r_{2}}\|f\|_{V}^{q}(r,x)dr,
\end{align*}
which proves \eqref{eqn:int-Q-Gg}.
\end{proof}

Motivated by Lemma 2.10 (for $\psi_{1}(t,\xi)=-|\xi|^{\gamma/2}$ and $q=2$) in \cite{I. Kim K.-H. Kim 2016},
we have the following lemma which will be applied
for the proof of Lemma \ref{lem: bound of double average}.

\begin{lemma}\label{lem: bound of sup of gradient of mathcal G}
Let $q\geq 2$ and let $r_{1}, r_{2}>0 $ satisfying $r_{1}^{1/\gamma_{2}}=r_{2}$.
Let $a\in\mathbb{R}$.
Then there exists a constant $C>0$ depending on $d,q,\boldsymbol{\mu}, \boldsymbol{\gamma}$ and $\boldsymbol{\kappa}$
such that for any $(t,x)\in Q_{r_{1},r_{2}}$, $l\geq 0$ and
$f\in C_{\rm c}^{\infty}(\mathbb{R}\times\mathbb{R}^{d};V)$ satisfying that $f(t,x)=0$ for $t\geq -8r_{1}$,
\begin{align}
\sup_{(s,y)\in Q_{r_{1},r_{2}}}|\nabla\mathcal{G}_{a,q}f(l,s,y)|^{q}
&\leq
Cr_{1}^{-\frac{q}{\gamma_{2}}}\mathbb{M}_{t}^{8r_{1}}\mathbb{M}_{x}^{2r_{2}}\|f\|_{V}^{q}(t,x),
\label{eqn: sup of gradient of mathcal G}\\
\sup_{(s,y)\in Q_{r_{1},r_{2}}}|\nabla\widetilde{\mathcal{G}}_{a,q}f(s,y)|^{q}
&\leq
Cr_{1}^{-\frac{q}{\gamma_{2}}}\mathbb{M}_{t}^{8r_{1}}\mathbb{M}_{x}^{2r_{2}}\|f\|_{V}^{q}(t,x).
\label{eqn: sup of gradient of mathcal G 2}
\end{align}
\end{lemma}

\begin{proof}
Since the proofs of \eqref{eqn: sup of gradient of mathcal G} and \eqref{eqn: sup of gradient of mathcal G 2}
are similar, we only prove \eqref{eqn: sup of gradient of mathcal G}.
Let $(t,x)\in Q_{r_{1},r_{2}}$, $l\geq 0$ and
$f\in C_{\rm c}^{\infty}(\mathbb{R}\times\mathbb{R}^{d};V)$ satisfying that $f(t,x)=0$ for $t\geq -8r_{1}$ be given.
If $a\geq -8r_{1}$, then, by the assumption, $f(r,x)=0$ for all $r\ge a \geq -8r_{1}$, and so we have
\begin{align*}
\mathcal{G}_{a,q}f(l,s,y)
=\left[\int_{a}^{s}(s-r)^{\frac{q\gamma_{1}}{\gamma_{2}}-1}\|L_{\psi_{1}}(l)p_{\psi_{2}}(s,r,\cdot)*f(r,\cdot)(y)\|_{V}^{q}dr\right]^{\frac{1}{q}}
=0.
\end{align*}
Therefore, there is nothing to prove because the LHS of \eqref{eqn: sup of gradient of mathcal G} is zero.
Hence, we assume that $a< -8r_{1}$.
Let $(s,y)\in Q_{r_{1},r_{2}}$.
Since $s>-2r_{1}>-8r_{1}$ and $f(r,x)=0$ for $r\geq -8r_{1}$, we obtain that
\begin{align}
\mathcal{G}_{a,q}f(l,s,y)
&=\left[\int_{a}^{s}(s-r)^{\frac{q\gamma_{1}}{\gamma_{2}}-1}
            \|L_{\psi_{1}}(l)p_{\psi_{2}}(s,r,\cdot)*f(r,\cdot)(y)\|_{V}^{q}dr\right]^{\frac{1}{q}}\nonumber\\
&=\left[\int_{a}^{-8r_{1}}(s-r)^{\frac{q\gamma_{1}}{\gamma_{2}}-1}
             \|L_{\psi_{1}}(l)p_{\psi_{2}}(s,r,\cdot)*f(r,\cdot)(y)\|_{V}^{q}dr\right]^{\frac{1}{q}}.
   \label{eqn: mathcal Ga g for r>-8r1}
\end{align}
Note that
\begin{align*}
\left|\frac{\|f(s+h,\cdot)\|-\|f(s,\cdot)\|}{h}\right| \leq \frac{\|f(s+h,\cdot)-f(s,\cdot)\|}{|h|}
\end{align*}
and then the derivative of a norm is less than or equal to the norm of the derivative if both exist,
which implies that
\begin{align}\label{eqn:bound of partial-y-G}
\left|\frac{\partial}{\partial y_{i}}\mathcal{G}_{a,q}f(l,s,y)\right|
\leq \left(\int_{a}^{-8r_{1}}(s-r)^{\frac{q\gamma_{1}}{\gamma_{2}}-1}
    \left\|\frac{\partial}{\partial y_{i}}K(l,s,r,\cdot)*f(r,\cdot)(y)\right\|_{V}^{q}dr\right)^{\frac{1}{q}},
\end{align}
where $K(l,s,r,y)=L_{\psi_{1}}(l)p_{\psi_{2}}(s,r,y)$.
Let $q'$ be the conjugate number of $q$.
By applying Lemma \ref{lem: bound of convolution} with $R_{1}=0$ and $R_{2}=2r_{2}$,
for any $r<-8r_{1}$, we have
\begin{align*}
&\left\|\frac{\partial}{\partial y_i}K(l,s,r,\cdot)*f(r,\cdot)(y)\right\|_{V}^{q}\\
&\quad\leq M\mathbb{M}_{x}^{2r_{2}}\|f\|_{V}^{q}(r,x)
\left[\int_{0}^{\infty}(2r_{2}+\rho)^{d}\left(\int_{\partial B_{1}}
      \left|\nabla \frac{\partial}{\partial y_i}K(l,s,r,\rho w)\right|^{q'}
      S_{1}(dw)\right)^{\frac{1}{q'}}d\rho\right]^{q},
\end{align*}
where we used the fact that
$|\nabla \frac{\partial}{\partial y_{i}}K(l,s,r,\rho w)\cdot w|
\le |\nabla \frac{\partial}{\partial y_{i}}K(l,s,r,\rho w)|$ for all $w\in\partial B_1$.
On the other hand, by applying the fact that
$\left|\nabla \frac{\partial}{\partial y_{i}}K(l,s,r,y)\right|\le\left|\nabla (\nabla K)(l,s,r,y)\right|$
and the Jensen's inequality, we obtain that
\begin{align*}
\left[\int_{0}^{\infty}(2r_{2}+\rho)^{d}\left(\int_{\partial B_{1}}
   \left|\nabla \frac{\partial}{\partial y_{i}}K(l,s,r,\rho w)\right|^{q'}
   S_{1}(dw)\right)^{\frac{1}{q'}}d\rho\right]^{q}
\leq 2^{q-1}(\mathcal{I}_{1}^{q}(r)+\mathcal{I}_{2}^{q}(r)),
\end{align*}
where
\begin{align*}
\mathcal{I}_{1}^{q}(r)
:&=\left(\int_{0}^{(s-r)^{\frac{1}{\gamma_{2}}}}(2r_{2}+\rho)^{d}
    \left(\int_{\partial B_{1}}\left|\nabla(\nabla K)(l,s,r,\rho w)\right|^{q'}S_{1}(dw)\right)^{\frac{1}{q'}}
    d\rho\right)^{q},\\
\mathcal{I}_{2}^{q}(r)
:&=\left(\int_{(s-r)^{\frac{1}{\gamma_{2}}}}^{\infty}(2r_{2}+\rho)^{d}
   \left(\int_{\partial B_{1}}\left|\nabla (\nabla K)(l,s,r,\rho w)\right|^{q'}S_{1}(dw)\right)^{\frac{1}{q'}}
   d\rho\right)^{q}.
\end{align*}
Therefore, we have
\begin{align*}
\left\|\frac{\partial}{\partial y_i}K(l,s,r,\cdot)*f(r,\cdot)(y)\right\|_{V}^{q}
\leq 2^{q-1}M\mathbb{M}_{x}^{2r_{2}}\|f\|_{V}^{q}(r,x)
    (\mathcal{I}_{1}^{q}(r)+\mathcal{I}_{2}^{q}(r)),
\end{align*}
from which and \eqref{eqn:bound of partial-y-G}, we have
\begin{align}\label{eqn:bound of partial-x-G-g}
\left|\frac{\partial}{\partial y_{i}}\mathcal{G}_{a,q}f(l,s,y)\right|^{q}
&\leq 2^{q-1}M\int_{a}^{-8r_{1}}(s-r)^{\frac{q\gamma_{1}}{\gamma_{2}}-1}\mathbb{M}_{x}^{2r_{2}}\|f\|_{V}^{q}(r,x)
  (\mathcal{I}_{1}^{q}(r)+\mathcal{I}_{2}^{q}(r))dr\nonumber\\
&=2^{q-1}M(\mathcal{J}_{1}+\mathcal{J}_{2}).
\end{align}
where
\begin{align*}
\mathcal{J}_{1}&:=\int_{a}^{-8r_{1}}(s-r)^{\frac{q\gamma_{1}}{\gamma_{2}}-1}\mathbb{M}_{x}^{2r_{2}}
                                                             \|f\|_{V}^{q}(r,x)\mathcal{I}_{1}^{q}(r)dr,\\
\mathcal{J}_{2}&:=\int_{a}^{-8r_{1}}(s-r)^{\frac{q\gamma_{1}}{\gamma_{2}}-1}\mathbb{M}_{x}^{2r_{2}}
                                                             \|f\|_{V}^{q}(r,x)\mathcal{I}_{2}^{q}(r)dr.
\end{align*}
If $a\le r\le -8r_{1}$, then it holds that $s>-2r_{1}>-8r_{1}>r$, and so we have
\begin{equation}\label{eqn: s-r greater than or equal const times r1}
(s-r)^{\frac{1}{\gamma_{2}}}
\geq (6r_{1})^{\frac{1}{\gamma_{2}}}
=6^{\frac{1}{\gamma_{2}}}r_{2}.
\end{equation}
First we estimate $\mathcal{J}_{1}$. If $0\le \rho\leq (s-r)^{\frac{1}{\gamma_{2}}}$,
then by \eqref{eqn: bound of nabla nabla K},
there exists a constant $C_{2}>0$ such that
\begin{align}\label{eqn: bound of nabla nabla K 1}
|\nabla(\nabla K)(l,s,r,\rho w)|
&\leq C_{2}(\rho^{-(\gamma_{1}+2+d)}\wedge (s-r)^{-\frac{\gamma_{1}+2+d}{\gamma_{2}}})\nonumber\\
&= C_{2}(s-r)^{-\frac{\gamma_{1}+2+d}{\gamma_{2}}}.
\end{align}
If $0\leq \rho\leq (s-r)^{\frac{1}{\gamma_{2}}}$, then by \eqref{eqn: s-r greater than or equal const times r1}, we have
\begin{align}
2r_{2}+\rho\leq (2\cdot 6^{-\frac{1}{\gamma_{2}}}+1)(s-r)^{\frac{1}{\gamma_{2}}}.
\label{eqn:2r2+rho upp bd}
\end{align}
Therefore, from \eqref{eqn: bound of nabla nabla K 1}, we obtain that
\begin{align}
\left(\mathcal{I}_{1}^{q}(r)\right)^{1/q}
&=\int_{0}^{(s-r)^{\frac{1}{\gamma_{2}}}}(2r_{2}+\rho)^{d}\left(\int_{\partial B_{1}}
     |\nabla(\nabla K)(l,s,r,\rho w)|^{q'}S_{1}(dw)\right)^{\frac{1}{q'}}d\rho\nonumber\\
&\leq C_{2}\left[S_{1}(\partial B_{1})\right]^{\frac{1}{q'}}
  (s-r)^{-\frac{\gamma_{1}+2+d}{\gamma_{2}}}\int_{0}^{(s-r)^{\frac{1}{\gamma_{2}}}}(2r_{2}+\rho)^{d}d\rho\nonumber\\
&\leq C_{2}\left[S_{1}(\partial B_{1})\right]^{\frac{1}{q'}}(2\cdot 6^{-\frac{1}{\gamma}}+1)^{d}
  (s-r)^{-\frac{\gamma_{1}+2+d}{\gamma_{2}}+\frac{d}{\gamma_{2}}+\frac{1}{\gamma_{2}}}\nonumber\\
&=C_{\mathcal{I}_1}(s-r)^{-\frac{1+\gamma_{1}}{\gamma_{2}}},
\quad
C_{\mathcal{I}_1}:=C_{2}\left(\frac{2\pi^{\frac{d}{2}}}{\Gamma\left(\frac{d}{2}\right)}\right)^{\frac{1}{q'}}
                                    (2\cdot 6^{-\frac{1}{\gamma_{2}}}+1)^{d}.\label{eqn:esti I1(r)}
\end{align}
If $a<r<-8r_{1}$, then it holds that $\frac{r}{2}<-4r_{1}< s$ and so we have
\begin{equation}\label{eqn: lower bound of s-r is -r}
s-r> \frac{-r}{2}.
\end{equation}
If $g\in C_{c}(\mathbb{R})$ such that $g(r)=0$ for all $r\geq u$ (for some fixed $u$),
then for any $c>0$, by the integration by parts formula, we see that
\begin{align}
\int_{-\infty}^{u}(s-r)^{-c}g(r)dr
=c\int_{-\infty}^{u}(s-r)^{-c-1}\left(\int_{r}^{0}g(v)dv\right)dr.\label{eqn:J1 int by part}
\end{align}
For any $a\in\mathbb{R}$ and $c>0$,
by applying \eqref{eqn:J1 int by part} and \eqref{eqn: lower bound of s-r is -r},
we obtain that
\begin{align}
&\int_{a}^{-8r_{1}}(s-r)^{-1-c}\mathbb{M}_{x}^{2r_{2}}\|f\|_{V}^{q}(r,x)dr \nonumber\\
&\leq
  \int_{-\infty}^{-8r_{1}}(s-r)^{-1-c}\mathbb{M}_{x}^{2r_{2}}\|f\|_{V}^{q}(r,x)dr \nonumber\\
&=\left(1+c\right)
   \int_{-\infty}^{-8r_{1}}(s-r)^{-2-c}
        \left(\int_{r}^{0}\mathbb{M}_{x}^{2r_{2}}\|f\|_{V}^{q}(s_{1},x)ds_{1}\right)dr \nonumber\\
&\leq \left(1+c\right)2^{2+c}
   \int_{-\infty}^{-8r_{1}}(-r)^{-1-c}
   \left(\frac{1}{(-r)}\int_{r}^{0}\mathbb{M}_{x}^{2r_{2}}\|f\|_{V}^{q}(s_{1},x)ds_{1}\right)dr \nonumber\\
&\leq \left(1+c\right)2^{2+c}
  \mathbb{M}_{t}^{8r_{1}}\mathbb{M}_{x}^{2r_{1}}\|f\|_{V}^{q}(t,x)
\int_{-\infty}^{-8r_{1}}(-r)^{-1-c}dr \nonumber\\
&=C_{c}r_1^{-c}\mathbb{M}_{t}^{8r_{1}}\mathbb{M}_{x}^{2r_{2}}\|f\|_{V}^{q}(t,x),
\quad C_{c}=\frac{1+c}{c}2^{2+c}8^{-c}.
 \label{eqn:bound of integral}
\end{align}
Therefore, by applying \eqref{eqn:esti I1(r)} and \eqref{eqn:bound of integral},
we obtain that
\begin{align} \label{eqn:bound of J1}
\mathcal{J}_{1}
&\leq C_{\mathcal{I}_1}^{q}
  \int_{a}^{-8r_{1}}(s-r)^{-1-\frac{q}{\gamma_{2}}}\mathbb{M}_{x}^{2r_{2}}\|f\|_{V}^{q}(r,x)dr \nonumber\\
&\le C_{\frac{q}{\gamma_2}}r_1^{-\frac{q}{\gamma_2}}\mathbb{M}_{t}^{8r_{1}}\mathbb{M}_{x}^{2r_{2}}\|f\|_{V}^{q}(t,x)
\end{align}
Now we estimate $\mathcal{J}_{2}$. If $\rho\geq (s-r)^{\frac{1}{\gamma_{2}}}$,
by \eqref{eqn: bound of nabla nabla K},
there exists a constant $C_{2}>0$ such that
\begin{align}
|\nabla(\nabla K)(l,s,r,\rho w))|
&\leq C_{2}\left(\rho^{-(\gamma_{1}+2+d)}\wedge (s-r)^{-\frac{\gamma_{1}+2+d}{\gamma_{2}}}\right)\nonumber\\
&= C_{2}\rho^{-(\gamma_{1}+2+d)}. \label{eqn: bound of nabla nabla K 2}
\end{align}
If $\rho\geq (s-r)^{\frac{1}{\gamma_{2}}}$, then by \eqref{eqn: s-r greater than or equal const times r1}, we have
\begin{align}
r_{2}\leq 6^{-\frac{1}{\gamma_{2}}}(s-r)^{\frac{1}{\gamma_{2}}}\leq 6^{-\frac{1}{\gamma_{2}}}\rho.
\label{eqn:r2 le rho}
\end{align}
Then, from \eqref{eqn:r2 le rho} and \eqref{eqn: bound of nabla nabla K 2}, we obtain that
\begin{align*}
\left(\mathcal{I}_{2}^{q}(r)\right)^{1/q}
&=\int_{(s-r)^{\frac{1}{\gamma_{2}}}}^{\infty}(2r_{2}+\rho)^{d}\left(\int_{\partial B_{1}}
            |\nabla(\nabla K)(l,s,r,\rho w)|^{q'}S_{1}(dw)\right)^{\frac{1}{q'}}d\rho\\
&\leq (2\cdot 6^{-\frac{1}{\gamma_{2}}}+1)^{d}\int_{(s-r)^{\frac{1}{\gamma_{2}}}}^{\infty}\rho^{d}
        \left(\int_{\partial B_{1}}|\nabla(\nabla K)(l,s,r,\rho w)|^{q'}S_{1}(dw)\right)^{\frac{1}{q'}}d\rho\\
&\leq (2\cdot 6^{-\frac{1}{\gamma_{2}}}+1)^{d}C_{2}\int_{(s-r)^{\frac{1}{\gamma_{2}}}}^{\infty}\rho^{d}
        \left(\int_{\partial B_{1}}\rho^{-q'(\gamma_{1}+2+d)}S_{1}(dw)\right)^{\frac{1}{q'}}d\rho\\
&=(2\cdot 6^{-\frac{1}{\gamma_{2}}}+1)^{d}
        C_{2}\left(\frac{2\pi^{\frac{d}{2}}}{\Gamma\left(\frac{d}{2}\right)}\right)^{\frac{1}{q^{\prime}}}
            \int_{(s-r)^{\frac{1}{\gamma_{2}}}}^{\infty}\rho^{-(\gamma_{1}+2)}d\rho\\
&=(2\cdot 6^{-\frac{1}{\gamma_{2}}}+1)^{d} C_{2}\left(\frac{2\pi^{\frac{d}{2}}}{\Gamma\left(\frac{d}{2}\right)}\right)^{\frac{1}{q^{\prime}}}\frac{1}{1+\gamma_{1}}
  (s-r)^{-\frac{1+\gamma_{1}}{\gamma_{2}}}\\
&=:C_{\mathcal{I}_{2}}(s-r)^{-\frac{1+\gamma_{1}}{\gamma_{2}}}.
\end{align*}
Then from the estimations obtained in \eqref{eqn:bound of J1}, we have
\begin{align}
\mathcal{J}_{2}
&\leq\left(1+\frac{q}{\gamma_{2}}\right)\frac{2^{2+\frac{q}{\gamma_{2}}}\gamma_{2}}{q}
        C_{\mathcal{I}_2}^{q}(8r_{1})^{-\frac{q}{\gamma_{2}}}
        \mathbb{M}_{t}^{8r_{1}}\mathbb{M}_{x}^{2r_{2}}\|f\|_{V}^{q}(t,x).  \label{eqn:bound of J2}
\end{align}
Therefore, by combining \eqref{eqn:bound of partial-x-G-g},
\eqref{eqn:bound of J1}, and \eqref{eqn:bound of J2}, we have
\begin{align*}
\left|\frac{\partial}{\partial y_{i}}\mathcal{G}_{a,q}f(l,s,y)\right|^{q}
&\leq 2^{q-1}M\left(1+\frac{q}{\gamma_{2}}\right)\frac{2^{2+\frac{q}{\gamma_{2}}}
            \gamma_{2}}{q} (C_{\mathcal{I}_1}^{q}+C_{\mathcal{I}_2}^{q})\\
 &\qquad\times (8r_{1})^{-\frac{q}{\gamma_{2}}}\mathbb{M}_{t}^{8r_{1}}\mathbb{M}_{x}^{2r_{2}}\|f\|_{V}^{q}(t,x).
\end{align*}
By the assumption $q\geq 2$, we have $\frac{q}{2}\geq 1$ and then by the Jensen's inequality,
we obtain that
\begin{align*}
|\nabla\mathcal{G}_{a,q}f(l,s,y)|^{q}
&=\left(\sum_{i=1}^{d}\left|\frac{\partial}{\partial y_{i}}\mathcal{G}_{a,q}f(l,s,y)\right|^{2}\right)^{\frac{q}{2}}
\leq C_{q,d}\sum_{i=1}^{d}\left|\frac{\partial}{\partial y_{i}}\mathcal{G}_{a,q}f(l,s,y)\right|^{q}\\
&\leq C_{q,d}d 2^{q-1} M\left(1+\frac{q}{\gamma_{2}}\right)\frac{2^{2+\frac{q}{\gamma_{2}}}\gamma_{2}}{q}
            (C_{\mathcal{I}_1}^{q}+C_{\mathcal{I}_2}^{q})\\
 &\qquad\times (8r_{1})^{-\frac{q}{\gamma_{2}}}\mathbb{M}_{t}^{8r_{1}}\mathbb{M}_{x}^{2r_{2}}\|f\|_{V}^{q}(t,x).
 \end{align*}
Therefore, there exists a constant $C>0$, which is independent of $(s,y)$, such that
\begin{align*}
|\nabla\mathcal{G}_{a,q}f(l,s,y)|^{q}
\leq Cr_{1}^{-\frac{q}{\gamma_{2}}}\mathbb{M}_{t}^{8r_{1}}\mathbb{M}_{x}^{2r_{2}}\|f\|_{V}^{q}(t,x).
\end{align*}
Hence, by taking supremum on the left hand side in $(s,y)$, we have \eqref{eqn: sup of gradient of mathcal G}.
\end{proof}
The following Lemmas \ref{lem: bound of Dt mathcalG g-1} and \ref{lem: bound of Dt mathcalG g-2}
will be used in the proof of Lemma \ref{lem: bound of Dt mathcalG g}.

\begin{lemma}\label{lem: bound of Dt mathcalG g-1}
Let $q\geq 2$ and let $r_{1}, r_{2}>0 $ satisfying $r_{1}^{1/\gamma_{2}}=r_{2}$.
Let $a \in \mathbb{R}$.
Then there exists a constant $C>0$ depending on $d,q, \boldsymbol{\mu}, \boldsymbol{\gamma}$
and $\boldsymbol{\kappa}$ such that
for any $(s,y),(t,x)\in Q_{r_{1},r_{2}}$, $l\geq 0$  and
$f\in C_{\rm c}^{\infty}(\mathbb{R}\times\mathbb{R}^{d};V)$ satisfying that $f(t,x)=0$ for $t\geq -8r_{1}$,
\begin{align*}
A:=\int_{a}^{-8r_{1}}(s-r)^{\frac{q\gamma_{1}}{\gamma_{2}}-1-q}
                \|L_{\psi_{1}}(l)p_{\psi_{2}}(s,r,\cdot)*f(r,\cdot)(y)\|_{V}^{q}dr
\leq Cr_{1}^{-q}\mathbb{M}_{t}^{8r_{1}}\mathbb{M}_{x}^{2r_{2}}\|f\|_{V}^{q}(t,x).
\end{align*}
\end{lemma}

\begin{proof}
Let $(s,y),(t,x)\in Q_{r_{1},r_{2}}$, $l\geq 0$  and
$f\in C_{\rm c}^{\infty}(\mathbb{R}\times\mathbb{R}^{d};V)$ satisfying that $f(t,x)=0$ for $t\geq -8r_{1}$
be given.
Let $q'$ be the conjugate number of $q$. By applying Lemma \ref{lem: bound of convolution} with $R_{1}=0$ and $R_{2}=2r_{2}$,
for any $r<-8r_{1}$, we have
\begin{align*}
&\left\|L_{\psi_{1}}(l)p_{\psi_{2}}(s,r,\cdot)*f(r,\cdot)(y)\right\|_{V}^{q}\\
&\leq M\mathbb{M}_{x}^{2r_{2}}\|f\|_{V}^{q}(r,x)
\left[\int_{0}^{\infty}(2r_{2}+\rho)^{d}\left(\int_{\partial B_{1}}
      \left|\nabla L_{\psi_{1}}(l)p_{\psi_{2}}(s,r,\rho w)\right|^{q'}
      S_{1}(dw)\right)^{\frac{1}{q'}}d\rho\right]^{q},
\end{align*}
where we used the fact that
$|\nabla L_{\psi_{1}}(l)p_{\psi_{2}}(s,r,\rho w)\cdot w|
\le |\nabla L_{\psi_{1}}(l)p_{\psi_{2}}(s,r,\rho w)|$ for all $w\in\partial B_1$.
On the other hand, by applying the Jensen's inequality, we have
\begin{align*}
\left[\int_{0}^{\infty}(2r_{2}+\rho)^{d}\left(\int_{\partial B_{1}}
      \left|\nabla L_{\psi_{1}}(l)p_{\psi_{2}}(s,r,\rho w)\right|^{q'}
      S_{1}(dw)\right)^{\frac{1}{q'}}d\rho\right]^{q}
\le 2^{q-1}\sum_{j=1}^{2}A_{j}^{q}(r),
\end{align*}
where
\begin{align*}
A_{1}^{q}(r)
:&=\left(\int_{0}^{(s-r)^{\frac{1}{\gamma_{2}}}}(2r_{2}+\rho)^{d}\left(\int_{\partial B_{1}}
           \left|\nabla L_{\psi_{1}}(l)p_{\psi_{2}}(s,r,\rho w)\right|^{q'}
           S_{1}(dw)\right)^{\frac{1}{q'}}d\rho\right)^{q},\\
A_{2}^{q}(r)
:&=\left(\int_{(s-r)^{\frac{1}{\gamma_{2}}}}^{\infty}(2r_{2}+\rho)^{d}\left(\int_{\partial B_{1}}
   \left|\nabla L_{\psi_{1}}(l)p_{\psi_{2}}(s,r,\rho w)\right|^{q'}
           S_{1}(dw)\right)^{\frac{1}{q'}}d\rho\right)^{q}.
\end{align*}
Therefore, we have
\begin{align*}
\left\|L_{\psi_{1}}(l)p_{\psi_{2}}(s,r,\cdot)*f(r,\cdot)(y)\right\|_{V}^{q}
\leq 2^{q-1}M\mathbb{M}_{x}^{2r_{2}}\|f\|_{V}^{q}(r,x)
  \left(A_{1}^{q}(r)+A_{2}^{q}(r)\right),
\end{align*}
from which, we obtain that
\begin{align}\label{eqn:bound of A1}
A
&\leq 2^{q-1}M
  \int_{a}^{-8r_{1}}(s-r)^{\frac{q\gamma_{1}}{\gamma_{2}}-1-q}
  \mathbb{M}_{x}^{2r_{1}}\|f\|_{V}^{q}(r,x)(A_{1}^{q}(r)+A_{2}^{q}(r))dr
  \nonumber\\
&= 2^{q-1}M(\mathcal{A}_{1}+\mathcal{A}_{2})
\end{align}
with
\begin{align*}
\mathcal{A}_{1}&:=\int_{a}^{-8r_{1}}(s-r)^{\frac{q\gamma_{1}}{\gamma_{2}}-1-q}
                                    \mathbb{M}_{x}^{2r_{1}}\|f\|_{V}^{q}(r,x)A_{1}^{q}(r)dr,\\
\mathcal{A}_{2}&:=\int_{a}^{-8r_{2}}(s-r)^{\frac{q\gamma_{1}}{\gamma_{2}}-1-q}
                                    \mathbb{M}_{x}^{2r_{2}}\|f\|_{V}^{q}(r,x)A_{2}^{q}(r)dr.
\end{align*}
We estimate $\mathcal{A}_{1}$. If $0\leq \rho\leq (s-r)^{\frac{1}{\gamma_{2}}}$,
then by \eqref{eqn: bound of nabla K}, there exists a constant $C_{1}>0$ such that
\begin{align}
\left|\nabla L_{\psi_{1}}(l)p_{\psi_{2}}(s,r,\rho w)\right|
&\leq C_{1}\left(\rho^{-(\gamma_{1}+1+d)}\wedge (s-r)^{-\frac{\gamma_{1}+1+d}{\gamma_{2}}}\right)\nonumber\\
&\leq C_{1} (s-r)^{-\frac{\gamma_{1}+1+d}{\gamma_{2}}}.\label{eqn:nab Lp1}
\end{align}
Therefore, by applying \eqref{eqn:nab Lp1} and \eqref{eqn:2r2+rho upp bd}, we obtain that
\begin{align*}
\left(A_{1}^{q}(r)\right)^{1/q}
&=\int_{0}^{(s-r)^{\frac{1}{\gamma_{2}}}}(2r_{2}+\rho)^{d}
    \left(\int_{\partial B_{1}}\left|\nabla L_{\psi_{1}}(l)p_{\psi_{2}}(s,r,\rho w)\right|^{q'}S_{1}(dw)\right)^{\frac{1}{q'}}d\rho\\
&\leq C_{1}\int_{0}^{(s-r)^{\frac{1}{\gamma_{2}}}}(2\cdot 6^{-\frac{1}{\gamma_{2}}}+1)^{d}(s-r)^{-\frac{1+\gamma_{1}}{\gamma_{2}}}
 \left(\int_{\partial B_{1}}S_{1}(dw)\right)^{\frac{1}{q'}}d\rho\\
&\leq C_{1}(2\cdot 6^{-\frac{1}{\gamma_{2}}}+1)^{d}
\left[S_{1}(\partial B_{1})\right]^{\frac{1}{q'}}(s-r)^{-\frac{\gamma_{1}}{\gamma_{2}}}\\
&=C_{A_{1}}(s-r)^{-\frac{\gamma_{1}}{\gamma_{2}}},
\quad
C_{A_{1}}:=C_{1}(2\cdot 6^{-\frac{1}{\gamma_{2}}}+1)^{d}
  \left(\frac{2\pi^{d/2}}{\Gamma\left(\frac{d}{2}\right)}\right)^{\frac{1}{q'}}.
\end{align*}
By applying \eqref{eqn:bound of integral}, we obtain that
\begin{align}
\mathcal{A}_{1}
&\leq C_{A_{1}}^{q} \int_{a}^{-8r_{1}}(s-r)^{\frac{q\gamma_{1}}{\gamma_{2}}-1-q}
            \mathbb{M}_{x}^{2r_{2}}\|f\|_{V}^{q}(r,x)(s-r)^{-\frac{q\gamma_{1}}{\gamma_{2}}}dr \nonumber\\
&\le C_{A_{1}}^{q}C_{q}r_{1}^{-q}\mathbb{M}_{t}^{8r_{1}}\mathbb{M}_{x}^{2r_{2}}\|f\|_{V}^{q}(t,x),
  \label{eqn:bound of mathcal A 11}
\end{align}
where $C_{q}$ is given in \eqref{eqn:bound of integral}.
Now we estimate $\mathcal{A}_{2}$. If $\rho\geq (s-r)^{\frac{1}{\gamma_{2}}}$,
by \eqref{eqn: bound of nabla K}, there exists a constant $C_{1}>0$ such that
\begin{align}
\left|\nabla L_{\psi_{1}}(l)p_{\psi_{2}}(s,r,\rho w)\right|
&\leq C_{1}\left(\rho^{-(\gamma_{1}+1+d)}\wedge (s-r)^{-\frac{\gamma_{1}+1+d}{\gamma_{2}}}\right)\nonumber\\
&= C_{1}\rho^{-(\gamma_{1}+1+d)}.\label{eqn:bound nabla L(l)p 1}
\end{align}
If $\rho\geq (s-r)^{\frac{1}{\gamma_{2}}}$, then by \eqref{eqn: lower bound of s-r is -r}, we have
\begin{equation}
2r_{2}+\rho
\leq 2 \cdot 6^{-\frac{1}{\gamma_{2}}}(s-r)^{\frac{1}{\gamma_{2}}}+\rho
\leq (2\cdot 6^{-\frac{1}{\gamma_{2}}}+1)\rho.\label{eqn: r2 less than or equal constant rho}
\end{equation}
By applying \eqref{eqn: r2 less than or equal constant rho} and \eqref{eqn:bound nabla L(l)p 1}, we obtain that
\begin{align*}
\left(A_{2}^{q}(r)\right)^{1/q}
&=\int_{(s-r)^{\frac{1}{\gamma_{2}}}}^{\infty}(2r_{2}+\rho)^{d}
 \left(\int_{\partial B_{1}}\left|\nabla L_{\psi_{1}}(l)
        p_{\psi_{2}}(s,r,\rho w)\right|^{q'}S_{1}(dw)\right)^{\frac{1}{q'}}d\rho\\
&\leq C_{1}(2\cdot 6^{-\frac{1}{\gamma_{2}}}+1)^{d}\int_{(s-r)^{\frac{1}{\gamma_{2}}}}^{\infty}
        \rho^{-\gamma_{1}-1}\left(\int_{\partial B_{1}}S_{1}(dw)\right)^{\frac{1}{q'}}d\rho\\
&=C_{1}(2\cdot 6^{-\frac{1}{\gamma_{2}}}+1)^{d}[S_{1}(\partial B_{1})]^{\frac{1}{q'}}\frac{1}{\gamma_{1}}(s-r)^{-\frac{\gamma_{1}}{\gamma_{2}}}\\
&=:C_{A_{2}}(s-r)^{-\frac{\gamma_{1}}{\gamma_{2}}},\quad C_{A_{12}}=C_{1}(2\cdot 6^{-\frac{1}{\gamma_{2}}}+1)^{d}\left(\frac{2\pi^{d/2}}{\Gamma\left(\frac{d}{2}\right)}\right)^{\frac{1}{q'}}\frac{1}{\gamma_{1}}.
\end{align*}
Then we have
\begin{align*}
\mathcal{A}_{2}
&\leq C_{A_{2}}^{q}
\int_{a}^{-8r_{1}}(s-r)^{-1-q}\mathbb{M}_{x}^{2r_{2}}\|f\|_{V}^{q}(r,x)dr.
\end{align*}
Therefore, by applying \eqref{eqn:bound of integral}, we have
\begin{align}
\mathcal{A}_{2}
&\leq C_{A_{2}}^{q}C_{q}r_{1}^{-q}\mathbb{M}_{t}^{8r_{1}}\mathbb{M}_{x}^{2r_{2}}\|f\|_{V}^{q}(t,x),
    \label{eqn:bound of mathcal A 12}
\end{align}
where $C_{q}$ is given in \eqref{eqn:bound of integral}.
Hence, by combining \eqref{eqn:bound of A1}, \eqref{eqn:bound of mathcal A 11},
and \eqref{eqn:bound of mathcal A 12}, we have
\begin{align*}%\label{eqn: esti of A1}
A
\leq 2^{q-1}M\left(C_{A_{1}}^{q}+C_{A_{2}}^{q}\right)C_{q} r_{1}^{-q}\mathbb{M}_{t}^{8r_{1}}\mathbb{M}_{x}^{2r_{2}}\|f\|_{V}^{q}(t,x),
\end{align*}
which proves the desired result.
\end{proof}

\begin{lemma}\label{lem: bound of Dt mathcalG g-2}
Let $q\geq 2$ and let $r_{1}, r_{2}>0 $ satisfying $r_{1}^{1/\gamma_{2}}=r_{2}$.
Then it holds that
\begin{itemize}
  \item [\rm (i)] if $a\in \mathbb{R}$, then there exists a constant $C_{1}>0$ depending on
  $d,q, \boldsymbol{\mu}, \boldsymbol{\gamma}$ and $\boldsymbol{\kappa}$ such that
  for any $(s,y), (t,x)\in Q_{r_{1},r_{2}}$, $l\geq 0$
  and $f\in C_{\rm c}^{\infty}(\mathbb{R}\times\mathbb{R}^{d};V)$ satisfying that $f(t,x)=0$ for $t\geq -8r_{1}$,
\begin{align}
A&:=\int_{a}^{-8r_{1}}(s-r)^{\frac{q\gamma_{1}}{\gamma_{2}}-1}
                \left\|\frac{\partial}{\partial s}\left(L_{\psi_{1}}(l)p_{\psi_{2}}(s,r,\cdot)\right)*f(r,\cdot)(y)\right\|_{V}^{q}dr\nonumber\\
&\quad\leq C_{1}r_{1}^{-q}\mathbb{M}_{t}^{8r_{1}}\mathbb{M}_{x}^{2r_{2}}\|f\|_{V}^{q}(t,x),\label{eqn: esti time deri of mathcal G1}
\end{align}
  \item [\rm (ii)] if $-\infty< a\le -2r_1<0\le b<\infty$, then there exists a constant $C_{2}>0$ depending on
  $d,q, \boldsymbol{\mu}, \boldsymbol{\gamma}$ and $\boldsymbol{\kappa}$ such that
  for any $(s,y),(t,x)\in Q_{r_{1},r_{2}}$ and
  $f\in C_{\rm c}^{\infty}(\mathbb{R}\times\mathbb{R}^{d};V)$ satisfying that $f(t,x)=0$ for $t\geq -8r_{1}$,
\begin{align}
&\int_{a}^{-8r_{1}}(s-r)^{\frac{q\gamma_{1}}{\gamma_{2}}-1}
                \left\|\frac{\partial}{\partial s}\left(L_{\psi_{1}}(s)p_{\psi_{2}}(s,r,\cdot)\right)*f(r,\cdot)(y)\right\|_{V}^{q}dr\nonumber\\
&\quad\leq C_{2}\left((b-a)+1\right)^{q}
            r_{1}^{-q}\mathbb{M}_{t}^{8r_{1}}\mathbb{M}_{x}^{2r_{2}}\|f\|_{V}^{q}(t,x).\label{eqn: esti time deri of mathcal G2}
\end{align}
\end{itemize}
\end{lemma}

\begin{proof}
Since the proofs of \eqref{eqn: esti time deri of mathcal G1} and \eqref{eqn: esti time deri of mathcal G2} are similar,
we only prove \eqref{eqn: esti time deri of mathcal G2}.
Let $(s,y), (t,x)\in Q_{r_{1},r_{2}}$ and $f\in C_{\rm c}^{\infty}(\mathbb{R}\times\mathbb{R}^{d};V)$
satisfying that $f(t,x)=0$ for $t\geq -8r_{1}$ be given.
Let $q'$ be the conjugate number of $q$. By applying Lemma \ref{lem: bound of convolution} with $R_{1}=0$ and $R_{2}=2r_{2}$,
for any $r<-8r_{1}$, and applying the Jensen's inequality,
we obtain that
\begin{align*}
&\left\|\frac{\partial}{\partial s}L_{\psi_{1}}(s)p_{\psi_{2}}(s,r,\cdot)*f(r,\cdot)(y)\right\|_{V}^{q}\\
&\leq M\mathbb{M}_{x}^{2r_{2}}\|f\|_{V}^{q}(r,x)
\left[\int_{0}^{\infty}(2r_{2}+\rho)^{d}\left(\int_{\partial B_{1}}
      \left|\nabla \frac{\partial}{\partial s}L_{\psi_{1}}(s)p_{\psi_{2}}(s,r,\rho w)\right|^{ q'}
      S_{1}(dw)\right)^{\frac{1}{q'}}d\rho\right]^{q}\\
&\leq 2^{q-1}M\mathbb{M}_{x}^{2r_{2}}\|f\|_{V}^{q}(r,x)
  \left(A_{1}^{q}(r)+A_{2}^{q}(r)\right),
\end{align*}
where
\begin{align*}
A_{1}^{q}(r)
:&=\left(\int_{0}^{(s-r)^{\frac{1}{\gamma_{2}}}}(2r_{2}+\rho)^{d}\left(\int_{\partial B_{1}}
           \left|\nabla \frac{\partial}{\partial s}L_{\psi_{1}}(s)p_{\psi_{2}}(s,r,\rho w)\right|^{q'}
           S_{1}(dw)\right)^{\frac{1}{q'}}d\rho\right)^{q},\\
A_{2}^{q}(r)
:&=\left(\int_{(s-r)^{\frac{1}{\gamma_{2}}}}^{\infty}(2r_{2}+\rho)^{d}\left(\int_{\partial B_{1}}
   \left|\nabla \frac{\partial}{\partial s}L_{\psi_{1}}(s)p_{\psi_{2}}(s,r,\rho w)\right|^{q'}
           S_{1}(dw)\right)^{\frac{1}{q'}}d\rho\right)^{q}.
\end{align*}
Therefore, we obtain that
\begin{align}\label{eqn:bound of partial-s-G-g}
A
&\leq 2^{q-1}M \int_{a}^{-8r_{1}}(s-r)^{\frac{q\gamma_{1}}{\gamma_{2}}-1}
  \mathbb{M}_{x}^{2r_{1}}\|f\|_{V}^{q}(r,x)(A_{1}^{q}(r)+A_{2}^{q}(r))dr
  \nonumber\\
&= 2^{q-1}M(\mathcal{A}_{1}+\mathcal{A}_{2})
\end{align}
with
\begin{align*}
\mathcal{A}_{1}&:=\int_{a}^{-8r_{1}}(s-r)^{\frac{q\gamma_{1}}{\gamma_{2}}-1}
                                    \mathbb{M}_{x}^{2r_{1}}\|f\|_{V}^{q}(r,x)A_{1}^{q}(r)dr,\\
\mathcal{A}_{2}&:=\int_{a}^{-8r_{2}}(s-r)^{\frac{q\gamma_{1}}{\gamma_{2}}-1}
                                    \mathbb{M}_{x}^{2r_{2}}\|f\|_{V}^{q}(r,x)A_{2}^{q}(r)dr.
\end{align*}
We estimate $\mathcal{A}_{1}$. If $0\leq \rho\leq (s-r)^{\frac{1}{\gamma_{2}}}$,
by \eqref{eqn: bound of dt nabla K l=t}, there exists a constant $C_{1}>0$
such that
\begin{align}
\left|\nabla \frac{\partial}{\partial s}L_{\psi_{1}}(s)p_{\psi_{2}}(s,r,\rho w)\right|
&\leq C_{1}\left((\rho^{-(\gamma_{1}+1+d)}+\rho^{-(\gamma_{1}+\gamma_{2}+1+d)})\right.\nonumber\\
&\qquad\left.\wedge \left((s-r)^{-\frac{\gamma_{1}+1+d}{\gamma_{2}}}
                                +(s-r)^{-\frac{\gamma_{1}+\gamma_{2}+1+d}{\gamma_{2}}}\right)\right)\nonumber\\
&= C_{1}\left((s-r)^{-\frac{\gamma_{1}+1+d}{\gamma_{2}}}
                                +(s-r)^{-\frac{\gamma_{1}+\gamma_{2}+1+d}{\gamma_{2}}}\right).\label{eqn:patial s L(s)p}
\end{align}
Therefore, by applying \eqref{eqn:2r2+rho upp bd} and \eqref{eqn:patial s L(s)p},
we obtain that
\begin{align*}
\left(A_{1}^{q}(r)\right)^{1/q}
&=\int_{0}^{(s-r)^{\frac{1}{\gamma_{2}}}}(2r_{2}+\rho)^{d}\left(\int_{\partial B_{1}}
        \left|\nabla \frac{\partial}{\partial s}L_{\psi_{1}}(s)p_{\psi_{2}}(s,r,\rho w)\right|^{q'}S_{1}(dw)\right)^{\frac{1}{q'}}d\rho\\
&\leq\int_{0}^{(s-r)^{\frac{1}{\gamma_{2}}}}(2\cdot 6^{-\frac{1}{\gamma_{2}}}+1)^{d}C_{1}
        \left((s-r)^{-\frac{\gamma_{1}+1}{\gamma_{2}}}
                +(s-r)^{-\frac{\gamma_{1}+\gamma_{2}+1}{\gamma_{2}}}\right)\\
 &\qquad\times\left(\int_{\partial B_{1}}S_{1}(dw)\right)^{\frac{1}{q'}}d\rho\\
&\leq C_{1}(2\cdot 6^{-\frac{1}{\gamma_{2}}}+1)^{d}\left[S_{1}(\partial B_{1})\right]^{\frac{1}{q'}}
            \left((s-r)^{-\frac{\gamma_{1}}{\gamma_{2}}}
                            +(s-r)^{-\frac{\gamma_{1}+\gamma_{2}}{\gamma_{2}}}\right)\\
&=:C_{A_{1}}\left((s-r)^{-\frac{\gamma_{1}}{\gamma_{2}}}+(s-r)^{-\frac{\gamma_{1}+\gamma_{2}}{\gamma_{2}}}\right),
\end{align*}
Since $a< r,s< b$, by applying \eqref{eqn:bound of integral}, we obtain that
\begin{align}
\mathcal{A}_{1}
&\leq C_{A_{1}}^{q}\int_{a}^{-8r_{1}}(s-r)^{\frac{q\gamma_{1}}{\gamma_{2}}-1}\mathbb{M}_{x}^{2r_{2}}\|f\|_{V}^{q}(r,x)
                \left((s-r)^{-\frac{\gamma_{1}}{\gamma_{2}}}
                    +(s-r)^{-\frac{\gamma_{1}+\gamma_{2}}{\gamma_{2}}}\right)^{q}dr \nonumber\\
&= C_{A_{1}}^{q}\int_{a}^{-8r_{1}}(s-r)^{-q-1}
            \mathbb{M}_{x}^{2r_{2}}\|f\|_{V}^{q}(r,x)
            \left((s-r)+1\right)^{q}dr \nonumber\\
&\leq  C_{A_{1}}^{q}\left((b-a)+1\right)^{q}
  \int_{a}^{-8r_{1}}(s-r)^{-q-1}\mathbb{M}_{x}^{2r_{2}}\|f\|_{V}^{q}(r,x)dr \nonumber\\
&\leq C_{A_{1}}^{q}\left((b-a)+1\right)^{q}
C_{q}r_{1}^{-q}\mathbb{M}_{t}^{8r_{1}}\mathbb{M}_{x}^{2r_{2}}\|f\|_{V}^{q}(t,x)\nonumber\\
&=: C_{\mathcal{A}_{1}}\left((b-a)+1\right)^{q}r_{1}^{-q}
        \mathbb{M}_{t}^{8r_{1}}\mathbb{M}_{x}^{2r_{2}}\|f\|_{V}^{q}(t,x),
\label{eqn:bound of mathcal A 21}
\end{align}
where $C_{q}$ is given in \eqref{eqn:bound of integral}.
Now we estimate $\mathcal{A}_{2}$. If $\rho\geq (s-r)^{\frac{1}{\gamma_{2}}}$,
then by \eqref{eqn: bound of dt nabla K}, there exists a constant $C_{1}>0$ such that
\begin{align}
\left|\nabla \frac{\partial}{\partial s}L_{\psi_{1}}(s)p_{\psi_{2}}(s,r,\rho w)\right|
&\leq C_{1}\left((\rho^{-(\gamma_{1}+1+d)}+\rho^{-(\gamma_{1}+\gamma_{2}+1+d)})\right.\nonumber\\
&\qquad\left.\wedge \left((s-r)^{-\frac{\gamma_{1}+1+d}{\gamma_{2}}}
                                    +(s-r)^{-\frac{\gamma_{1}+\gamma_{2}+1+d}{\gamma_{2}}}\right)\right)\nonumber\\
&\leq C_{1}\left(\rho^{-(\gamma_{1}+1+d)}+\rho^{-(\gamma_{1}+\gamma_{2}+1+d)}\right).
\label{eqn:bound nabla partial s L(s)p}
\end{align}
By applying \eqref{eqn:bound nabla partial s L(s)p} and \eqref{eqn: r2 less than or equal constant rho},
we obtain that
\begin{align*}
\left(A_{2}^{q}(r)\right)^{1/q}
&=\int_{(s-r)^{\frac{1}{\gamma_{2}}}}^{\infty}(2r_{2}+\rho)^{d} \left(\int_{\partial B_{1}}
        \left|\nabla \frac{\partial}{\partial s}L_{\psi_{1}}(s)p_{\psi_{2}}(s,r,\rho w)\right|^{q'}S_{1}(dw)\right)^{\frac{1}{q'}}d\rho\\
&\leq C_{1}(2\cdot 6^{-\frac{1}{\gamma_{2}}}+1)^{d}\\
&\qquad\times \int_{(s-r)^{\frac{1}{\gamma_{2}}}}^{\infty}\left(\rho^{-(\gamma_{1}+1)}
                        +\rho^{-(\gamma_{1}+\gamma_{2}+1)}\right)\left(\int_{\partial B_{1}}
                        S_{1}(dw)\right)^{\frac{1}{q'}}d\rho\\
&=C_{1}(2\cdot 6^{-\frac{1}{\gamma_{2}}}+1)^{d}[S_{1}(\partial B_{1})]^{\frac{1}{q'}}\\
    &\qquad \times \left(\frac{1}{\gamma_{1}}(s-r)^{-\frac{\gamma_{1}}{\gamma_{2}}}
                    +\frac{1}{\gamma_{1}+\gamma_{2}} (s-r)^{-\frac{\gamma_{1}+\gamma_{2}}{\gamma_{2}}}\right)\\
&\leq C_{A_{2}}\left((s-r)^{-\frac{\gamma_{1}}{\gamma_{2}}}
                +(s-r)^{-\frac{\gamma_{1}+\gamma_{2}}{\gamma_{2}}}\right),
\end{align*}
where
\begin{align*}
C_{A_{2}}=C_{1}(2\cdot 6^{-\frac{1}{\gamma_{2}}}+1)^{d}[S_{1}(\partial B_{1})]^{\frac{1}{q'}}\frac{1}{\gamma_{1}}.
\end{align*}
Therefore, by the estimations obtained in \eqref{eqn:bound of mathcal A 21}, we have
\begin{align}\label{eqn:bound of mathcal A 22}
\mathcal{A}_{2}
&\leq C_{A_{2}}^{q} \int_{a}^{-8r_{1}}(s-r)^{\frac{q\gamma_{1}}{\gamma_{2}}-1}
                    \mathbb{M}_{x}^{2r_{2}}\|f\|_{V}^{q}(r,x)
                    \left((s-r)^{-\frac{\gamma_{1}}{\gamma_{2}}}
                            +(s-r)^{-\frac{\gamma_{1}+\gamma_{2}}{\gamma_{2}}}\right)^q dr\nonumber\\
&\leq C_{A_{2}}^{q}\left((b-a)+1\right)^{q}
C_{q}r_{1}^{-q}\mathbb{M}_{t}^{8r_{1}}\mathbb{M}_{x}^{2r_{2}}\|f\|_{V}^{q}(t,x)\nonumber\\
&=: C_{\mathcal{A}_{2}}\left((b-a)+1\right)^{q}r_{1}^{-q}
        \mathbb{M}_{t}^{8r_{1}}\mathbb{M}_{x}^{2r_{2}}\|f\|_{V}^{q}(t,x).
\end{align}
Hence by combining \eqref{eqn:bound of partial-s-G-g},
\eqref{eqn:bound of mathcal A 21},
and \eqref{eqn:bound of mathcal A 22}, we obtain that
\begin{align*}%\label{eqn: esti of A2}
A
&\leq 2^{q-1}M (C_{\mathcal{A}_{1}}+C_{\mathcal{A}_{2}})\left((b-a)+1\right)^{q}
            r_{1}^{-q}\mathbb{M}_{t}^{8r_{1}}\mathbb{M}_{x}^{2r_{2}}\|f\|_{V}^{q}(t,x),
\end{align*}
which proves the inequality given in \eqref{eqn: esti time deri of mathcal G2}.
\end{proof}

Motivated by Lemma 2.11 (for $\psi_{1}(t,\xi)=-|\xi|^{\gamma/2}$ and $q=2$) in \cite{I. Kim K.-H. Kim 2016},
we have the following lemma which will be applied
for the proof of Lemma \ref{lem: bound of double average}.

\begin{lemma}\label{lem: bound of Dt mathcalG g}
Let $q\geq 2$
and let $r_{1}, r_{2}>0 $ satisfying $r_{1}^{1/\gamma_{2}}=r_{2}$.
Then it holds that
\begin{itemize}
  \item [\rm (i)] if  $a\in\mathbb{R}$, then there exists a constant $C_{1}>0$ depending on $d,q, \boldsymbol{\mu}, \boldsymbol{\gamma}$ and $\boldsymbol{\kappa}$ such that for any $(t,x)\in Q_{r_{1},r_{2}}$, $l\geq 0$ and
        $f\in C_{\rm c}^{\infty}(\mathbb{R}\times\mathbb{R}^{d};V)$ satisfying that $f(t,x)=0$ for $t\geq -8r_{1}$,
\begin{equation}\label{eqn: sup of derivative wrt s of mathcal G}
\sup_{(s,y)\in Q_{r_{1},r_{2}}}\left|\frac{\partial}{\partial s}\mathcal{G}_{a,q}f(l,s,y)\right|^{q}
\leq C_{1} r_{1}^{-q}
        \mathbb{M}_{t}^{8r_{1}}\mathbb{M}_{x}^{2r_{2}}\|f\|_{V}^{q}(t,x),
\end{equation}
  \item [\rm (ii)] if $-\infty< a\le -2r_{1}<0\le b< \infty$, then there exist constants $C_{2},C_{3}>0$ depending on $d,q, \boldsymbol{\mu}, \boldsymbol{\gamma}$ and $\boldsymbol{\kappa}$ such that for any $(t,x)\in Q_{r_{1},r_{2}}$ and
            $f\in C_{\rm c}^{\infty}(\mathbb{R}\times\mathbb{R}^{d};V)$ satisfying that $f(t,x)=0$ for $t\geq -8r_{1}$,
\begin{equation}\label{eqn: sup of derivative wrt s of mathcal G2}
\sup_{(s,y)\in Q_{r_{1},r_{2}}}\left|\frac{\partial}{\partial s}\widetilde{\mathcal{G}}_{a,q}f(s,y)\right|^{q}
\leq (C_{2}+C_{3}(b-a)^{q}) r_{1}^{-q}
        \mathbb{M}_{t}^{8r_{1}}\mathbb{M}_{x}^{2r_{2}}\|f\|_{V}^{q}(t,x).
\end{equation}
\end{itemize}
\end{lemma}

\begin{proof}
Since the proofs of \eqref{eqn: sup of derivative wrt s of mathcal G} and \eqref{eqn: sup of derivative wrt s of mathcal G2}
are similar, we only prove \eqref{eqn: sup of derivative wrt s of mathcal G2}.
Let $(t,x),(s,y)\in Q_{r_{1},r_{2}}$ and
$f\in C_{\rm c}^{\infty}(\mathbb{R}\times\mathbb{R}^{d};V)$ satisfying that $f(t,x)=0$ for $t\geq -8r_{1}$ be given.
If $a\geq -8r_{1}$, then by the assumption, $f(r,x)=0$ for all $r\geq a \geq -8r_{1}$, and so we have
\begin{align*}
\widetilde{\mathcal{G}}_{a,q}f(s,y)
=\left[\int_{a}^{s}(s-r)^{\frac{q\gamma_{1}}{\gamma_{2}}-1}
            \|L_{\psi_{1}}(s)p_{\psi_{2}}(s,r,\cdot)*f(r,\cdot)(y)\|_{V}^{q}dr\right]^{\frac{1}{q}}
=0.
\end{align*}
Therefore, there is nothing to prove because the LHS of \eqref{eqn: sup of derivative wrt s of mathcal G} is zero.
Therefore, we assume that $a< -8r_{1}$.
Since $s>-2r_{1}>-8r_{1}$ and $f(r,x)=0$ for $r\geq -8r_{1}$, we obtain that
\begin{align*}
\widetilde{\mathcal{G}}_{a,q}f(s,y)
&=\left[\int_{a}^{s}(s-r)^{\frac{q\gamma_{1}}{\gamma_{2}}-1}
        \|L_{\psi_{1}}(s)p_{\psi_{2}}(s,r,\cdot)*f(r,\cdot)(y)\|_{V}^{q}dr\right]^{\frac{1}{q}}\\
&=\left[\int_{a}^{-8r_{1}}(s-r)^{\frac{q\gamma_{1}}{\gamma_{2}}-1}
        \|L_{\psi_{1}}(s)p_{\psi_{2}}(s,r,\cdot)*f(r,\cdot)(y)\|_{V}^{q}dr\right]^{\frac{1}{q}}.
\end{align*}
By the fact that the derivative of a norm less than or equal the norm of the derivative,
and the triangle inequality, we obtain that
\begin{align}
&\left|\frac{\partial}{\partial s}\widetilde{\mathcal{G}}_{a,q}f(s,y)\right|\nonumber\\
&=\left|\frac{\partial}{\partial s}\left[\int_{a}^{-8r_{1}}  (s-r)^{\frac{q\gamma_{1}}{\gamma_{2}}-1}
            \|L_{\psi_{1}}(s)p_{\psi_{2}}(s,r,\cdot)*f(r,\cdot)(y)\|_{V}^{q}dr\right]^{\frac{1}{q}}\right|\nonumber\\
&=\left|\frac{\partial}{\partial s}\left[\int_{a}^{-8r_{1}}  \|(s-r)^{\frac{\gamma_{1}}{\gamma_{2}}-\frac{1}{q}}
            L_{\psi_{1}}(s)p_{\psi_{2}}(s,r,\cdot)*f(r,\cdot)(y)\|_{V}^{q}dr\right]^{\frac{1}{q}}\right|\nonumber\\
&\leq\left[\int_{a}^{-8r_{1}}\left\|\frac{\partial}{\partial s}  \left((s-r)^{\frac{\gamma_{1}}{\gamma_{2}}-\frac{1}{q}}
            L_{\psi_{1}}(s)p_{\psi_{2}}(s,r,\cdot)*f(r,\cdot)(y)\right)\right\|_{V}^{q}dr\right]^{\frac{1}{q}}\nonumber\\
&=\left[\int_{a}^{-8r_{1}}\left\|\left(\frac{\gamma_{1}}{\gamma_{2}}-\frac{1}{q}\right)
        (s-r)^{\frac{\gamma_{1}}{\gamma_{2}}-\frac{1}{q}-1}L_{\psi_{1}}(s)p_{\psi_{2}}(s,r,\cdot)*f(r,\cdot)(y)\right.\right.\nonumber\\
&\qquad\left.\left.+(s-r)^{\frac{\gamma_{1}}{\gamma_{2}}-\frac{1}{q}}\frac{\partial}{\partial s}
                    \left(L_{\psi_{1}}(s)p_{\psi_{2}}(s,r,\cdot)\right)*f(r,\cdot)(y)\right\|_{V}^{q}dr\right]^{\frac{1}{q}}\nonumber\\
&\leq \left[\int_{a}^{-8r_{1}}\left|\frac{\gamma_{1}}{\gamma_{2}}-\frac{1}{q}\right|^{q}(s-r)^{\frac{q\gamma_{1}}{\gamma_{2}}-1-q}
    \|L_{\psi_{1}}(s)p_{\psi_{2}}(s,r,\cdot)*f(r,\cdot)(y)\|_{V}^{q}dr\right]^{\frac{1}{q}}\nonumber\\
       &\qquad+\left[\int_{a}^{-8r_{1}}(s-r)^{\frac{q\gamma_{1}}{\gamma_{2}}-1}
     \left\|\frac{\partial}{\partial s}
     \left(L_{\psi_{1}}(s)p_{\psi_{2}}(s,r,\cdot)\right)*f(r,\cdot)(y)\right\|_{V}^{q}dr\right]^{\frac{1}{q}}\nonumber\\
&=:\left|\frac{\gamma_{1}}{\gamma_{2}}-\frac{1}{q}\right| A_{1}+A_{2},\label{eqn:bound of partial-G}
\end{align}
where
\begin{align*}
A_{1}&=\left[\int_{a}^{-8r_{1}}(s-r)^{\frac{q\gamma_{1}}{\gamma_{2}}-1-q}
                \|L_{\psi_{1}}(s)p_{\psi_{2}}(s,r,\cdot)*f(r,\cdot)(y)\|_{V}^{q}dr\right]^{\frac{1}{q}},\\
A_{2}&=\left[\int_{a}^{-8r_{1}}(s-r)^{\frac{q\gamma_{1}}{\gamma_{2}}-1}
                \left\|\frac{\partial}{\partial s}L_{\psi_{1}}(s)p_{\psi_{2}}(s,r,\cdot)*f(r,\cdot)(y)\right\|_{V}^{q}dr\right]^{\frac{1}{q}}.
\end{align*}
By applying Lemma \ref{lem: bound of Dt mathcalG g-1}, \eqref{eqn: esti time deri of mathcal G2}
and the Jensen's inequality, we obtain that
\begin{align*}
A_{1}^{q}&\leq C_{A_{1}}r_{1}^{-q}\mathbb{M}_{t}^{8r_{1}}\mathbb{M}_{x}^{2r_{2}}\|f\|_{V}^{q}(t,x),\\
A_{2}^{q}&\leq C_{A_{2}}\left((b-a)+1\right)^{q}
            r_{1}^{-q}\mathbb{M}_{t}^{8r_{1}}\mathbb{M}_{x}^{2r_{2}}\|f\|_{V}^{q}(t,x)\\
&\leq C_{A_{2}}2^{q-1}\left((b-a)^q+1\right)
            r_{1}^{-q}\mathbb{M}_{t}^{8r_{1}}\mathbb{M}_{x}^{2r_{2}}\|f\|_{V}^{q}(t,x)
\end{align*}
for some constant $C_{A_{1}},C_{A_{2}}>0$.
Therefore, by applying \eqref{eqn:bound of partial-G},
we obtain that
\begin{align}\label{eqn: esti of partial s G}
\left|\frac{\partial}{\partial s}\widetilde{\mathcal{G}}_{a,q}f(s,y)\right|^{q}
&\leq 2^{q-1}\left(\left|\frac{\gamma_{1}}{\gamma_{2}}-\frac{1}{q}\right| ^{q}A_{1}^{q}+A_{2}^{q}\right)\nonumber\\
&\leq (C_{2}+C_{3}(b-a)^{q}) r_{1}^{-q}\mathbb{M}_{t}^{8r_{1}}\mathbb{M}_{x}^{2r_{2}}\|f\|_{V}^{q}(t,x),
\end{align}
where
\begin{align*}
C_{2}=2^{q-1}\left(\left|\frac{\gamma_{1}}{\gamma_{2}}-\frac{1}{q}\right| ^{q}C_{A_{1}}+C_{A_{2}}2^{q-1}\right),
\quad
C_{3}=2^{2q-2}C_{A_{2}}.
\end{align*}
Hence by taking supremum for $(s,y)$ on left hand side of \eqref{eqn: esti of partial s G},
we have the inequality given in \eqref{eqn: sup of derivative wrt s of mathcal G2}.
\end{proof}

For each $R>0$, we put
\begin{align*}
Q_{R}=(-2R,0)\times B_{R^{1/\gamma_{2}}},
\quad
\hbox{\sout{$\displaystyle\int$}}_{Q_{R}}g(s,y)dsdy=\frac{1}{|Q_{R}|}\int_{Q_{R}}g(s,y)dsdy.
\end{align*}

Motivated by Lemma 2.12 (for $\psi_{1}(t,\xi)=-|\xi|^{\gamma/2}$ and $q=2$) in \cite{I. Kim K.-H. Kim 2016},
we have the following lemma which will be applied
for the proof of Theorem \ref{thm: sharp fct less than maximal}.

\begin{lemma}\label{lem: bound of double average}
Let $q\geq 2$ and $R>0$. Then it holds that
\begin{itemize}
  \item [\rm (i)] if $-\infty<a\le -2R<0\le b< \infty$, then there exist constants $C_{1},C_{2},C_{3}>0$ depending on $a,b ,q,\boldsymbol{\mu},\boldsymbol{\kappa}$ and $\boldsymbol{\gamma}$ such that
      for any $(t,x)\in Q_{R}$, $l\geq 0$ and $f\in C_{\rm c}^{\infty}(\mathbb{R}\times\mathbb{R}^{d};V)$,
\begin{align}
&\frac{1}{|Q_{R}|^{2}}\int_{Q_{R}}\int_{Q_{R}}
            |\mathcal{G}_{a,q}f(l,s,y)-\mathcal{G}_{a,q}f(l,r,z)|^{q}dsdydrdz\nonumber\\
&\qquad\leq C_{1}\mathbb{M}_{t}\mathbb{M}_{x}\|f\|_{V}^{q}(t,x)\label{eqn: double aver 1}
\end{align}
and
\begin{align}
&\frac{1}{|Q_{R}|^{2}}\int_{Q_{R}}\int_{Q_{R}}
            |\widetilde{\mathcal{G}}_{a,q}f(s,y)-\widetilde{\mathcal{G}}_{a,q}f(r,z)|^{q}dsdydrdz\nonumber\\
&\qquad\leq (C_{2}+C_{3}(b-a)^{q})\mathbb{M}_{t}\mathbb{M}_{x}\|f\|_{V}^{q}(t,x),\label{eqn: double aver 2}
\end{align}
  \item [\rm (ii)] if $q=2$ and $-\infty\leq a<b\leq \infty$,
  then there exists a constant $C_{4}>0$ depending on $\boldsymbol{\mu},\boldsymbol{\kappa}$ and $\boldsymbol{\gamma}$
such that for any $(t,x)\in Q_{R}$, $l \geq 0$ and $f\in C_{\rm c}^{\infty}(\mathbb{R}\times\mathbb{R}^{d};V)$,
\begin{align}
&\frac{1}{|Q_{R}|^{2}}\int_{Q_{R}}\int_{Q_{R}}
    |\mathcal{G}_{a,2}f(l,s,y)-\mathcal{G}_{a,2}f(l,r,z)|^{2}dsdydrdz\nonumber\\
&\qquad\leq C_{4}\mathbb{M}_{t}\mathbb{M}_{x}\|f\|_{V}^{2}(t,x).\label{eqn: double aver 3}
\end{align}
\end{itemize}
\end{lemma}

\begin{proof}
Since the proofs of \eqref{eqn: double aver 1}, \eqref{eqn: double aver 2}
and \eqref{eqn: double aver 3} are similar, we only prove \eqref{eqn: double aver 1}.
Let $(t,x)\in Q_{R}$, $l\geq 0$ and $f\in C_{\rm c}^{\infty}(\mathbb{R}\times\mathbb{R}^{d};V)$ be given.
We take a function $\zeta\in C_{\rm c}^{\infty}(\mathbb{R})$ such that $0\leq \zeta\leq 1$ and
\begin{align*}
\zeta=\left\{
        \begin{array}{ll}
          1 & \hbox{on } \quad [-8R,8R],\\
          0 & \hbox{on } \quad [-10R,10R]^{\rm c},
        \end{array}
      \right.
\end{align*}
where $A^{\rm c}$ is the complement of a set $A$. Put
\begin{align*}
f_{1}(s,y)=f(s,y)\zeta(s),\quad f_{2}(s,y)=f(s,y)(1-\zeta(s)).
\end{align*}
Then it is obvious that $f=f_1+f_2$, and by direct computation, we can see that
\begin{align}
\mathcal{G}_{a,q}f
&\leq\mathcal{G}_{a,q}f_{1}+\mathcal{G}_{a,q}f_{2}\label{eqn: 1st ineq},\\
\mathcal{G}_{a,q}f_{2}
&\leq \mathcal{G}_{a,q}f\label{eqn: 2nd ineq}.
\end{align}
In fact, since
\begin{align*}
 L_{\psi_{1}}(l)p_{\psi_{2}}(t,s,\cdot)*f_{2}(s,\cdot)(x)
 =(1-\zeta(s))L_{\psi_{1}}(l)p_{\psi_{2}}(t,s,\cdot)*f(s,\cdot)(x)
\end{align*}
and $|1-\zeta(s)|\leq 1$, we obtain that
\begin{align*}
\mathcal{G}_{a,q}f_{2}(l,t,x)
&=\left[\int_{a}^{t}(t-s)^{\frac{q\gamma_{1}}{\gamma_{2}}-1}
            \|L_{\psi_{1}}(l)p_{\psi_{2}}(t,s,\cdot)*f_{2}(s,\cdot)(x)\|_{V}^{q}ds\right]^{\frac{1}{q}}\\
&=\left[\int_{a}^{t}|1-\zeta(s)|^{q}(t-s)^{\frac{q\gamma_{1}}{\gamma_{2}}-1}
            \|L_{\psi_{1}}(l)p_{\psi_{2}}(t,s,\cdot)*f(s,\cdot)(x)\|_{V}^{q}ds\right]^{\frac{1}{q}}\\
&\leq \left[\int_{a}^{t}(t-s)^{\frac{q\gamma_{1}}{\gamma_{2}}-1}
            \|L_{\psi_{1}}(l)p_{\psi_{2}}(t,s,\cdot)*f(s,\cdot)(x)\|_{V}^{q}ds\right]^{\frac{1}{q}}\\
&=\mathcal{G}_{a,q}f(l,t,x),
\end{align*}
which proves \eqref{eqn: 2nd ineq}.
Then the proof of \eqref{eqn: 1st ineq} is straightforward by using $f=f_{1}+f_{2}$
and applying the Minkowski's inequality.
Also, we observe that for any constant $c$,
it holds that $|\mathcal{G}_{a,q}f-c|\leq |\mathcal{G}_{a,q}f_{1}|+|\mathcal{G}_{a,q}f_{2}-c|$
(see Lemma 2.12 in \cite{I. Kim K.-H. Kim 2016}).
Hence, for any constant $c$, we obtain that
\begin{align}
&\frac{1}{|Q_{R}|^{2}}\int_{Q_{R}}\int_{Q_{R}}
\left|\mathcal{G}_{a,q}f(l,s,y)-\mathcal{G}_{a,q}f(l,r,z)\right|^{q}dsdydrdz\nonumber\\
&=\frac{1}{|Q_{R}|^{2}}\int_{Q_{R}}\int_{Q_{R}}
  \left|\mathcal{G}_{a,q}f(l,s,y)-c+c-\mathcal{G}_{a,q}f(l,r,z)\right|^{q}dsdydrdz\nonumber\\
&\leq 2^{q-1}\frac{1}{|Q_{R}|^{2}}\left(\int_{Q_{R}}\int_{Q_{R}}|\mathcal{G}_{a,q}f(l,s,y)-c|^{q}dsdydrdz\right.\nonumber\\
&\qquad\qquad \left.+\int_{Q_{R}}\int_{Q_{R}}|c-\mathcal{G}_{a,q}f(l,r,z)|^{q}drdzdsdy\right)\nonumber\\
&=2^{q}\hbox{\sout{$\displaystyle\int$}}_{Q_{R}}|\mathcal{G}_{a,q}f(l,s,y)-c|^{q}dsdy\nonumber\\
&\leq 2^{2q-1}\hbox{\sout{$\displaystyle\int$}}_{Q_{R}}|\mathcal{G}_{a,q}f_{1}(l,s,y)|^{q}dsdy
+2^{2q-1}\hbox{\sout{$\displaystyle\int$}}_{Q_{R}}|\mathcal{G}_{a,q}f_{2}(l,s,y)-c|^{q}dsdy.\label{eqn:esti double aver}
\end{align}
Therefore, by taking $c=\mathcal{G}_{a,q}f_{2}(l,t,x)$ in \eqref{eqn:esti double aver}, we have
\begin{align}
&\frac{1}{|Q_{R}|^{2}}\int_{Q_{R}}\int_{Q_{R}}
       \left|\mathcal{G}_{a,q}f(l,s,y)-\mathcal{G}_{a,q}f(l,r,z)\right|^{q}dsdydrdz\nonumber\\
&\qquad \leq 2^{2q-1}\hbox{\sout{$\displaystyle\int$}}_{Q_{R}}|\mathcal{G}_{a,q}f_{1}(l,s,y)|^{q}dsdy
     +2^{2q-1}\hbox{\sout{$\displaystyle\int$}}_{Q_{R}}|\mathcal{G}_{a,q}f_{2}(l,s,y)
    -\mathcal{G}_{a,q}f_{2}(l,t,x)|^{q}dsdy. \label{eqn:decom of integral}
\end{align}
On the other hand, by the mean value theorem,
there exists a point $(s_{0},y_{0})$ on the line segment connecting $(s,y)$ and $(t,x)$
such that
\begin{align*}
\mathcal{G}_{a,q}f_{2}(l,s,y)-\mathcal{G}_{a,q}f_{2}(l,t,x)
&=\partial_{s}\mathcal{G}_{a,q}f_{2}(l,s_{0},y_{0})(s-t)+\nabla \mathcal{G}_{a,q}f_{2}(l,s_{0},y_{0})\cdot (y-x),
\end{align*}
where $\partial_{s}$ means the derivative with respect to the variable $s$.
Also, if $(s,y)\in Q_{R}$, then we have
\begin{align*}
|s-t|< 2R,\quad |y-x|<2R^{\frac{1}{\gamma_{2}}}.
\end{align*}
Therefore, by the Jensen's inequality, we obtain that
\begin{align*}
&|\mathcal{G}_{a,q}f_{2}(l,s,y)-\mathcal{G}_{a,q}f_{2}(l,t,x)|^{q}\\
&\qquad\leq 2^{q-1}\left|\partial_{s}\mathcal{G}_{a,q}f_{2}(l,s_{0},y_{0})\right|^{q}\left|s-t\right|^{q}
   +2^{q-1}|\nabla \mathcal{G}_{a,q}f_{2}(l,s_{0},y_{0})|^{q}| y-x|^{q}\\
&\qquad\leq 2^{2q-1}\left|R\partial_{s}\mathcal{G}_{a,q}f_{2}(l,s_{0},y_{0})\right|^{q}
   +2^{2q-1}|R^{\frac{1}{\gamma_{2}}}\nabla \mathcal{G}_{a,q}f_{2}(l,s_{0},y_{0})|^{q}\\
&\qquad\leq 2^{2q-1}
  \sup_{(s,y)\in Q_{R}}\left(\left|R\frac{\partial }{\partial s}\mathcal{G}_{a,q}f_{2}(l,s,y)\right|^{q}
    +|R^{\frac{1}{\gamma_{2}}}\nabla \mathcal{G}_{a,q}f_{2}(l,s,y)|^{q}\right),
\end{align*}
from which we have
\begin{align*}%\label{eqn: integral of Gg2(s,y)-Gg2(t,x) estimation}
&\hbox{\sout{$\displaystyle\int$}}_{Q_{R}}|\mathcal{G}_{a,q}f_{2}(l,s,y)
            -\mathcal{G}_{a,q}f_{2}(l,t,x)|^{q}dsdy\nonumber\\
&\leq 2^{2q-1}\sup_{(s,y)\in Q_{R}}
       \left(\left|R\frac{\partial }{\partial s}\mathcal{G}_{a,q}f_{2}(l,s,y)\right|^{q}
    +|R^{\frac{1}{\gamma_{2}}}\nabla \mathcal{G}_{a,q}f_{2}(l,s,y)|^{q}\right).
\end{align*}
Therefore, by \eqref{eqn:decom of integral}, we have
\begin{align}
&\frac{1}{|Q_{R}|^{2}}\int_{Q_{R}}\int_{Q_{R}}
       \left|\mathcal{G}_{a,q}f(l,s,y)-\mathcal{G}_{a,q}f(l,r,z)\right|^{q}dsdydrdz \nonumber\\
&\qquad \leq 2^{2q-1}\hbox{\sout{$\displaystyle\int$}}_{Q_{R}}|\mathcal{G}_{a,q}f_{1}(l,s,y)|^{q}dsdy \nonumber\\
&\qquad\qquad     +4^{2q-1}\sup_{(s,y)\in Q_{R}}
        \left(\left|R\frac{\partial }{\partial s}\mathcal{G}_{a,q}f_{2}(l,s,y)\right|^{q}
    +|R^{\frac{1}{\gamma_{2}}}\nabla \mathcal{G}_{a,q}f_{2}(l,s,y)|^{q}\right).
  \label{eqn:bound for LL}
\end{align}
We now show that
\begin{equation}\label{eqn: bound of mathcalG A1 on QR}
\int_{Q_{R}}|\mathcal{G}_{a,q}f_{1}(l,s,y)|^{q}dsdy
  \leq C_0|Q_{R}|\mathbb{M}_{t}\mathbb{M}_{x}\|f\|_{V}^{q}(t,x)
\end{equation}
for a constant $C_0$. We take $\eta\in C_{\rm c}^{\infty}(\mathbb{R}^{d})$ such that $0\leq \eta\leq 1$ and
\begin{align*}
\eta=\left\{
       \begin{array}{ll}
1\quad&\text{on}\quad B_{2R^{1/\gamma_{2}}},\\
0\quad&\text{on}\quad B_{3R^{1/\gamma_{2}}}^{\rm c},
       \end{array}
     \right.
\end{align*}
and we put $f_{11}=\eta f_{1}$ and $f_{12}=(1-\eta)f_{1}$.
Then it holds that $\mathcal{G}_{a,q}f_{1}\leq \mathcal{G}_{a,q}f_{11}+\mathcal{G}_{a,q}f_{12}$,
and by applying Lemma \ref{lem: g equal 0 outside of B3r1},
we obtain that
\begin{align}
\int_{Q_{R}}|\mathcal{G}_{a,q}f_{11}(l,s,y)|^{q}dsdy
&=\int_{-2R}^{0}\int_{B_{R^{1/\gamma_{2}}}}  |\mathcal{G}_{a,q}f_{11}(l,s,y)|^{q}dsdy \nonumber\\
&\leq C^{(1)}\int_{-2R}^{0}\int_{B_{3R^{1/\gamma_{2}}}}  \|f_{11}(s,y)\|_{V}^{q}dsdy \nonumber\\
&\leq C^{(1)}C^{(2)}|Q_{R}| \mathbb{M}_{t}\mathbb{M}_{x}\|f_{11}(t,x)\|_{V}^{q}\label{eqn:bound of g11}
\end{align}
for some constants $C^{(1)},C^{(2)}>0$.
By applying Lemma \ref{lem: bound of integral of mathcal G over Qr2r1}
with $(r_{1},r_{2})=(R,R^{1/\gamma_{2}})$, we obtain that
\begin{align}
\int_{Q_{R}}|\mathcal{G}_{a,q}f_{12}(l,s,y)|^{q}dsdy
&\leq C^{(3)}R^{\frac{d}{\gamma_{2}}+1}\mathbb{M}_{t}\mathbb{M}_{x}\|f\|_{V}^{q}(t,x) \nonumber\\
&= C^{(3)}C^{(4)}|Q_{R}|\mathbb{M}_{t}\mathbb{M}_{x}\|f\|_{V}^{q}(t,x)
    \label{eqn:bound of g12}
\end{align}
for some constants $C^{(3)},C^{(4)}>0$. Therefore, by combining \eqref{eqn:bound of g11} and \eqref{eqn:bound of g12},
we see that \eqref{eqn: bound of mathcalG A1 on QR} holds.
By applying Lemma \ref{lem: bound of sup of gradient of mathcal G} with $(r_{1},r_{2})=(R,R^{1/\gamma_{2}})$,
we have
\begin{align*}
\sup_{(s,y)\in Q_{R}}|\nabla \mathcal{G}_{a,q}f_{2}(l,s,y)|^{q}
\leq C^{(5)}R^{-\frac{q}{\gamma_{2}}}\mathbb{M}_{t}\mathbb{M}_{x}\|f_{2}\|_{V}^{q}(t,x)
\end{align*}
for some constant $C^{(5)}>0$, which implies that
\begin{align}
\sup_{(s,y)\in Q_{R}}|R^{\frac{1}{\gamma_{2}}}\nabla \mathcal{G}_{a,q}f_{2}(l,s,y)|^{q}
\leq C^{(5)}\mathbb{M}_{t}\mathbb{M}_{x}\|f_{2}\|_{V}^{q}(t,x).
\label{eqn: bound of sup of R gamma gradient on QR}
\end{align}
Similarly, by applying Lemma \ref{lem: bound of Dt mathcalG g} with $(r_{1},r_{2})=(R,R^{1/\gamma_{2}})$,
we have
\begin{align}
\sup_{(s,y)\in Q_{R}}\left|R\frac{\partial }{\partial s} \mathcal{G}_{a,q}f_{2}(l,s,y)\right|^{q}
&\leq C^{(6)}\mathbb{M}_{t}\mathbb{M}_{x}\|f_{2}\|_{V}^{q}(t,x)\label{eqn: bound of sup of R Dt G on QR}
\end{align}
for some constants $C^{(6)}>0$.
Finally, by applying \eqref{eqn: bound of mathcalG A1 on QR},
\eqref{eqn: bound of sup of R Dt G on QR} and \eqref{eqn: bound of sup of R gamma gradient on QR}
to \eqref{eqn:bound for LL},
we have the inequality given in \eqref{eqn: double aver 1}.
\end{proof}

\textbf{A proof of Theorem \ref{thm: sharp fct less than maximal}}
Since the proofs of \eqref{eqn: sharp fct esti}, \eqref{eqn: sharp fct esti 2} and \eqref{eqn: sharp fct esti 3} are similar,
we only prove \eqref{eqn: sharp fct esti}.
Let $f\in C_{\rm c}^{\infty}((a,b)\times\mathbb{R}^{d};V)$, $l\geq 0$ and $(t,x)\in (a,b)\times \mathbb{R}^{d}$ be given.
 Recall that
\begin{align*}
(\mathcal{G}_{a,q}f(l,\cdot,\cdot))^{\#}(t,x)
&=\sup_{Q}\frac{1}{|Q|}\int_{Q}|\mathcal{G}_{a,q}f(l,r,z)-(\mathcal{G}_{a,q}f(l,\cdot,\cdot))_{Q}|drdz
\end{align*}
where $(\mathcal{G}_{a,q}f(l,\cdot,\cdot))_{Q}=\frac{1}{|Q|}\int_{Q}\mathcal{G}_{a,q}f(l,r,z)drdz$ and
the supremum is taken all $Q$ containing $(t,x)$ of the type
\begin{align*}
Q=(t-R,t+R)\times B_{R^{1/\gamma_{2}}}(x),\quad R>0.
\end{align*}
Note that by applying Jensen's inequality and the definition of $(\mathcal{G}_{a,q}f(l,\cdot,\cdot))_{Q}$,
we obtain that
\begin{align}
&\left(\frac{1}{|Q|}\int_{Q}|\mathcal{G}_{a,q}f(l,s,y)-(\mathcal{G}_{a,q}f(l,\cdot,\cdot))_{Q}|dsdy\right)^{q}\nonumber\\
&\leq \frac{1}{|Q|}\int_{Q}|\mathcal{G}_{a,q}f(l,s,y)-(\mathcal{G}_{a,q}f(l,\cdot,\cdot))_{Q}|^{q}dsdy\nonumber\\
&\leq \frac{1}{|Q|^{2}}\int_{Q}\int_{Q}
            |\mathcal{G}_{a,q}f(l,s,y)-\mathcal{G}_{a,q}f(l,r,z)|^{q}dsdydrdz.\label{eqn: bound of sharp mathcalG 2}
\end{align}
Also we note that for any $c_{1}\in\mathbb{R}$, $c_{2}\in\mathbb{R}^{d}$
and for $K_{\psi_{1},\psi_{2}}(l,t,s,x)=L_{\psi_{1}}(l)p_{\psi_{2}}(t,s,x)$,
if we put
\begin{align*}
\bar{f}(t,x)=f(t-c_{1},x-c_{2}),\quad
\bar{K}_{\psi_{1},\psi_{2}}(l,t,s,x)=K_{\psi_{1},\psi_{2}}(l,t-c_{1},s-c_{1},x),
\end{align*}
then by applying the change of variables, it holds that
\begin{align*}
&\mathcal{G}_{a,q}f(l,t-c_{1},x-c_{2})\\
&=\left(\int_{a}^{t-c_{1}}((t-c_{1})-s)^{\frac{q\gamma_{1}}{\gamma_{2}}-1}
        \|K(t-c_{1},s,\cdot)*f(s,\cdot)(x-c_{2})\|_{V}^{q}ds\right)^{\frac{1}{q}}\\
&=\left(\int_{a+c_{1}}^{t}(t-s)^{\frac{q\gamma_{1}}{\gamma_{2}}-1}
        \|\bar{K}_{\psi_{1},\psi_{2}}(l,t,s,\cdot)*\bar{f}(s,\cdot)(x)\|_{V}^{q}ds\right)^{\frac{1}{q}}\\
&=\mathcal{G}_{a+c_{1},q}\bar{f}(l,t,x).
\end{align*}
Therefore, we may assume that $Q=Q_{R}=(-2R,0)\times B_{R^{1/\gamma_{2}}}$.
By applying \eqref{eqn: double aver 1}, we have
\begin{align}\label{eqn:esti double average 2}
&\frac{1}{|Q_{R}|^{2}}\int_{Q_{R}}\int_{Q_{R}}
                |\mathcal{G}_{a,q}f(l,s,y)-\mathcal{G}_{a,q}f(l,r,z)|^{q}dsdydrdz\nonumber\\
&\qquad\leq C_{1}\mathbb{M}_{t}\mathbb{M}_{x}\|f\|_{V}^{q}(t,x)
\end{align}
for some constant $C_{1}>0$.
By combining \eqref{eqn: bound of sharp mathcalG 2} and \eqref{eqn:esti double average 2}, we have
\begin{align*}
\frac{1}{|Q|}\int_{Q}|\mathcal{G}_{a,q}f(l,s,y)-(\mathcal{G}_{a,q}f(l,\cdot,\cdot))_{Q}|dsdy
\leq  C_{1}^{\frac{1}{q}}(\mathbb{M}_{t}\mathbb{M}_{x}\|f\|_{V}^{q})^{\frac{1}{q}}(t,x),
\end{align*}
and then by taking the supremum in $Q$, we have the inequality given in \eqref{eqn: sharp fct esti}. \hfill $\blacksquare$

%%%%%%%%%%%%%%%%%%%%%%%%%%%%%%%%%%%%%%%%%%%%%%%%%%%%%%%%%%%%%%%%%%
\section{A Proof of Theorem \ref{thm:FS}}\label{sec:Appendix C}

Let $(X,\mathcal{F},\mu)$ be a complete measure space such that $\mu(X)=\infty$ and $\mu$ is $\sigma$-finite,
and let $\mathcal{F}^{0}$ be the subset of $\mathcal{F}$ consisting of $A$ with $\mu(A)<\infty$.
For each $A\in\mathcal{F}^{0}$ and an integrable function $f$ on $A$, we put
\begin{align*}
f_{A}:=\frac{1}{\mu(A)}\int_{A}f(x)d\mu(x),
\end{align*}
where $f_{A}$ is considered as $0$ if $\mu(A)=0$.

A countable collection $\mathcal{P}\subset \mathcal{F}^{0}$ is called a partition of $X$
if $\mathcal{P}$ is a mutually disjoint family and $\bigcup_{A\in\mathcal{P}}A=X$.
Let $\{\mathcal{P}_{n}\,|\, n\in\mathbb{Z}\}$ be a sequence of partitions of $X$.
Then for each $x\in X$ and $n\in\mathbb{Z}$, there exists a unique $P\in\mathcal{P}_{n}$, denoted by $P:=P_n(x)$ such that $x\in P$.

The following definition is taken from \cite{Krylov 2008}.

\begin{definition}\label{def:filtration}
\upshape
 A sequence of partitions $\{\mathcal{P}_{n}\,|\, n\in\mathbb{Z}\}$ is called a filtration of partitions
(with respect to $\mathbb{L}$, where $\mathbb{L}$ is a fixed dense subset of $L^{1}(X)$)
if it satisfies the following conditions:
\begin{itemize}
  \item [\rm (i)] $\inf_{P\in \mathcal{P}_{n}}\mu(P)\rightarrow \infty$ as $n\rightarrow -\infty$ and
  \begin{align*}
  \lim_{n\rightarrow\infty}\frac{1}{\mu(P_{n}(x))}\int_{P_{n}(x)}f(y)d\mu(y)=f(x)\quad \text{ a.e. $x$}
  \end{align*}
  for all $f\in\mathbb{L}$,

  \item [\rm (ii)] for each $n\in\mathbb{Z}$ and $P\in\mathcal{P}_{n}$, there is a unique $P^{\prime}\in \mathcal{P}_{n-1}$
  such that $P\subset P^{\prime}$ and
  \begin{align*}
  \mu(P^{\prime})\leq N_{0}\mu(P),
  \end{align*}
  where $N_{0}$ is a constant independent of $n, P$ and $P^{\prime}$.
\end{itemize}
\end{definition}
For a filtration of partitions $\mathbf{P}=\{\mathcal{P}_{n}\,|\,n\in\mathbb{Z}\}$
and an integrable function $f$ on all $A\in \mathcal{F}^{0}$,
we define the sharp function $f^{\#,\mathbf{P}}$ (related to $\mathbf{P}$) by
\begin{align*}
f^{\#,\mathbf{P}}(x)=\sup_{n\in\mathbb{Z}}\frac{1}{\mu(P_{n}(x))}\int_{P_{n}(x)}|f(y)-f_{P_{n}(x)}|d\mu(y).
\end{align*}
The following theorem is the Fefferman-Stein theorem for the sharp function $f^{\#,\mathbf{P}}$.

\begin{theorem}[\cite{Krylov 2008}, Theorem 3.2.10]\label{thm: FS filtration}
Let $1<p<\infty$. Then for any $f\in L^{p}(X)$,
\begin{align*}
\|f\|_{L^{p}(X)}\leq C\|f^{\#,\mathbf{P}}\|_{L^{p}(X)},
\end{align*}
where the constant $C$ is given by $C=\left(\frac{2p}{p-1}\right)^{p}N_{0}^{p-1}$
with the constant $N_{0}$ appeared in (ii) of Definition \ref{def:filtration}.
\end{theorem}

From now on, we consider the case $X=\mathbb{R}^{d+1}$ and $\mu$ is the Lebesgue measure.
Let $\gamma>0$ be given. For each $n\in\mathbb{Z}$, we define
\begin{align}
\mathcal{P}_{n}=\left\{[i_{0}2^{-n},(i_{0}+1)2^{-n})\times \prod_{k=1}^{d}[i_{k}2^{-\frac{n}{\gamma}},(i_{k}+1)2^{-\frac{n}{\gamma}})\,|\,i_{0},i_{1},\cdots,i_{d}\in\mathbb{Z}\right\}.
\label{eqn: partition cube}
\end{align}
Then it is obvious that for each $n\in\mathbb{Z}$, $\mathcal{P}_{n}$ is a partition of $\mathbb{R}^{d+1}$.

\begin{lemma}
For each $n\in\mathbb{Z}$, let $\mathcal{P}_{n}$ be a partition given by \eqref{eqn: partition cube}.
Then the family $\mathbf{P}=\{\mathcal{P}_{n}\,|\,n\in\mathbb{Z}\}$ is a filtration of partitions
with respect to $\mathbb{L}$ (the dense subset of $L^{1}(\mathbb{R}^{d+1})$ consisting of all
infinitely differentiable functions with compact support).
\end{lemma}

\begin{proof}
We check the conditions (i) and (ii) in Definition \ref{def:filtration}.

(i)\enspace Let $n\in\mathbb{Z}$ and let $P\in\mathcal{P}_{n}$ given by
\begin{align*}
P=[i_{0}2^{-n},(i_{0}+1)2^{-n})\times \prod_{k=1}^{d}[i_{k}2^{-\frac{n}{\gamma}},(i_{k}+1)2^{-\frac{n}{\gamma}})
\end{align*}
for some $i_{0},i_{1},\cdots,i_{d}\in\mathbb{Z}$. Then we have
\begin{align*}
|P|=2^{-n(1+\frac{d}{\gamma})}.
\end{align*}
Since $P\in\mathcal{P}_{n}$ is arbitrary, we have
\begin{align*}
\inf_{P\in\mathcal{P}_{n}}|P|=2^{-n(1+\frac{d}{\gamma})}\rightarrow \infty
\quad\text{as}\quad n\rightarrow -\infty.
\end{align*}
On the other hand, by Lebesgue differentiation theorem (see, e.g., Theorem 3.21 in \cite{Folland 1999}), we see that for any $f\in\mathbb{L}$,
\begin{align*}
\lim_{n\rightarrow\infty}\frac{1}{|P_{n}(t,x)|}\int_{P_{n}(t,x)}f(s,y)dsdy=f(t,x)\quad \text{a.e.,}
\end{align*}
where, for each $n\in\mathbb{Z}$ and $(t,x)\in\mathbb{R}^{d+1}$,
$P_{n}(t,x)$ is the element of $\mathcal{P}_{n}$ containing $(t,x)$.

(ii)\enspace Let $n\in\mathbb{Z}$ and let $P\in\mathcal{P}_{n}$ given by
\begin{align*}
P=[i_{0}2^{-n},(i_{0}+1)2^{-n})\times \prod_{k=1}^{d}[i_{k}2^{-\frac{n}{\gamma}},(i_{k}+1)2^{-\frac{n}{\gamma}})
\end{align*}
for some $i_{0},i_{1},\cdots,i_{d}\in\mathbb{Z}$. If we put
\begin{align*}
P^{\prime}=[i_{0}2^{-(n-1)},(i_{0}+1)2^{-(n-1)})\times \prod_{k=1}^{d}[i_{k}2^{-\frac{n-1}{\gamma}},(i_{k}+1)2^{-\frac{n-1}{\gamma}}),
\end{align*}
then it holds that $P\subset P^{\prime}$ and
\begin{align*}
\frac{|P^{\prime}|}{|P|}
=\frac{2^{-(n-1)(1+\frac{d}{\gamma})}}{2^{-n(1+\frac{d}{\gamma})}}
=2^{1+\frac{d}{\gamma}}.
\end{align*}
Thus, by taking $N_{0}$ such that $N_{0}\geq 2^{1+\frac{d}{\gamma}}$, we have
\begin{align*}
|P^{\prime}|\leq N_{0} |P|.
\end{align*}
The proof is complete.
\end{proof}

Let $\gamma>0$ be given. For each $(t,x)\in \mathbb{R}^{d+1}$ and $R>0$,
we define
\begin{align*}
Q_{R}(t,x)=(t-R,t+R)\times B_{R^{1/\gamma}}(x),
\end{align*}
where $B_{R^{1/\gamma}}(x)=\{y\in\mathbb{R}^{d}\,|\, |x-y|<R^{1/\gamma}\}$.
Put
\begin{align*}
\mathbf{Q}=\{Q_{R}(t,x)\,|\, (t,x)\in\mathbb{R}^{d+1}, R>0\}.
\end{align*}
For a locally integrable function $f$ on $\mathbb{R}^{d+1}$,
we define the sharp function $f^{\#,\mathbf{Q}}$ by
\begin{align*}
f^{\#,\mathbf{Q}}(t,x)=\sup_{R>0}\int_{Q_{R}(t,x)}|f(s,y)-f_{Q_{R}(t,x)}|dsdy,
\quad (t,x)\in \mathbb{R}^{d+1},
\end{align*}
 where $f_{Q_{R}(t,x)}=\frac{1}{|Q_{R}(t,x)|}\int_{Q_{R}(t,x)}f(s,y)dsdy$.

\bigskip
\textbf{A proof of Theorem \ref{thm:FS}.}\enspace
Let $\mathbf{P}=\{\mathcal{P}_{n}\,|\,n\in\mathbb{Z}\}$,
 where $\mathcal{P}_{n}$ is a partition of $\mathbb{R}^{d+1}$ given in \eqref{eqn: partition cube}.
 For each $(t,x)\in\mathbb{R}^{d+1}$ and $n\in\mathbb{Z}$,
there exists a unique $P_{n}(t,x)\in\mathcal{P}_{n}$
 such that $(t,x)\in P_{n}(t,x)$ and
 \begin{align*}
P_{n}(t,x)=[i_{0}2^{-n},(i_{0}+1)2^{-n})\times \prod_{k=1}^{d}[i_{k}2^{-\frac{n}{\gamma}},(i_{k}+1)2^{-\frac{n}{\gamma}})
\end{align*}
 for some $i_{0},i_{1},\cdots,i_{d}\in\mathbb{Z}$.
  Then, by taking $R_{0}=2^{-n}d^{\gamma}$, we have
 \begin{align}
  P_{n}(t,x)\subset (t-2^{-n}d^{\gamma},t+2^{-n}d^{\gamma})\times B_{2^{-(n/\gamma)}d}(x)=Q_{R_{0}}(t,x)
 \label{eqn:P subset QR0}
 \end{align}
 and
 \begin{align}
 \frac{|Q_{R_{0}}(t,x)|}{|P_{n}(t,x)|}=\frac{2d^{\gamma+d}\pi^{d/2}}{\Gamma(d/2)}=:N_{1}.
 \label{eqn:QR0/P}
 \end{align}
 Then, by applying the definition of $f_{P_{n}(t,x)}$, \eqref{eqn:P subset QR0},
 \eqref{eqn:QR0/P} and the triangle inequality, we obtain that
 \begin{align}
 &\frac{1}{|P_{n}(t,x)|}\int_{P_{n}(t,x)}|f(s,y)-f_{P_{n}(t,x)}|dsdy\nonumber\\
 &\leq \frac{1}{|P_{n}(t,x)|^{2}}
    \int_{P_{n}(t,x)}\int_{P_{n}(t,x)}|f(s,y)-f(r,u)|drdudsdy \nonumber\\
 &\leq N_{1}^{2}\frac{1}{|Q_{R_{0}}(t,x)|^{2}}
    \int_{Q_{R_{0}}(t,x)}\int_{Q_{R_{0}}(t,x)}|f(s,y)-f(r,u)|drdudsdy\nonumber\\
 &=N_{1}^{2}\frac{1}{|Q_{R_{0}}(t,x)|^{2}}
 \int_{Q_{R_{0}}(t,x)}\int_{Q_{R_{0}}(t,x)}|f(s,y)-f_{Q_{R_{0}}(t,x)}+f_{Q_{R_{0}}(t,x)}-f(r,u)|drdudsdy\nonumber\\
 &\leq 2N_{1}^{2}\frac{1}{|Q_{R_{0}}(t,x)|}\int_{Q_{R_{0}}(t,x)}|f(s,y)-f_{Q_{R_{0}}(t,x)}|dsdy\nonumber\\
 &\leq 2N_{1}^{2}f^{\#,\mathbf{Q}}(t,x).\label{eqn:sharp esti}
 \end{align}
 Therefore, by taking supremum for $n\in\mathbb{Z}$ on the left hand side of \eqref{eqn:sharp esti},
 we have for any $(t,x)\in\mathbb{R}^{d+1}$,
 \begin{align*}
 f^{\#,\mathbf{P}}(t,x)\leq 2N_{1}^{2}f^{\#,\mathbf{Q}}(t,x).
 \end{align*}
 Hence, by applying Theorem \ref{thm: FS filtration}, we obtain that
 \begin{align*}
 \|f\|_{L^{p}(\mathbb{R}^{d+1})}^{p}
 &\leq C_{1}^{p}\left\|f^{\#,\mathbf{P}}\right\|_{L^{p}(\mathbb{R}^{d+1})}^{p}
 =C_{1}^{p}\int_{\mathbb{R}^{d+1}}\left|f^{\#,\mathbf{P}}(t,x)\right|^{p}dtdx\\
 &\leq C_{1}^{p}(2N_{1}^{2})^{p}\int_{\mathbb{R}^{d+1}}\left|f^{\#,\mathbf{Q}}(t,x)\right|^{p}dtdx\\
 &\leq C_{1}^{p}(2N_{1}^{2})^{p}\left\|f^{\#,\mathbf{Q}}\right\|_{L^{p}(\mathbb{R}^{d+1})}^{p},
 \end{align*}
 and then by taking $C=C_{1}(2N_{1}^{2})$, the proof is complete.

%%%%%%%%%%%%%%%%%%%%%%%%%%%%%%%%%%%%%%%%%%
\section*{Acknowledgment}
This paper was supported by a Basic Science Research Program through the NRF funded by the MEST (NRF-2022R1F1A1067601)
and the MSIT (Ministry of Science and ICT), Korea, under the ITRC(Information Technology Research Center) support program (IITP-RS-2024-00437284) supervised by the IITP(Institute for Information \& Communications Technology Planning \& Evaluation).

\section*{Data Availability }
Data sharing is not applicable to this article as no datasets were generated or analysed during the current study.

\section*{Conflict of Interest}
The authors declare that they have no conflict of interest regarding this work.
%%%%%%%%%%%%%%%%%%%%%%%%%%%%%%%%%%%%%%%%%%%%%%%%%%%%%


\begin{thebibliography}{00}

\bibitem{Bergh 1976}
J. Bergh and J. L\"{o}fstr\"{o}m,
{``Interpolation Spaces. An Introduction,''}
Grundlehren der Mathematischen Wissenschaften No. \textbf{223}
Springer-Verlag, Berlin-New York, 1976.

\bibitem{Folland 1999}
 G. B. Folland,
{``Real Analysis. Modern Techniques and Their Applications,''}
Second Edition, Pure Appl. Math., Wiley-Intersci. Publ., John Wiley \& Sons, Inc., New York, 1999.

\bibitem{Grafakos 2014}
L. Grafakos,
{``Modern Fourier Analysis,''}
Grad. Texts in Math. \textbf{250}, Springer, New York, 2009.


\bibitem{Ji-Kim 2025}
 U. C. Ji and J. H. Kim,
\textit{Littlewood-Paley type inequality for evolution
systems associated with pseudo-differential operators},
 Preprint, 2025.

\bibitem{I. Kim 2012}
 I. Kim and K.-H. Kim,
\textit{A generalization of the Littlewood-Paley inequality for the fractional Laplacian $(-\Delta)^{\alpha/2}$},
 J. Math. Anal. Appl. \textbf{388} (2012), no. 1, 175–190.

\bibitem{I. Kim 2016}
 I. Kim and K.-H. Kim,
\textit{An $L_{p}$-theory for stochastic partial differential equations driven by Lévy processes with pseudo-differential operators of arbitrary order},
 Stochastic Process. Appl. \textbf{126} (2016) no. 9, 2761–2786.

\bibitem{I. Kim K.-H. Kim 2016}
 I. Kim, K.-H. Kim and S. Lim,
\textit{Parabolic Littlewood-Paley inequality for a class of time-dependent pseudo-differential operators of arbitrary order, and applications to high-order stochastic PDE},
J. Math. Anal. Appl. \textbf{436} (2016) no. 2, 1023–1047.

\bibitem{Krylov 1994}
N. V. Krylov,
\textit{A generalization of the Littlewood-Paley inequality and some other results related to stochastic partial differential equations},
Ulam Quart. \textbf{2} (1994) no. 4, 16--26.

\bibitem{Krylov 1999}
N. V. Krylov,
\textit{An analytic approach to SPDEs}, in
``Stochastic Partial Differential Equations: Six Perspectives,'' pp. 185--242,
 Math. Surveys Monogr. \textbf{64}, Amer. Math. Soc., Providence, RI, 1999.

\bibitem{Krylov 2008}
 N. V. Krylov,
{``Lectures on Elliptic and Parabolic Equations in Sobolev Spaces,''}
Grad. Stud. Math. \textbf{96}, Amer. Math. Soc., Providence, RI, 2008.

\bibitem{Miao 2008}
C. Miao, B. Yuan and B. Zhang,
\textit{Well-posedness of the Cauchy problem for the fractional power dissipative equations},
Nonlinear Anal. \textbf{68} (2008), no. 3, 461–484.

\bibitem{Stein 1970}
E. M. Stein,
{``Singular Integrals and Differentiability Properties of Functions,''}
Princeton Math. Ser. \textbf{30}, Princeton University Press, Princeton, NJ, 1970.

\bibitem{Stein 1993}
E. M. Stein,
{``Harmonic Analysis: Real-Variable Methods, Orthogonality, and Oscillatory Integrals,"}
Princeton Math. Ser. \textbf{43},
Princeton University Press, Princeton, NJ, 1993.

\end{thebibliography}
\end{document}